\documentclass[reqno,11pt]{amsart}

\usepackage{latexsym,geometry} 
\usepackage{amsmath,amsthm,amsopn,amstext,amscd,amsfonts,amssymb}

\usepackage{eqnarray,amsrefs,color}

\usepackage{latexsym,enumerate,graphicx}

\usepackage{fullpage,yhmath}
\usepackage{hyperref}
\usepackage{eucal}
\usepackage{color}
\usepackage{verbatim}
\usepackage[normalem]{ulem}

\usepackage[colorinlistoftodos]{todonotes}

\newtheorem{theorem}{Theorem}[section]
\newtheorem{lemma}[theorem]{Lemma}
\newtheorem{prop}[theorem]{Proposition}

\newtheorem{defi}[theorem]{Definition}
\newtheorem{rem}[theorem]{Remark\/}
\newtheorem{ex}{Example\/}

\newcommand{\R}{\mathbb{R}}
\newcommand{\dsum}{\displaystyle\sum}

\numberwithin{equation}{section}

\begin{document}

\title[Singular anisotropic elliptic equations]{Singular anisotropic elliptic equations with gradient-dependent lower order terms}

\author{Barbara Brandolini}
\address[Barbara Brandolini]{Dipartimento di Matematica e Informatica, Universit\`a degli Studi di Palermo, via Archirafi 34, 90123 Palermo, Italy}
\email{barbara.brandolini@unipa.it}

\author{Florica C. C\^irstea}
\address[Florica C. C\^irstea]{School of Mathematics and Statistics, The University of Sydney,
	NSW 2006, Australia}
\email[Corresponding author]{florica.cirstea@sydney.edu.au}



\begin{abstract}
We  prove 
the existence of weak solutions for a general class of Dirichlet anisotropic elliptic problems 
of the form $$\mathcal A u+\Phi(x,u,\nabla u)=\Psi(u,\nabla u)+\mathfrak Bu +f $$ on a bounded open subset $\Omega\subset \mathbb R^N$ $(N\geq 2)$, where
$f\in L^1(\Omega)$ is arbitrary. 
Our models are
$ \mathcal Au=-\sum_{j=1}^N \partial_j (|\partial_j u|^{p_j-2}\partial_j u)$ and
$\Phi(u,\nabla u)=\left(1+\sum_{j=1}^N \mathfrak{a}_j |\partial_j u|^{p_j}\right)|u|^{m-2}u$,  
with
$m,p_j>1$, $\mathfrak{a}_j\geq 0$ for $1\leq j\leq N$ and $\sum_{k=1}^N (1/p_k)>1$. The main novelty is the inclusion of a possibly singular gradient-dependent term $\Psi(u,\nabla u)=\sum_{j=1}^N |u|^{\theta_j-2}u\, |\partial_j u|^{q_j}$, where $\theta_j>0$ and $0\leq q_j<p_j$ for $1\leq j\leq N$. 
Under suitable conditions, we prove the existence of solutions 
by distinguishing two cases:
1) for every $1\leq j\leq N$, we have $\theta_j> 1$
and   
2) there exists $1\leq j\leq N$ such that $\theta_j\leq 1$. 
In the latter situation, assuming that $f \ge 0$ a.e. in $\Omega$, we obtain non-negative solutions for our problem.
\end{abstract}

\maketitle


\section{Introduction and main results}\label{Sec1}

This paper is a continuation of our study initiated in \cite{baci} to obtain existence of solutions for general anisotropic elliptic equations in a bounded, open subset $\Omega\subset  \mathbb R^N$ ($N\geq 2$), subject to a homogeneous Dirichlet boundary condition, $u=0$ on $\partial \Omega$. 
We impose no smoothness assumptions on the boundary of $\Omega$. The equations under consideration feature a low summability data $f\in L^1(\Omega)$, 
a lower-order term $\Phi(x,u,\nabla u)$ satisfying a ``good sign" condition, an ``anisotropic natural growth" in the gradient and no upper bound restriction in $|u|$ (see \eqref{cond1} and \eqref{etc}). The novelty of our work here, compared with \cite{baci}, consists in the introduction of a possibly singular gradient-dependent term $\Psi(u,\nabla u)$ (as in \eqref{ass}) which cannot be incorporated in $\Phi$. 
The main contribution in this paper is to show that, under suitable assumptions, our problem \eqref{eq1} admits solutions $u$ in the anisotropic Sobolev space 
$W_0^{1,\overrightarrow{p}}(\Omega)$  such that $\Phi(x,u,\nabla u)\in L^1(\Omega)$. This answers a question we raised in \cite{baci}.

Let $W_0^{1,\overrightarrow{p}}(\Omega)$ be the closure of   
$C_c^\infty (\Omega)$ (the set of smooth functions with compact support in $\Omega$)   
with respect to the norm 
$ \|u\|_{W_0^{1,\overrightarrow{p}}(\Omega)}=\sum_{j=1}^N \|\partial_j u\|_{L^{p_j}(\Omega)}$, where 
we  
assume that 
\begin{equation} \label{IntroEq0}
1<p_j\leq p_{j+1} <\infty \quad \mbox{for every }1\leq j\leq N-1 \quad  \text{and}\quad 
p<N.
\end{equation}
Here, $p$ is the harmonic mean of $p_1,\ldots, p_N$, that is, $p:=N/\sum_{j=1}^N (1/p_j)$. We write $\partial_j u$ for the partial derivative $\partial u/\partial x_j$. 
We use $W^{-1,\overrightarrow{p}'}(\Omega)$ for the dual of $W_0^{1,\overrightarrow{p}}(\Omega)$ and $\langle \cdot, \cdot \rangle$ for 
the duality between $W^{-1,\overrightarrow{p}'}(\Omega)$ and 
$W_0^{1,\overrightarrow{p}}(\Omega)$. Since $p<N$, the embedding $W_0^{1,\overrightarrow{p}}(\Omega)\hookrightarrow L^s(\Omega)$ is continuous 
for every
$s\in [1,p^\ast]$ and compact for every $s\in [1,p^\ast)$, 
where $p^\ast:=Np/(N-p)$ stands for the anisotropic Sobolev exponent (see Remark~\ref{an-sob} in the Appendix).  

Before introducing our general problem in Section~\ref{genpr} and the main results associated with it (Theorems~\ref{mainth} and \ref{teo2}), we present a model. 
For every $(t,\xi)\in \mathbb R\times \mathbb R^N$, we define 
\begin{eqnarray}
 \Phi_0(t,\xi)&=&\left(\mathfrak a_0+ \dsum_{j=1}^N \mathfrak{a}_j |\xi_j|^{p_j}\right)|t|^{m-2} t  \mbox{ with } 
m>1,\  \mathfrak a_0>0, \  \mathfrak a_j\geq  0\ \mbox{for } 1\leq j\leq N,\label{phiz}\\
 \Psi(t,\xi)&=&\dsum_{j=1}^N  |t|^{\theta_j-2}t\, |\xi_j|^{q_j}\ 
\mbox{ with } \theta_j>0\  \mbox{and}\ 
0\leq q_j<p_j \  \mbox{for all } 1\leq j\leq N.  \label{ass}
\end{eqnarray}
Let $h\in W^{-1,\overrightarrow{p}'}(\Omega)$ and $f\in L^1(\Omega)$ be arbitrary. 
The model for our problem is as follows:
\begin{equation} \label{yvc}
\left\{  \begin{aligned}
&-\sum_{j=1}^N  \partial_j (|\partial_j u|^{p_j-2}\partial_j u)+  \Phi_0(u,\nabla u)
=\Psi(u,\nabla u)+h+f\quad \mbox{in } \Omega,\\
& u\in W_0^{1,\overrightarrow{p}}(\Omega), \quad \Phi_0(u,\nabla u) \in L^1(\Omega).
\end{aligned}\right.
\end{equation} 
 
Regarding $\{\theta_j\}_{1\leq j\leq N}$, we distinguish two cases:

\vspace{0.2cm}
{\bf Case 1:} ({\em Non-singular}) For every $1\leq j\leq N$, we have $\theta_j>1$. 

{\bf Case 2:} ({\em Mildly singular}) We have $\theta_j\leq 1$ for some $1\leq j\leq N$. In this case, we will impose some restrictions, such as  $h=0$ and $f\geq 0$ a.e. in $\Omega$, to obtain non-negative solutions of \eqref{yvc}. 
\vspace{0.2cm}

\noindent To give the notion of solution of \eqref{yvc}, for $v\in W_0^{1,\overrightarrow{p}}(\Omega)\cap L^\infty(\Omega)$ and $U_0\in W_0^{1,\overrightarrow{p}}(\Omega)$, we define
\begin{equation} \label{iuv} I_{U_0}(v):=\int_{\{|U_0|>0\}} \Psi(U_0,\nabla U_0)\,v\,dx .
\end{equation}
By a solution of \eqref{yvc} we mean a function $U_0\in W_0^{1,\overrightarrow{p}}(\Omega) $, which is non-negative in Case~2, such that 
$\Phi_0(U_0,\nabla U_0) \in L^1(\Omega)$ and for every $v\in W_0^{1,\overrightarrow{p}}(\Omega)\cap L^\infty(\Omega)$, we have $I_{U_0}(v)\in \mathbb R$ and 
	\begin{equation} \label{omin} 
	 \int_\Omega \sum_{j=1}^N |\partial_j U_0|^{p_j-2}\partial_j U_0\,\partial_j v\,dx
	+ \int_\Omega \Phi_0(U_0,\nabla U_0) \,v\,dx
=I_{U_0}(v)+
\langle h,v\rangle+\int_\Omega f v\,dx.
 \end{equation}
We leverage $\Phi_0$ to get the existence of solutions of \eqref{yvc}  for every $f\in L^1(\Omega)$. This is reflected
in a (lower bound) condition on $m>1$.  
To be precise, we define 
\begin{equation} \label{np}
\begin{aligned}
& N_{\overrightarrow{\mathfrak{a}}}:=\left\{1\leq j\leq N:\  
\mathfrak{a}_j q_j=0,\ \  \frac{\theta_j p_j}{p_j-q_j} \geq p
\right\}, \\
& P_{\overrightarrow{\mathfrak{a}}}:=\left\{1\leq j\leq N :\ 
\mathfrak{a}_j q_j>0,\ 
\mathfrak{m}_j>1\right\},\quad \mbox{where } \mathfrak{m}_j:= \frac{p_j-q_j}{q_j} \left( \frac{ \theta_j p_j}{p_j-q_j} - p \right).
\end{aligned}
\end{equation}
When either $N_{\overrightarrow{\mathfrak{a}}}$ or $P_{\overrightarrow{\mathfrak{a}}}$ is non-empty, we need $m>1$ to satisfy
 \begin{equation}\label{m}
	m> \max_{j\in N_{\overrightarrow{\mathfrak{a}}}}  \frac{\theta_j p_j}{p_j-q_j}\quad  
	\text{and} \quad  m>\min\left\{\theta_j , \mathfrak{m}_j\right\} \ \  \text{for every } j\in P_{\overrightarrow{\mathfrak{a}}}.
	\end{equation} 
	
	We first illustrate our main results for the model problem in 
\eqref{yvc}.  
 
 \begin{theorem} Let \eqref{IntroEq0}--\eqref{ass} and \eqref{m} hold. 
 Let $h\in W^{-1,\overrightarrow{p}'}(\Omega)$ and $f\in L^1(\Omega)$ be arbitrary. 
 When $f\not=0$, we assume that
 $\min_{1\leq j\leq N} \mathfrak{a}_j>0$. Assume Case 1 or Case 2 and, in the latter, let $h=0$ and $f\geq 0$ a.e. in $\Omega$.
 Then, \eqref{yvc} has a solution $U_0\in W_0^{1,\overrightarrow{p}}(\Omega) $. Moreover, for $f=0$, we have that $\Phi_0(U_0,\nabla U_0)\,U_0$ and $\Psi(U_0,\nabla U_0)\,U_0$ belong to $L^1(\Omega)$ and \eqref{omin} holds for $v=U_0$. 
	 \end{theorem}

\subsection{A brief history of the problem}

To understand how our results fit within the literature, we review
what is known in the isotropic case, where the model problem is the following:
\begin{equation}\label{iso}
\left\{\begin{array}{ll}
-\Delta_p u + \lambda |u|^{m-2} u =c(u) |\nabla u|^q +f & \mbox{in } \Omega,
\\ \\
u=0 & \mbox{on } \partial \Omega.
\end{array}
\right.
\end{equation}
Here, $-\Delta_p u=-\mbox{div}\left(|\nabla u|^{p-2}\nabla u\right)$ is the $p$-Laplacian operator with $1<p<\infty$, $\lambda\ge0$, $m>1$, $q\ge0$  and $c(\cdot)$ is a continuous, non-negative function.  We start by considering $\lambda=0$, $c(\cdot)$ constant and $f$ summable enough. 
The case $0 \le q < p-1$ is well-known. Indeed, 
the existence of a solution $u$ in $W_0^{1,p}(\Omega)$ follows easily from {\em a priori} estimates, which are obtained using $u$ as a test function. 
This is part of the general theory of pseudo-monotone operators by J. Leray and J.-L. Lions (see, for example, \cite{LL}). When $f$ has low summability, the main questions appear to be solved (see, for instance, \cite{AM}, \cite{BMMP1} and the references therein). The limiting case $q=p-1$  is more difficult since the operator $-\Delta_p u- c\,|\nabla u|^q$ is not coercive for large $c$. This difficulty has been first overcome by Bottaro and Marina in \cite{BM} when $p=2$, and by various authors in the nonlinear case (see, for example, \cites{BMMP1,DVP}).

We now focus our attention on the case $p-1<q\le p$. 
 When $q=p$, the existence of a bounded weak solution is proved in \cite{BMP} when $f\in L^r(\Omega)$ with $r>N/ p$. The case $f \in L^{N/p}(\Omega)$ is treated in \cite{FM}, which shows   
that there exists a positive constant $C=C(\beta, N,p)$ such that, if $\|f\|_{L^{N/p}(\Omega)}<C$, 
 then a solution $u\in W_0^{1,p}(\Omega)$ of problem \eqref{iso} exists such that 
$ \exp\left(\frac{\beta}{p-1}|u|\right)-1\in W_0^{1,p}(\Omega).$
 Similar results are proved in the case $p-1<q<p$ (see \cites{FMe,GMP} and the references therein). The authors of \cite{AFM} consider the case $p-1<q\le p$ and look for sharp assumptions on $f$ in order to have a solution obtained as a limit of approximations (SOLA).  
 
As far as we know, the more challenging case is $q >p$: it requires a completely different approach and it appears to be largely open (see, for instance, \cite{CDLP} and the references therein).
   
The case $\lambda=0$, $c(u)=u^\alpha$ with $\alpha \ge 0$ and $p=q=2$ is considered in the paper \cite{ADP}. Among other things, the authors prove that if $\alpha>0$ and $f\ge 0$ is sufficiently small, then there exists a positive solution in $H_0^1(\Omega)$.  
In  \cite{AGPW1} (see also \cites{GPS,CLLM,ABLP}) any value of $\alpha\in \R$ and $1<q\le 2$  is allowed. The authors prove that: if $\alpha < -1/q$ and $f\in L^1(\Omega)$, then there exists a distributional solution; if $-1/q\le \alpha < 0$ and $f\in L^r(\Omega)$ with $r>N/2$, then there exists a solution in $H_0^1(\Omega)$; if $\alpha \ge 0$, then there exists a solution only if $f$ is small enough. 
In \cite{GM} the presence of an absorption term, which corresponds to $\lambda>0$ and $m=2$, is used to prove the existence of a bounded solution in $H^1_{\rm{loc}}(\Omega)$ when $\alpha<0$, $p=q=2$ and $f$ is a bounded, non-negative function.

Sharp {\em a priori} estimates for solutions to anisotropic problems with $\lambda=0$ and $c \equiv 0$ have been proved by Cianchi \cite{C} (see also \cites{AdBF1,AdBF2})  by introducing a convenient notion of rearrangement satisfying an anisotropic version of the P\'olya-Szeg\"o  principle. For other results on anisotropic problems we refer the interested reader to the recent papers \cite{ACCZ,CV,dBFZ,FVV,FGK,FGL,GLR}.

We end this section by recalling the paper \cite{BO} (see also \cites{BCT,BFM,GMM} and \cite{LM} for the anisotropic equivalent), where the Dirichlet homogeneous problem relative to the equation $-\Delta u=f/u^\alpha$ is considered. The authors distinguish three cases:
$0<\alpha<1$, $\alpha=1$ and $\alpha>1$. The first two cases can be treated using approximation techniques and providing the existence of a unique solution in $H_0^1(\Omega)$. The validity of a strong comparison principle is a fundamental tool in order to prove the monotonicity, and also a  uniform bound far from zero, of the sequence of solutions of the approximate problems. We stress that this kind of arguments cannot be generalized to the anisotropic setting because of the lack of a strong maximum principle (see \cite{V}).

\subsection{Our general problem} \label{genpr}
We remark that the principal part in \eqref{yvc} is the 
anisotropic $\overrightarrow{p}$-Laplacian operator $\mathcal A u=-\sum_{j=1}^N  \partial_j (|\partial_j u|^{p_j-2}\partial_j u)$. It is the prototype
of a coercive, bounded, continuous and pseudo-monotone operator $\mathcal A:W_0^{1,\overrightarrow{p}}(\Omega)\to W^{-1,\overrightarrow{p}'}(\Omega)$ in divergence form 
\begin{equation} \label{ynz} \mathcal Au= -\sum_{j=1}^N \partial_j (A_j(x,u,\nabla u)).\end{equation} 
In this paper, we give existence results for general singular anisotropic elliptic problems such as 
\begin{equation} \label{eq1}
\left\{ \begin{aligned}
& \mathcal A u+ \Phi(x,u,\nabla u)+\Theta(x,u,\nabla u)=\Psi(u,\nabla u)+ \mathfrak{B} u+f\quad& \mbox{in } \Omega,\\ 
& u\in W_0^{1,\overrightarrow{p}}(\Omega), \quad \Phi(x,u,\nabla u)\in L^1(\Omega),
\end{aligned} \right.
\end{equation}
where $f\in L^1(\Omega)$ and $\mathcal A$ is as in \eqref{ynz} with $A_j(x,t,\xi):\Omega\times \mathbb R\times \mathbb R^N\to  \mathbb R$ a Carath\' eodory function for 
each $1\leq j\leq N$ (that is, measurable on $\Omega$ for every $(t,\xi)\in  \mathbb R\times \mathbb R^N$ and continuous in 
$t,\xi$ for a.e. $x\in \Omega$). Moreover, $\Phi(x,t,\xi),\, \Theta(x,t,\xi):\Omega\times \mathbb R\times \mathbb R^N\to \mathbb R$ are also Carath\' eodory functions. For any $r>1$, let $r'=r/(r-1)$ be the conjugate exponent of $r$. 

The conditions on $\mathcal A$, $\Phi$ and $\Theta$ are similar to those in \cite{baci}. 
We assume that there exist constants $\nu_0,\nu>0$ and non-negative functions  
$\eta_j\in L^{p_j'}(\Omega)$ for $1\leq j\leq N$ 
such that for a.e. $x\in \Omega$, for all $t \in \R$ and every $\xi,\widehat \xi \in \R^N$, we have
\begin{equation} \label{ellip}
\left.\begin{aligned}
& \sum_{i=1}^N A_i(x,t,\xi) \,\xi_i\geq \nu_0\sum_{i=1}^N |\xi_i|^{p_i} && \mbox{[coercivity]},&\\
& \sum_{i=1}^N\left(A_i(x,t,\xi)- A_i(x,t,\widehat \xi) \right)\left(\xi_i-\widehat \xi_i\right)>0\quad \text{if } \xi\not=\widehat \xi && \mbox{[monotonicity]},&\\
& |A_j(x,t,\xi)|\leq \nu \left[ \eta_j(x)+|t|^{p^\ast/p_j'}+\left(\sum_{i=1}^N |\xi_i|^{p_i}\right)^{1/p_j'}
\right]&& \mbox{[growth condition]}.&
\end{aligned}\right\}
\end{equation}

\noindent We note that in the growth condition in \eqref{ellip}, we take the greatest exponent for $|t|$ regarding the anisotropic Sobolev inequalities. For the pseudo-monotonicity of $\mathcal A$, see 
\cite{baci}*{Lemma 2.7}. 

Assume that there exist a constant $C_\Theta>0$, a non-negative function $c\in L^1(\Omega)$ and a continuous non-decreasing function 
$ \phi :\mathbb R\to \mathbb R^+$ 
such that for a.e. $x\in \Omega$ and all $(t,\xi)\in  \mathbb R\times \mathbb R^N$, 
\begin{eqnarray} 
& |\Theta(x,t,\xi)|\leq C_\Theta,\quad t\,\Phi(x,t,\xi)\geq 0, \ \  
|\Phi(x,t,\xi)|\leq \phi(|t|) \left( \dsum_{j=1}^N |\xi_j|^{p_j} +c(x)\right),  \label{cond1}\\ 
& |\Phi(x,t,\xi)|\geq |\Phi_0(t,\xi)|,\quad \mbox{where } \Phi_0 \ \mbox{is as in }\eqref{phiz}.
\label{etc}
\end{eqnarray} 
Compared with \cite{baci}, we have the extra assumption \eqref{etc} to 
deal with the new term $\Psi$ in \eqref{ass}.

The operator $\mathfrak B$ in \eqref{eq1} belongs to the 
general class $\mathfrak{BC}$ 
introduced in \cite{baci}. 
By $\mathfrak{BC}$ we denote the class
of {\em  bounded} operators 
$\mathfrak{B}:W_0^{1,\overrightarrow{p}}(\Omega)\to W^{-1,\overrightarrow{p}'}(\Omega)$ 
satisfying two properties:  
\begin{itemize}
	\item[$(P_1)$]  
	The operator $ \mathcal A-\mathfrak B$ from $W_0^{1,\overrightarrow{p}}(\Omega)$ into $W^{-1,\overrightarrow{p}'}(\Omega)$ is {\em coercive}
in the sense that
$$ \frac{\langle   \mathcal Au-\mathfrak Bu,u\rangle }{\|u\|_{W_0^{1,\overrightarrow{p}}(\Omega)}}\to \infty\quad \mbox{as }\   \|u\|_{W_0^{1,\overrightarrow{p}}(\Omega)}\to \infty.
$$
	\item[$(P_2)$] If $u_\ell  \rightharpoonup u$ and $v_\ell  \rightharpoonup v$ (weakly) in $W_0^{1,\overrightarrow{p}}(\Omega)$ as $\ell\to \infty$, then
	$$\lim_{\ell\to \infty} 	\langle \mathfrak{B} u_\ell,v_\ell\rangle= \langle \mathfrak{B} u, v \rangle.$$
\end{itemize}
We recall from \cite{baci} that 
our assumption $(P_2)$ is somehow reminiscent of $(iii)$ in the Hypothesis $(II)$ of Theorem~1 in the celebrated paper \cite{LL} by Leray and Lions. 
Every operator satisfying $(P_2)$ is strongly continuous (see \cite{baci}) and thus pseudo-monotone 
(cf. \cite[p. 586]{ze}). 
However, unlike  $\mathcal A$, the operator $-\mathfrak{B}$
is not necessarily coercive (see Example~\ref{exg}). 

Let $\mathfrak{BC}_+$ be the class of operators in $\mathfrak {BC}$ satisfying the extra condition
	
$(P_3)$ For $\nu_0>0$ in the coercivity condition of \eqref{ellip} and each $k>0$, it holds 
	\begin{equation*}
	 \nu_0 \sum_{j=1}^N \|\partial_j u\|_{L^{p_j}(\Omega)}^{p_j}-\langle \mathfrak{B}u,T_k u \rangle
	\to \infty \ \mbox{as }	 \|u\|_{W_0^{1,\overrightarrow{p}}(\Omega)}\to \infty.	
	\end{equation*}
 We use $T_k$ for the truncation at height $k$, see \eqref{trun}. 
 
 To indicate that the operator $\mathcal A$ is associated with the class $\mathfrak {BC}$ and $\mathfrak{BC}_+$, respectively, 
we shall write $\mathfrak {BC}(\mathcal A)$ and $\mathfrak{BC}_+(\mathcal A)$, respectively. 
We recall from \cite{baci} examples of $\mathfrak B$ in $\mathfrak {BC}(\mathcal A)$. 

\begin{ex} \label{exg}
	{\rm Let $F\in L^{(p^\ast)'}(\Omega)$ and $h,\widetilde h\in W^{-1,\overrightarrow{p}'}(\Omega)$ be arbitrary. 
		Let $\rho,\alpha_k\in \mathbb R$ for $0\leq k\leq 4$. 
		For every $u\in W_0^{1,\overrightarrow{p}}(\Omega)$,  we define 
		
		(1) $\mathfrak{B} u =h$;
		
		(2) $\mathfrak{B} u = F+\rho\,|u|^{\vartheta-2}u$ with $1 < \vartheta <p$ if $\rho>0$ and $1<\vartheta<p^\ast$ if $\rho<0$;
		
		(3) $\mathfrak{B} u =\left(\alpha_0+\alpha_1 \|u\|^{\mathfrak{b}_1}_{L^r(\Omega)} +\alpha_2
		|\langle \widetilde h,u\rangle|^{\mathfrak{b_2}} \right) \left(
		\alpha_3 h+ \alpha_4  F  \right)$,
		where $ r\in [1,p^\ast)$; we take $ \mathfrak{b}_1\in (0,p/p_1')$ and $ \mathfrak{b}_2\in (0,p_1-1)$  if  $\alpha_3\not=0$; $ \mathfrak{b}_1\in (0,p-1)$ and $\mathfrak{b}_2\in (0,p_1/p')$ if $\alpha_3=0$;
		
		(4) $\mathfrak{B} u = -\sum_{j=1}^N \partial_j\left(\beta_j+|u|^{\sigma_j-1}u\right)$, 
		where $ \beta_j\in L^{p_j'}(\Omega)$ and $\sigma_j\in (0,p/p_j')$ 
		for all $1 \le j \le N$.	
				
In each example, $\mathfrak{B}$ belongs to the class $\mathfrak{BC} ( (1-\varepsilon)\mathcal A)$ for every $\varepsilon\in [0,1)$.}
\end{ex}

\begin{defi} \label{newdef}
A function $U_0\in W_0^{1,\overrightarrow{p}}(\Omega)$, which is non-negative in Case~2, 
is said to be a {\em solution} of \eqref{eq1} if 
	$\Phi(x,U_0,\nabla U_0)\in  L^1(\Omega)$ and for every $v\in W_0^{1,\overrightarrow{p}}(\Omega)\cap L^\infty(\Omega)$,
	\begin{eqnarray} 
	&& \hskip -1.5cm S_{U_0,\Theta,f}(v)=\langle \mathfrak B U_0,v\rangle \quad \mbox{if } \Psi=0, \label{sesim}
	\\
	&& \hskip -1.5cm 
	S_{U_0,\Theta,f}(v)=\langle \mathfrak B U_0,v\rangle+I_{U_0}(v)  \quad \mbox{if } \Psi\not=0 , \label{ff1}
	\end{eqnarray} 
where $I_{U_0}(v)$ and $S_{U_0,\Theta,f}(v)$ are given respectively by \eqref{iuv} and
$$  S_{U_0,\Theta,f}(v):=\langle \mathcal A U_0,v\rangle+ \int_\Omega \Phi(x,U_0,\nabla U_0) \,v\,dx+\int_\Omega \Theta(x,U_0,\nabla U_0)\,v\,dx-\int_\Omega f v\,dx.
$$
\end{defi}
To simplify the notation, we have not included $\mathcal A$ and $\Phi$ in the symbol $S_{U_0,\Theta,f}(v)$. When $f=0$, we simply write
$S_{U_0,\Theta}(v)$ instead of $S_{U_0,\Theta,f}(v)$.

Assuming \eqref{ellip} and \eqref{cond1}, we have shown in \cite{baci}*{Theorem 1.3} that  when $\Psi=0$ and $f=0$, then \eqref{eq1} has a solution $U_0$ for every $\mathfrak B$ in the class $\mathfrak {BC}(\mathcal A)$. 
Moreover, $\Phi(x,U_0,\nabla U_0) \,U_0\in L^1(\Omega)$ and 
 \eqref{sesim} holds for $v=U_0$. If, in addition, there exist constants $l,\gamma>0$ such that 
 \begin{equation} \label{}
 |\Phi(x,t,\xi)|\geq \gamma \sum_{j=1}^N |\xi_j|^{p_j}\quad \mbox{for all } |t|\geq l,\ \mbox{a.e. } x\in \Omega\ \mbox{and all } \xi\in \mathbb R^N,
 \end{equation}
 then \eqref{eq1} with $\Psi=0$ has at least a solution for every $f\in L^1(\Omega)$ and $\mathfrak B$ in the class $\mathfrak {BC}_+(\mathcal A)$. 
 
In this paper, under suitable hypotheses, 
 we prove the existence of solutions for \eqref{eq1} with $\Psi$ in \eqref{ass} (see Theorems~\ref{mainth} and \ref{teo2} below).  Let $v^\pm=\max\{\pm v,0\}$ be the positive and negative parts of $v$. 
In Case 2, we 
look for non-negative solutions of \eqref{eq1} and assume, in addition, that 
\begin{equation} \label{bno}
\left.\begin{aligned}
&\langle \mathfrak{B} v,v^-\rangle\geq 0, \quad  
\langle  \mathfrak{B} w, z\rangle\geq 0 \ \text{for all }
v,w,z\in W_0^{1,\overrightarrow{p}}(\Omega)\ \mbox{with } w,z\geq 0,\\
& f(x)\geq 0,\quad \Theta(x,t,\xi)\leq 0 \quad \mbox{for a.e. } x\in \Omega \ \mbox{and all } (t,\xi)\in \mathbb R\times \mathbb R^N, \\
& \Phi(x,0,0)=0\quad \mbox{and}\quad A_j(x,0,0)=0\ \mbox{a.e. }x\in \Omega,\ \mbox{for all } 1\le j\le N.  
\end{aligned}\right\}
\end{equation} 
  Without further mention, we henceforth understand that \eqref{bno} holds whenever Case 2 occurs.

 Our main results are stated below. 

\begin{theorem} \label{mainth} 
	Let \eqref{IntroEq0}, \eqref{ass}, \eqref{m}, and \eqref{ellip}--\eqref{etc} hold. Let $f=0$ in \eqref{eq1}. Suppose that $\mathfrak B$ belongs to the class $\mathfrak{BC}((1-\varepsilon)\mathcal A)$ for some $\varepsilon\in (0,1)$. 	
	Assume Case~1 or Case~2.  
	Then, there exists a solution $U_0\in W_0^{1,\overrightarrow{p}}(\Omega) \cap L^m(\Omega)$ of \eqref{eq1}. Moreover, both $\Phi(x,U_0,\nabla U_0)\,U_0$ and $\Psi(U_0,\nabla U_0) \,U_0$ belong to 
	 $L^1(\Omega)$ and \eqref{ff1} holds for $v=U_0$.
\end{theorem}

When $N_{\overrightarrow{\mathfrak{a}}}\cup P_{\overrightarrow{\mathfrak{a}}}=\emptyset$, then Theorem~\ref{mainth} gives that 
\eqref{eq1} admits a solution for {\em every} $m>1$. 

If in the framework of Theorem~\ref{mainth}, we have $\min_{1\leq j \leq N} \mathfrak a_j>0$ (in relation to \eqref{etc}), then we obtain the existence of solutions for \eqref{eq1} for every $f\in L^1(\Omega)$ and $\mathfrak B$ in the class $\mathfrak{BC}_+((1-\varepsilon)\mathcal A)$ for some $\varepsilon\in (0,1)$. More precisely, we prove the following result.

\begin{theorem} \label{teo2} 
	Let \eqref{IntroEq0}, \eqref{ass}, \eqref{m} and \eqref{ellip}--\eqref{etc} hold and, in addition, $\min_{1\leq j \leq N} \mathfrak a_j>0$. 
	Let $f\in L^1(\Omega)$. Suppose that $\mathfrak B$ belongs to the class $\mathfrak{BC}_+((1-\varepsilon)\mathcal A)$ for some $\varepsilon\in (0,1)$. 	
	Assume Case~1 or Case~2. Then, \eqref{eq1} has at least a solution  $u_0\in W_0^{1,\overrightarrow{p}}(\Omega) $. 
\end{theorem}

\subsection{ Notation} As usual, in the following sections, we will denote by $C$ a positive constant, the value of which can change from line to line. 
\noindent For $k>0$, we let $T_k:\mathbb R\to \mathbb R$ stand for the truncation at height $k$, that is,  
\begin{equation} \label{trun} T_k(s)=s \quad \mbox{if } |s| \le k, \quad T_k(s)=k\, \frac{s}{|s|} \quad \mbox{if }  |s|>k.\end{equation}
\noindent Moreover, we define $G_k:\mathbb R\to \mathbb R$ by  
\begin{equation}\label{gk}
G_k(s)=s-T_k(s)\quad \mbox{for every }s\in \R,
\end{equation}
 so that $G_k=0$ on $[-k,k]$.  

\noindent For every $u\in  W_0^{1,\overrightarrow{p}}(\Omega)$ and 
for a.e. $x\in \Omega$, we define  
\begin{equation*} \label{zeta}  
\begin{aligned}
& A_j(u)(x):=A_j(x,u(x),\nabla u(x))\quad \text{for every } 1\leq j\leq N, \\
& \Phi(u)(x):=\Phi(x,u(x),\nabla u(x)),\quad \Phi_0(u)(x)=\Phi_0(u(x),\nabla u(x)), \\
&  \Theta(u)(x):=\Theta(x,u(x),\nabla u(x)), \quad \Psi(u)(x):=\Psi(u(x),\nabla u(x)).
\end{aligned}
\end{equation*}
For $u,v,w\in W_0^{1,\overrightarrow{p}}(\Omega)$, we introduce $\mathcal E_u (v,w)$ as follows
\begin{equation} \label{muz}   
\mathcal E_u(v,w):=\sum_{j=1}^N \left[A_j(x,u,\nabla v) -A_j(x,u,\nabla w)\right] \,\partial_j(v-w).
\end{equation}

\noindent We set $\overrightarrow{p}=\left(p_1,p_2,\ldots,p_N\right)$ and 
$\overrightarrow{p}'=(p_1',p_2',\ldots,p_N')$. 

\noindent As usual, $\chi_\omega$ stands for the characteristic function of a set $\omega\subset \mathbb R^N$.

\subsection{Strategy for the proof of Theorems~\ref{mainth} and \ref{teo2}} \label{strat}
We first take $f=0$ in \eqref{eq1} and in the framework of Theorem~\ref{mainth}, we obtain a solution $U_0$ 
(with additional properties that $\Phi(U_0)\, U_0\in L^1(\Omega)$ and 
$\Psi(U_0) \,U_0\in L^1(\Omega)$, allowing us to take $v=U_0$ in \eqref{ff1}). The difficulty in our analysis arises from the interaction of the absorption term $\Phi$ with 
the gradient-dependent lower order term $\Psi$. We point out that $\Psi$ cannot be integrated into $\Phi$ since they have the same sign but appear in the opposite sides of \eqref{eq1}. 
Moreover, $\Psi(u)$ is not part of $\mathfrak B u$ either (except in very special cases such as $q_j=0$ and $1<\theta_j<p$ for all $1\leq j\leq N$). 
Hence, we cannot tackle $\Psi(u)$ directly in the framework of our paper \cite{baci}. We overcome this obstacle by approximating $\Psi(u)$ by {\em bounded} functions $\Psi_n(u)$ with $\|\Psi_n(u)\|_{L^\infty(\Omega)}\leq Nn$ for every $n\geq 1$ (see Section~\ref{appr}).  

We consider a sequence of approximate problems corresponding to \eqref{eq1} with $f=0$ and $\Psi$ replaced by $\{\Psi_n\}_{n\geq 1}$. Then, for each $n\geq 1$, by applying Theorem 1.3 in \cite{baci}, we obtain the existence of a solution $U_n\in W_0^{1,\overrightarrow{p}}(\Omega) \cap L^m(\Omega)$ for the approximate problem
\begin{equation}  \label{e1b}
\left\{\begin{aligned}
& \mathcal A U  + \Phi(U)+\Theta(U)=\Psi_n (U)+\mathfrak{B} U \quad 
\mbox{in } \Omega,\\
& U\in W_0^{1,\overrightarrow{p}}(\Omega), \quad \Phi(U)\in L^1(\Omega).
\end{aligned} \right.
\end{equation}
Moreover, $U_n$ is non-negative in Case~2 in view of the hypothesis \eqref{bno} (see Lemma  \ref{pol}).
We capture the properties of $U_n$ in Proposition~\ref{red} to be proved in Section~\ref{sec3}. 
We are able to get a suitable upper bound for $\int_\Omega \Psi(U_n) \,U_n\,dx$ via Lemma~\ref{intorj}. 
To show that $ \{U_n\}_n$ is bounded in $W_0^{1,\overrightarrow{p}}(\Omega)$ and also in $L^m(\Omega)$, we rely on \eqref{m} and the property $(P_1)$ of $\mathfrak{B}$ in the class 
$\mathfrak{BC} ((1-\varepsilon)\,\mathcal A)$. Hence, up to a subsequence, $\{U_n\}_{n\geq 1}$ converges weakly in both $W_0^{1,\overrightarrow{p}}(\Omega)$ and $L^m(\Omega)$ 
to a function $U_0\in W_0^{1,\overrightarrow{p}}(\Omega)\cap L^m(\Omega)$. It turns out that $U_0$ is a good candidate for a solution of \eqref{eq1}. 
In addition to $\Psi_n (U_n) $, we need to handle another gradient-dependent 
term, namely, $\Phi(U_n)$. 
To deal with these terms, we show in Proposition~\ref{red2} that, up to a subsequence, 
\begin{equation} \label{oxx} U_n\to U_0\ \mbox{(strongly) in } W_0^{1,\overrightarrow{p}}(\Omega)\ \mbox{as } n\to \infty.
\end{equation}
To prove \eqref{oxx}, it is enough to show that 
for a subsequence of $\{U_n\}_n$, we have
 \begin{equation} \label{alu} 
  \mathcal E_{U_n}(U_n^\pm,U_0^\pm):=\sum_{j=1}^N [A_j(x,U_n,\nabla U_n^\pm) -A_j(x,U_n,\nabla  U_0^\pm)] \,\partial_j(U_n^\pm-U_0^\pm)
\to 0\quad \mbox{in } L^1(\Omega)
\end{equation}
as $ n\to \infty$. 
Indeed, 
from \eqref{alu} we obtain  that, up to a subsequence,
$\nabla U_n^\pm\to \nabla U_0^\pm$ a.e. in $\Omega$ and 
 $U_n^\pm \to U_0^\pm$ (strongly) in $ W_0^{1,\overrightarrow{p}}(\Omega)$ as $n\to \infty$. 
 For details, see Lemma~\ref{cori} in the Appendix.

Broadly speaking, the proof of \eqref{alu} is inspired by the approach in the celebrated paper \cite{BBM} dealing with Leray--Lions operators from $W_0^{1,p}(\Omega)$ into $W^{-1,p'}(\Omega)$. 
We point out that, in our case, the analysis becomes more technically involved given the anisotropic setting with the modified growth condition in \eqref{ellip} and the inclusion of $\mathfrak B$ and $\Psi$. 
Based on the property $(P_2)$ of $\mathfrak{B}$ and a careful
use of the absorption term, we show that $\limsup_{n\to \infty} \|G_k(U_n)\|_{W_0^{1,\overrightarrow{p}}(\Omega)}\leq W_k$, where $\lim_{k\to \infty} W_k=0$, see Lemma~\ref{step3}. This is an essential tool not only in the proof of \eqref{alu} but also in that of 
Lemma~\ref{cori} (see Remark~\ref{copp}). The technical details in the proof of Proposition~\ref{red2} are deferred to Section~\ref{redpro}.

Then, by Propositions~\ref{red} and \ref{red2}, we can apply Vitali's Theorem to obtain that 
\begin{equation} \label{c42} \Phi(U_n)\to \Phi(U_0)\quad\mbox{ in } L^1(\Omega)\quad\mbox{ as }n\to \infty.\end{equation} 
We end the proof of Theorem~\ref{mainth} by showing that, up to a subsequence of $U_n$, we have  
\begin{equation} \label{c4}
I_{U_0}(v)=\lim_{n\to \infty}\int_\Omega \Psi_n (U_n)\,v\,dx=
S_{U_0,\Theta}(v)
-\langle\mathfrak{B} U_0, v\rangle
\end{equation} 
for every $ v \in W_0^{1,\overrightarrow{p}}(\Omega)\cap L^\infty(\Omega)$. 
For details see Section 3.

We remark that it is possible to make the proof of Theorem~\ref{mainth} work with only the strong convergence in $W_0^{1,\overrightarrow{p}}(\Omega)$ for 
 the truncations $T_k(U_n)$, namely, proving that up to a subsequence, 
 \begin{equation} \label{repo} T_k(U_n)\to T_k(U_0)\ \mbox{in } W_0^{1,\overrightarrow{p}}(\Omega)\ \mbox{ as } n\to \infty,\ \mbox{ for every } k\geq 1. \end{equation} 

It is this latter strategy that we adopted in our paper \cite{baci} for $\Psi=0$ (inspired by \cite{BGM}), first to obtain the existence of solutions for $f=0$ and then building upon it also for $f\in L^1(\Omega)$. But unlike 
 Theorem~\ref{mainth}, the approximation argument for $f=0$ in \cite{baci} concerned the absorption term $\Phi$.

For Theorem~\ref{teo2} dealing with a low summability term $f\in L^1(\Omega)$, we use a well-known approximation: we replace $f$ in \eqref{eq1} by a sequence $\{f_n\}_{n\geq 1}$ of $L^\infty(\Omega)$-functions such that
$|f_n|\leq |f|$ for each $n\geq 1$ and $f_n\to f$ in $L^1(\Omega)$ as $n\to \infty$. For the approximate problem, we use Theorem~\ref{mainth} to gain a solution $u_n\in W_0^{1,\overrightarrow{p}}(\Omega)$. 
The additional assumption $\min_{1\leq j \leq N} \mathfrak a_j>0$ and the extra property $(P_3)$ for $\mathfrak B$ in $\mathfrak{BC}_+((1-\varepsilon)\mathcal A)$ are needed to obtain in Proposition~\ref{lem-ad} that the solutions $u_n$ are uniformly bounded in $W_0^{1,\overrightarrow{p}}(\Omega)$ with respect to $n$. Since here we test the approximate problem with $T_k(u_n)$ (and not $u_n$, which is potentially unbounded), we can only derive that $\{\Phi(u_n)\}_n$ (and not $\{\Phi(u_n)\,u_n\}_{n\geq 1}$)  is uniformly bounded in $L^1(\Omega)$ uniformly with respect to $n$. However, this suffices to get that $\Phi(u_0)\in L^1(\Omega)$, where $u_0$ is the weak limit in $W_0^{1,\overrightarrow{p}}(\Omega)$ of (a subsequence of) $\{u_n\}_{n\geq 1}$. 
In Proposition~\ref{st-con}, we establish 
the analogue of \eqref{repo}. To this end, we use \cite{baci}*{Lemma~A.5} (and a diagonal argument) to reduce the proof to showing that for every $k\geq 1$ and, up to subsequence, 
 \begin{equation} \label{alt}  \mathcal E_{u_n}(T_k(u_n),T_k(u_0))\to 0\quad \mbox{ in } L^1(\Omega)\ \mbox{ as }n\to \infty. \end{equation}
(For the definition of $\mathcal E_u$, see \eqref{muz}.) To prove \eqref{alt}, we adapt 
the approach in our paper \cite{baci} 
by testing the approximate problem with $$v=(T_k(u_n)-T_k(u_0))\,\exp\left(\lambda\, (T_k(u_n)-T_k(u_0))^2 \right)$$ for $\lambda=\lambda(k)>0$ large enough. The new ingredient here corresponds to getting a good control of $I_{u_n}$ for this test function (see Lemma~\ref{lema74}).

Bearing in mind
the strong convergence    
of $T_k(u_n)$ to $T_k(u_0)$ in $W_0^{1,\overrightarrow{p}}(\Omega)$ as $n\to \infty$, we can obtain the analogue of \eqref{c42} and then pass to the limit in the approximate problem 
to obtain suitable counterparts of \eqref{c4}  (see Lemma~\ref{lemi} for details). Putting together the above results, we conclude that $u_0\in W_0^{1,\overrightarrow{p}}(\Omega)$ is a solution of \eqref{eq1}.

\subsection{Structure of the paper} In Section \ref{appr} we consider the sequence of approximate problems \eqref{e1b} and we establish the existence of solutions, which are non-negative in Case 2. We state the {\it a priori} estimates and the strong convergence of such solutions  in $W_0^{1,\overrightarrow{p}}(\Omega)$, deferring their proofs to Section \ref{sec3} and \ref{redpro}, respectively. Based on these properties, we complete the proof of Theorem \ref{mainth} in Section \ref{sec-th1}. In Section \ref{sec5} we include several results that are invoked in Section \ref{redpro}.  Section \ref{sec7} contains the proof of Theorem \ref{teo2}. For the reader's convenience, in the Appendix we present some details which are modifications of  arguments known in the literature or already contained in our recent paper \cite{baci}.


\section{Approximate problems} \label{appr}

We always assume that \eqref{IntroEq0}, \eqref{ass}, \eqref{ellip} and \eqref{cond1} hold. Unless otherwise stated, we also understand that $\Phi$ satisfies \eqref{etc} and $\mathfrak B$ belongs to the class $\mathfrak{BC} ((1-\varepsilon)\mathcal A)$ for some $\varepsilon\in (0,1)$ (see Section~\ref{appre} below for an exception). We first take $f=0$ in \eqref{eq1}. 

\subsection{Setting up the approximation}
We introduce the sets 
$$ J_1:=\{ 1\leq j\leq N:\quad  \theta_j>1 \},\quad 
J_2:=\{ 1\leq j\leq N:\quad 0< \theta_j \leq 1\}. 
$$
Case~1 (respectively, Case 2) in Theorem~\ref{mainth} corresponds to $J_2=\emptyset$ (respectively, $J_2\ne \emptyset$).  

 Let $n\geq 1$ be arbitrary. For each $1\leq j\leq N$, we define 
\begin{equation} \label{foa} H_{j,n}(t_1,t_2)= 
 \frac{|t_1|^{\theta_j-2}t_1\, |t_2|^{q_j}} 
{ 1+\frac{1}{n} |t_1|^{\theta_j-1} |t_2|^{q_j}}\quad \mbox{for every } (t_1,t_2)\in (0,\infty)\times \mathbb R. 
\end{equation}

In Case 1, for each $1\leq j\leq N$, we extend $H_{j,n}(t_1,t_2)$ on $(-\infty,0]\times \mathbb R$ with the same formula as in \eqref{foa}. 
In Case~2, for each 
$j\in J_1$ (when $J_1$ is not empty), we extend $H_{j,n}(t_1,t_2)$ on $(-\infty,0]\times \mathbb R$ so that it becomes an even function in the first variable. 

In Case 1 or Case 2, we define 
$\Psi_n$ 
from $W_0^{1,\overrightarrow{p}}(\Omega)$ into $ L^\infty(\Omega)$ as follows
\begin{equation} \label{psi}
\Psi_n(u):=  \Psi_{n,J_1}(u)+\Psi_{n,J_2}(u),
\end{equation}
where $\Psi_{n,J_1}(u)$ and $\Psi_{n,J_2}(u)$ are functions from $\Omega$ to $\mathbb R$ given by 
\begin{equation} \label{psia}   \Psi_{n,J_1}(u)(x):= \sum_{j\in J_1} H_{j,n}(u(x),\partial_j u(x)), \quad 
\Psi_{n,J_2}(u)(x):=\sum_{j\in J_2} H_{j,n}\left(|u(x)|+1/n,\partial_j u(x)\right).
\end{equation}

Clearly, $\Psi_n(u)\in L^\infty(\Omega)$ for every $u\in W_0^{1,\overrightarrow{p}}(\Omega)$ and  
$ \|\Psi_{n}(u)\|_{L^\infty(\Omega)}\le N n$. 

As explained in Section~\ref{strat}, we consider the approximate problem \eqref{e1b}.

\subsection{Existence of solutions for \eqref{e1b}} \label{appre}
We point out that for the existence of solutions of \eqref{e1b}, we do not need $\Phi$ to satisfy \eqref{etc}. Moreover, the operator $\mathfrak B$ can be taken in the class $\mathfrak{BC} (\mathcal A)$ (rather than $\mathfrak{BC} ((1-\varepsilon)\mathcal A)$ for some $\varepsilon\in (0,1)$).

\begin{lemma} \label{pol} 
Let \eqref{IntroEq0}, \eqref{ass}, \eqref{ellip} and \eqref{cond1} hold. Suppose that $\mathfrak B$ belongs to the class $\mathfrak{BC}(\mathcal A)$.  
	Assume Case 1 or Case 2. 
	For every $n\geq 1$, problem \eqref{e1b} admits a solution $U_n$, which in addition satisfies $\Phi(U_n)\,U_n\in L^1(\Omega)$ and  
	\begin{equation} \label{e22v}
	S_{U_n,\Theta}(U_n)
	=\int_\Omega 
	\Psi_n (U_n)\,U_n\,dx +\langle \mathfrak{B}U_n,U_n\rangle .
	\end{equation}  
	Moreover, in Case 2, we have $U_n\geq 0$ a.e. in $\Omega$.
	\end{lemma}

\begin{proof} By applying Theorem~1.3 in \cite{baci}, with $\Theta$ there replaced by $\Theta-\Psi_n$, we obtain that \eqref{e1b} has a solution $U_n$ (in the sense of Definition \ref{newdef} with $\Psi=0$), 
 satisfying 
 \begin{equation} \label{fff2}
S_{U_n,\Theta}(v)=
\int_\Omega \Psi_n(U_n)\,v\,dx +\langle \mathfrak{B}U_n,v\rangle \qquad \mbox{ for all } \ v \in W_0^{1,\overrightarrow{p}}(\Omega) \cap L^\infty(\Omega).
\end{equation} 
Moreover, $\Phi(U_n)\,U_n\in L^1(\Omega)$ and \eqref{e22v} holds. 
We now show that in   
Case 2, we have $U_n\geq 0$ a.e. in $\Omega$. 
Since $U_n^-$ may not be in $L^\infty(\Omega)$, we cannot directly use $v=U_n^-$ in \eqref{fff2}. 
However, for every $k>0$, we have $T_k(U_n^-)\in W_0^{1,\overrightarrow{p}}(\Omega) \cap L^\infty(\Omega)$. Hence, by taking $v=T_k (U_n^-)$ in \eqref{fff2}, we obtain that 
\begin{equation} \label{laci}  
S_{U_n,\Theta}(T_k(U_n^-))
=
\int_\Omega \Psi_n(U_n)\,T_k(U_n^-)\,dx+\langle \mathfrak{B}U_n,T_k(U_n^-)\rangle.
\end{equation}
Notice that $\|T_k (U_n^-)\|_{W_0^{1,\overrightarrow{p}}(\Omega)}\leq \|U_n^-\|_{W_0^{1,\overrightarrow{p}}(\Omega)}$ for all $k>0$. 
Moreover,  $\partial_j (T_k (U_n^-))\to \partial_j U_n^-$ a.e. in $\Omega$ as $k\to \infty$, for every $1\leq j\leq N$,  so that 
$  T_k (U_n^-)\rightharpoonup U_n^-$ (weakly) in $W_0^{1,\overrightarrow{p}}(\Omega)$ as $k\to \infty.$ 
Since $\mathcal A U_n$ and $\mathfrak{B} U_n$ belong to $W^{-1,\overrightarrow{p}'}(\Omega)$, it follows that 
\begin{equation*} \label{tas1} 
\lim_{k\to \infty} \langle \mathcal AU_n, T_k (U_n^-)\rangle = \langle \mathcal AU_n, U_n^-\rangle \quad \mbox{and } \quad \lim_{k \to \infty} \langle \mathfrak{B}U_n,T_k(U_n^-)\rangle=\langle \mathfrak{B} U_n, U_n^-\rangle. 
\end{equation*}
Recalling that $\Phi(U_n)\, U_n \in L^1(\Omega)$, $ \|\Psi_{n}(U_n)\|_{L^\infty(\Omega)}\le N n$ and \eqref{cond1} holds, from the Dominated Convergence Theorem, we can pass to the limit $k\to \infty$ in \eqref{laci} to find that 
\begin{equation} \label{nov} 
S_{U_n,\Theta}(U_n^-)=\int_\Omega \Psi_n(U_n)\,U_n^-\,dx+
\langle \mathfrak{B}U_n,U_n^-\rangle.
\end{equation}
In view of \eqref{bno}, we see that the right-hand side of \eqref{nov} is non-negative. Using also the coercivity condition in \eqref{ellip}, we infer that the left-hand side of \eqref{nov} is bounded above by 
\begin{equation} \label{qua}
-\left(\nu_0\sum_{j=1}^N\int_{\{U_n<0\}} |\partial_j U_n|^{p_j}\,dx +\int_{\{U_n<0\}} \Phi(U_n)\,U_n\, dx +\int_{\{U_n<0\}}  \Theta(U_n) \,U_n\,dx\right) .
\end{equation}
From the sign-conditions on $\Phi$ and $\Theta$ in \eqref{cond1} and \eqref{bno}, respectively, we  see that all terms contained in the round brackets of \eqref{qua} are non-negative.
Hence, $\mbox{meas}\,(\{U_n<0\})=0$ and so $U_n \ge 0$ a.e. in $\Omega$.
\end{proof}

\begin{rem} \label{8nb} If, in addition, $\Phi$ satisfies \eqref{etc}, then for the solution $U_n$ of \eqref{e1b} provided by Lemma~\ref{pol}, we have
$U_n\in L^m(\Omega)$. This follows from the property $\Phi(U_n) \,U_n\in L^1(\Omega)$. 
\end{rem}

\subsection{Strong convergence of $U_n$} 
Throughout this section, we work in the framework of Theorem~\ref{mainth}. Then, Lemma~\ref{pol} and Remark~\ref{8nb} give that for every $n\geq 1$, the approximate problem \eqref{e1b} has a solution $U_n\in W_0^{1,\overrightarrow{p}}(\Omega)\cap L^m(\Omega)$. 
In Proposition~\ref{red}  we derive essential {\em a priori} estimates in $W_0^{1,\overrightarrow{p}}(\Omega)$ and in $L^m(\Omega)$ for the sequence of solutions $\{U_n\}_{n\geq 1}$,
which up to a subsequence, converges weakly to some $U_0$ both in $W_0^{1,\overrightarrow{p}}(\Omega)$ and in $L^m(\Omega)$. 
In Proposition~\ref{red2}, we show that, up to a subsequence, $\{U_n\}_{n\ge 1}$ converges strongly to $U_0$ in $W_0^{1,\overrightarrow{p}}(\Omega)$ as $n\to \infty$, see \eqref{gtree}. We aim to prove that  $U_0$  is a solution of \eqref{eq1} with $f=0$.  
In Sections~\ref{sec3} and \ref{redpro}, respectively, we prove Propositions~\ref{red} and \ref{red2}, which are
the crux of the proof of Theorem~\ref{mainth}. 

\begin{prop} \label{red} 
Let \eqref{IntroEq0}, \eqref{ass}, \eqref{m} and \eqref{ellip}--\eqref{etc} hold. Let $f=0$. Suppose that $\mathfrak B$ belongs to the class $\mathfrak{BC}((1-\varepsilon)\mathcal A)$ for some $\varepsilon\in (0,1)$. Assume Case~1 or Case~2. 
\begin{itemize} 
\item[(a)] There exists a constant $C>0$ such that for every $n\geq 1$, the solution 
$U_n$ given by Lemma~\ref{pol} satisfies
\begin{equation}\label{e2} 
\|U_n\|_{W_0^{1,\overrightarrow{p}}(\Omega)} + \|U_n\|_{L^m(\Omega)} +
  \int_\Omega \Phi(U_n)\,U_n\, dx+
 \int_\Omega \Psi(U_n)\,U_n\, dx \leq C.
\end{equation}
\item[(b)] There 
exists $U_0 \in W_0^{1,\overrightarrow{p}}(\Omega)\cap L^m(\Omega)$ such that, up to a subsequence, 
\begin{equation} \label{uni} 
\begin{aligned}
&U_n \rightharpoonup U_0\ \mbox{(weakly) both in }  W_0^{1,\overrightarrow{p}}(\Omega)\ 
\mbox{and in } L^m(\Omega)\ \mbox{as } n\to \infty,
\\
&U_n \to U_0 \mbox{ a.e. in }\Omega\ \mbox{as } n\to \infty.
\end{aligned}
\end{equation}

\end{itemize}
\end{prop}

\begin{prop} \label{red2} In the framework of Proposition~\ref{red}, up to a subsequence, 
we have
\begin{eqnarray} 
& \nabla U_n\to \nabla U_0 \ \mbox{a.e. in } \Omega\ \mbox{as } n\to \infty, \label{gtr}\\
& U_n\to U_0\ \mbox{(strongly) in } W_0^{1,\overrightarrow{p}}(\Omega)\ \mbox{as } n\to \infty. \label{gtree}
\end{eqnarray}   
\end{prop}

\begin{rem} \label{r24} Under the same assumptions as in Proposition~\ref{red}, by Fatou's Lemma we immediately infer that $\Phi(U_0)\in L^1(\Omega)$. Furthermore, using Fatou's Lemma and \eqref{e2}--\eqref{gtr}, we find that 
$\Phi(U_0)\,U_0$ and $\Psi(U_0) \,U_0$ belong to $L^1(\Omega)$. 
\end{rem}


\section{Proof of Theorem~\ref{mainth} completed} \label{sec-th1}

Let $m$ satisfy \eqref{m} and $f=0$. 
We show that the function $U_0$ in Proposition~\ref{red2} is a solution of \eqref{eq1}. Once this is established, we readily obtain  that \eqref{ff1} holds  for $v=U_0$ in both Case~1 and Case~2 with a reasoning similar to Lemma~\ref{pol}. Indeed, by taking $v=T_k (U_0)$ in \eqref{ff1} and letting $k$ go to infinity, we get the claim.

\vspace{0.2cm}
We now prove \eqref{ff1}. As already pointed out in  Section 1.4, we just need to check \eqref{c4} for every  $v \in W_0^{1,\overrightarrow{p}}(\Omega)\cap L^\infty(\Omega)$.
We first establish the second identity in \eqref{c4}, that is 
\begin{equation} \label{lli} 
\lim_{n\to \infty} \int_\Omega \Psi_n(U_n)\,v\,dx=
S_{U_0,\Theta}(v)
-\langle\mathfrak{B} U_0, v\rangle .
\end{equation}

{\em Proof of \eqref{lli}.} Since $U_n\to U_0$ and $\nabla U_n\to \nabla U_0$ a.e. in $\Omega$ as $n\to \infty$, we have 
\begin{equation} \label{sma} 
 \Theta(U_n) \to \Theta(U_0)\quad \mbox{and}\quad A_j(U_n)\to A_j(U_0)\ \ \mbox{a.e. in } \Omega \ \ \mbox{for } 1\leq j\leq N. 
 \end{equation} 
Let $v \in W_0^{1,\overrightarrow{p}}(\Omega)\cap L^\infty(\Omega)$ be arbitrary. Now, $\Theta$ satisfies \eqref{cond1}. Thus, by the Dominated Convergence Theorem, we obtain that 
$ \Theta(U_n)\, v\to \Theta(U_0)\, v$ in $L^1(\Omega)$ as $n\to \infty$. 
Since $\{A_j(U_n)\}_{n\geq 1}$ is uniformly bounded in $L^{p_j'}(\Omega)$ with respect to $n$, from \eqref{sma} we get that, up to a subsequence, 
 \begin{equation} \label{jok2} A_j(U_n)\rightharpoonup A_j(U_0)\ \mbox{ (weakly) in } L^{p_j'}(\Omega)\ \mbox{as } n\to \infty\end{equation}
 for every $1\leq j\leq N$. 
It follows that $\lim_{n\to \infty}  \langle \mathcal{A} U_n, v\rangle = \langle \mathcal{A} U_0, v\rangle $. 
Using that $U_n\rightharpoonup U_0$ (weakly) in $W_0^{1,\overrightarrow{p}}(\Omega)$ as $n\to \infty$, the property $(P_2)$ for the operator $\mathfrak B$ yields that 
$\lim_{n\to \infty}  \langle \mathfrak{B} U_n, v\rangle = \langle \mathfrak{B} U_0, v\rangle$. Thus, by passing to the limit as $n\to \infty$ in \eqref{fff2}, 
we gain \eqref{lli} whenever
\begin{equation}\label{gli}  
\Phi(U_n)\to \Phi(U_0) \text{ (strongly) in } L^1(\Omega)\ \mbox{as } n\to \infty. 
\end{equation} 
Since 
$\Phi(U_n)\to \Phi(U_0)$ a.e. in $\Omega$ as $n\to \infty$ and $\Phi(U_0)\in L^1(\Omega)$ (see Remark~\ref{r24}), 
by Vitali's Theorem,  
it is enough to show that $\{\Phi(U_n)\}_n$ is uniformly integrable over $\Omega$. Let $\omega$ be 
any measurable subset  of $\Omega$ and $M>0$ be arbitrary.  
By the growth condition of $\Phi$ in \eqref{cond1}, we have
\begin{equation} \label{net}  
\int_{\omega \cap \{|U_n| \le M\}} |\Phi(U_n)| \, dx \le 
\phi(M)\left(\sum_{j=1}^N  \|\partial_j T_M(U_n)\|^{p_j}_{L^{p_j}(\omega)}+
\int_{\omega} c(x)\,dx
\right).
\end{equation}
On the other hand, using \eqref{e2} and the sign-condition on $\Phi$ in \eqref{cond1}, we see that 
\begin{equation} \label{new1} 
\int_{\omega \cap \{|U_n| > M\}} |\Phi(U_n)| \, dx \le 
\frac{1}{M}\int_{\omega} \Phi(U_n) \,U_n\,dx\leq \frac{C}{M}, 
\end{equation} where $C>0$ is a constant independent of $n$ and $\omega$. 
Since $c\in L^1(\Omega)$ and $\partial_j T_M(U_n)\to \partial_j T_M(U_0)$ in $L^{p_j}(\Omega)$ as $n\to \infty$ for all $1 \le j \le N$ (see \eqref{gtree}), from \eqref{net} and \eqref{new1} we get the equi-integrability of $\{\Phi(U_n)\}_n$ over $\Omega$.
By Vitali's Theorem, we end the proof of \eqref{gli}. \qed

\vskip 0.2cm
It remains to show the first identity in \eqref{c4}, that is

\begin{equation} \label{lima}
\lim_{n\to \infty} \int_\Omega \Psi_n(U_n)\,v\,dx=
\int_{\{|U_0|>0\}}  \Psi(U_0)\,v\,dx
\end{equation} 
for every  $v \in W_0^{1,\overrightarrow{p}}(\Omega)\cap L^\infty(\Omega)$. 

\noindent Recall that $N_{\overrightarrow{\mathfrak{a}}}$ and $P_{\overrightarrow{\mathfrak{a}}}$ are given in \eqref{np}. We define 
\begin{equation} \label{npc}
\begin{aligned}
& N_{\overrightarrow{\mathfrak{a}}}^{\,c}:=\left\{1\leq j\leq N:\  
\mathfrak{a}_j q_j=0,\quad \frac{\theta_j p_j}{p_j-q_j} < p
\right\}, \\
& P_{\overrightarrow{\mathfrak{a}}}^c:=\left\{1\leq j\leq N :\ 
\mathfrak{a}_jq_j>0,\quad 
\mathfrak{m}_j\leq 1\right\}.
\end{aligned}
\end{equation}
It follows that  
\begin{equation} \label{univ} 
 \{1\leq j\leq N:\ \mathfrak{a}_jq_j=0\}=N_{\overrightarrow{\mathfrak{a}}}\cup N_{\overrightarrow{\mathfrak{a}}}^c,\qquad 
\{1\leq j\leq N:\ \mathfrak{a}_jq_j>0\}=
P_{\overrightarrow{\mathfrak{a}}}\cup P_{\overrightarrow{\mathfrak{a}}}^c.
\end{equation}

For every $1\leq j\leq N$, we introduce the notation
	\begin{equation} \label{impj} 
	\mathfrak{I}_{m,p_j} (U_n):=\int_\Omega
		|U_n|^m |\partial_j U_n|^{p_j}\,dx,\qquad 
		\mathfrak{I}_{\theta_j,q_j}(U_n):=\int_\Omega |U_n|^{\theta_j} |\partial_j U_n|^{q_j}\,dx.
\end{equation}
To prove \eqref{lima}, we treat Case 1 in Subsection~\ref{sec-c1} and Case 2 in Subsection~\ref{sec-th2}.  

\subsection{Proof of \eqref{lima} in Case~1} \label{sec-c1}
Here,  $\theta_j>1$ for each $1\leq j\leq N$. 
Since $J_2=\emptyset$, from \eqref{psi} and \eqref{psia}, we find that  
$\Psi_n(U_n)=\Psi_{n,J_1}(U_n)=\sum_{j=1}^N H_{j,n} (U_n,\partial_j U_n)$,  with $H_{j,n}(\cdot,\cdot)$ defined in \eqref{foa}.
So, to prove \eqref{lima}, it suffices to show that (up to a subsequence) 
\begin{equation} \label{ioc} \lim_{n\to \infty}
\int_\Omega H_{j,n}\left(U_n,\partial_jU_n\right) v\,dx= 
\int_\Omega  |U_0|^{\theta_j-2} U_0 \,|\partial_j U_0|^{q_j}\,v\,dx
\end{equation} 
for every $1\leq j\le N$ and all  $v \in W_0^{1,\overrightarrow{p}}(\Omega)\cap L^\infty(\Omega)$. 

\noindent Let $1\leq j\leq N$ be arbitrary. By Proposition~\ref{red2}, we have
\begin{equation} \label{cob1}  H_{j,n}(U_n,\partial_j U_n )\to |U_0|^{\theta_j-2}U_0\, |\partial_j U_0|^{q_j}\quad \mbox{a.e. in }\Omega\ \mbox{as } n\to \infty. \end{equation}

\noindent We next show that there exists $s>1$ (depending on $j$) such that 
	\begin{equation} \label{lsmvbf}
 \| H_{j,n}(U_n,\partial_j U_n )\|_{L^{s}(\Omega)}\leq C  \end{equation} for a positive constant $C$ independent of $n$. 
We distinguish the following two situations.
	
\medskip
\noindent $(a)$ 
Let $j\in N_{\overrightarrow{\mathfrak{a}}}\cup  N_{\overrightarrow{\mathfrak{a}}}^{\, c}$ (when $\mathfrak{a}_j q_j=0$). We define $s$ as follows
\begin{equation*} \label{slgmhn} 
	s=m'\ \mbox{if } j\in N_{\overrightarrow{\mathfrak{a}}}\ \mbox{ and } s=p'\mbox{ if } j\in N_{\overrightarrow{\mathfrak{a}}}^{\, c}.
	\end{equation*} 
Let $c_j$ be given by
\begin{equation} \label{xij}  c_j:=(\mbox{meas}\,(\Omega))^{1/\lambda_j},\ \mbox{where } 
	\frac{1}{\lambda_j}:=
	1-\frac{\theta_j}{s'}-\frac{q_j}{p_j}.  
\end{equation} 
By H\"older's inequality and Proposition~\ref{red}, we infer that 
\begin{equation*} \label{cob3} 
\| H_{j,n}(U_n,\partial_j U_n )\|_{L^{s}(\Omega)} \leq 
	c_{j}\, \|U_n\|_{L^{s'}(\Omega)}^{\theta_j-1}\|\partial_j U_n\|_{L^{p_j}(\Omega)}^{q_j}\leq C,
\end{equation*} 
where $C$ is a positive constant independent of $n$.

\vspace{0.2cm}
\noindent $(b)$ Let $j\in P_{\overrightarrow{\mathfrak{a}}}\cup P_{\overrightarrow{\mathfrak{a}}}^{\, c}$ (when $\mathfrak{a}_j q_j>0$).  
Let $\mathfrak{I}_{m,p_j}(U_n)$ be as in \eqref{impj}. 
In each of the situations below, we use H\"older's inequality and Proposition~\ref{red} to obtain \eqref{lsmvbf} for suitable $s>1$. 
\begin{itemize}
\item[$(b_1)$]
If $m\geq (\theta_j-1)p_j/q_j$, then by choosing $1<s<p_j/q_j$, we see that 
$$  \| H_{j,n}(U_n,\partial_j U_n )\|_{L^{s}(\Omega)} \leq (\mbox{meas}\,(\Omega))^{\frac{1}{s}-\frac{q_j}{p_j}}
\left( \mathfrak{I}_{m,p_j} (U_n) \right)^{\frac{\theta_j-1}{m}} \|\partial_j U_n\|_{L^{p_j}(\Omega)}^{q_j-\frac{(\theta_j-1)p_j}{m}}. 
$$ 
\item[$(b_2)$] If $\theta_j-1<m<(\theta_j-1)p_j/q_j$, then for $1<s<m/(\theta_j-1)$, we have
$$  \| H_{j,n}(U_n,\partial_j U_n )\|_{L^{s}(\Omega)} \leq (\mbox{meas}\,(\Omega))^{\frac{1}{s}-\frac{\theta_j-1}{m}}
\left(  \mathfrak{I}_{m,p_j} (U_n) \right)^{\frac{q_j}{p_j}} \|U_n\|_{L^{m}(\Omega)}^{\theta_j-1-\frac{q_jm}{p_j}}. 
$$ 
\item[$(b_3)$] If $1<m\leq \theta_j-1$, then we always have
$m>\mathfrak{m}_j$. Indeed, if $j\in P_{\overrightarrow{\mathfrak{a}}}$, then the assumption \eqref{m} gives that 
$m>\min\{\theta_j,\mathfrak{m}_j\}=\mathfrak{m}_j$. If, in turn, $j\in P_{\overrightarrow{\mathfrak{a}}}^{\, c}$, then 
$\mathfrak{m}_j\leq 1<m$. Hence, 
$m>\mathfrak{m}_j$ for $ j\in P_{\overrightarrow{\mathfrak{a}}}\cup P_{\overrightarrow{\mathfrak{a}}}^{\, c}$ leads to   
$$  \| H_{j,n}(U_n,\partial_j U_n )\|_{L^{p'}(\Omega)} \leq (\mbox{meas}\,(\Omega))^{\frac{q_j(m-\mathfrak{m}_j)}{pp_j}}
\left(  \mathfrak{I}_{m,p_j} (U_n) \right)^{\frac{q_j}{p_j}} \|U_n\|_{L^{p}(\Omega)}^{\theta_j-1-\frac{q_jm}{p_j}}. 
$$ Thus, \eqref{lsmvbf} holds with $s=p'$. 
\end{itemize}
This proves \eqref{lsmvbf} for every $1\leq j\leq N$. Then, using \eqref{cob1}, 
	we have, up to a subsequence, 
	\begin{equation*} \label{hau}
H_{j,n}\left(U_n,\partial_jU_n\right) \rightharpoonup   |U_0|^{\theta_j-2} U_0\, |\partial_j U_0|^{q_j}\quad 
	\mbox{(weakly) in } L^{s}(\Omega)\ \mbox{as } n\to \infty,
	\end{equation*} 
where $s>1$ is chosen according to $(a)$, $(b_1)$, $(b_2)$ or $(b_3)$ (for the latter, we take $s=p'$). 
Thus, \eqref{ioc} follows for every $1\leq j\leq N$ and all $v \in W_0^{1,\overrightarrow{p}}(\Omega)\cap L^\infty(\Omega)$.

\subsection{Proof of \eqref{lima} in Case~2} \label{sec-th2}
Let $v$ be any non-negative function in $W_0^{1,\overrightarrow{p}}(\Omega)\cap L^\infty(\Omega)$.   
By Lemma~\ref{pol}, for each $n\geq 1$, we have $U_n\geq 0$ a.e. in $\Omega$ and the same applies to 
$U_0$. Hence, proving \eqref{lima} amounts to showing that  
\begin{equation} \label{barb}  \lim_{n\to \infty}   \int_\Omega \left(\Psi_{n,J_1} (U_n)+ \Psi_{n,J_2} (U_n) \right)v\,dx=\sum_{j=1}^N \int_{\{U_0>0\}} \frac{|\partial_j U_0|^{q_j}}{U_0^{1-\theta_j}}\,v\,dx,
\end{equation}
where $\Psi_{n,J_1} (U_n)$ and $\Psi_{n,J_2} (U_n) $ can be obtained from \eqref{psia} replacing $u$ by $U_n$.  

\noindent From $U_0\in W_0^{1,\overrightarrow{p}}(\Omega)$, it follows that $\nabla U_0=0$ a.e. in $\{U_0=0\}$.
For every $j\in J_1$ we have $\theta_j>1$ so that with the same argument given for Case~1 in Section~\ref{sec-c1}, we find that 
\begin{equation*} \label{neo}
 \lim_{n\to \infty}   \int_\Omega \Psi_{n,J_1} (U_n)\,v\,dx
=
\lim_{n\to \infty}
\sum_{j\in J_1} \int_\Omega H_{j,n}\left(U_n,\partial_jU_n\right) v\,dx= 
\sum_{j\in J_1} \int_{\{U_0>0\}} U_0^{\theta_j-1}  |\partial_j U_0|^{q_j}\,v\,
dx.\end{equation*}   
Hence, using \eqref{lima}, we infer that there exists
$\lim_{n\to \infty} \int_\Omega \Psi_{n,J_2}(U_n)\,v\,dx$. 
To reach \eqref{barb}, it remains to show that
\begin{equation} \label{barbi} 
\lim_{n\to \infty} \int_{\Omega} \Psi_{n,J_2}(U_n)\,v\,dx=\sum_{j\in J_2} \int_{\{U_0>0\}}  \frac{|\partial_j U_0|^{q_j}}{U_0^{1-\theta_j} }\,v\,
dx.\end{equation} 

\noindent To this aim, let us notice that, for every $\sigma>0$, we have
\begin{equation} \label{run} 
\int_{\Omega} \Psi_{n,J_2}(U_n)\,v\,dx=
\int_{\{U_{n} > \sigma\}} \Psi_{n,J_2}(U_{n})\,v\,dx+ 
\int_{\{U_{n} \le \sigma\}} \Psi_{n,J_2}(U_{n})\,v\,dx.
\end{equation}
Fix $\sigma>0$ such that $\sigma\not\in \mathcal E$, where we define 
\begin{equation} \label{sete}
\mathcal{E}:=\{\sigma>0:\, \mbox{meas}\,(\{U_0=\sigma\})>0\}.
\end{equation}
We show that
\begin{equation} \label{comm} 
\begin{aligned}
&(i)\  \lim_{n\to \infty}  \int_{\{U_{n} > \sigma\}} \Psi_{n,J_2}\left(U_{n}\right)\,v\,dx=\sum_{j\in J_2} \int_{\{U_0>\sigma\}} \frac{|\partial_j U_0|^{q_j}}{U_0^{1-\theta_j}}\,v\,dx, \\
&
(ii)\ \lim_{\sigma \to 0} \lim_{n\to\infty}  \int_{\{U_{n} \le \sigma\}} \Psi_{n,J_2}\left(U_{n}\right)\,v\,dx=0.
\end{aligned}
\end{equation}
Assuming that the assertions in \eqref{comm} have been proved, we end the proof of \eqref{barbi} as follows.
We have $\chi_{\{U_0>\sigma_2\}}\leq \chi_{\{U_0>\sigma_1\}}$ for $0<\sigma_1<\sigma_2$, and 
the set $\mathcal{E}$ in \eqref{sete} is at most countable.  
Moreover, from \eqref{run} and \eqref{comm}, we see that
$$ \sum_{j\in J_2}\int_{\{U_0>\sigma\}} \frac{|\partial_j U_0|^{q_j}}{U_0^{1-\theta_j}}\,v\,dx\leq \lim_{n\to \infty} 
\int_{\Omega} \Psi_{n,J_2}(U_n)\,v\,dx<\infty.$$
Hence, by the Monotone Convergence Theorem, we deduce that  
\begin{equation} \label{kk}  
\lim_{\sigma\to 0, \,\sigma\notin \mathcal E}\lim_{n\to \infty}  \int_{\{U_{n} > \sigma\}} \Psi_{n,J_2}\left(U_{n}\right)\,v\,dx 
= \sum_{j\in J_2} \int_{\{U_0>0\}} \frac{|\partial_j U_0|^{q_j}}{U_0^{1-\theta_j}}\,v\,dx<\infty. 
\end{equation}
Using \eqref{comm} and \eqref{kk} in \eqref{run}, we obtain \eqref{barbi}. It remains to show \eqref{comm}. 

\noindent $(i)$ Let $j\in J_2$ be arbitrary. We conclude $(i)$ by proving that 
\begin{equation} \label{nim} \lim_{n\to \infty}  \int_{\{U_{n} > \sigma\}} 
H_{j,n}\left(U_n+\frac{1}{n},\partial_j U_n\right)
v\,dx=\int_{\{U_0>\sigma\}} \frac{|\partial_j U_0|^{q_j}}{U_0^{1-\theta_j}}\,v\,dx.
\end{equation}
For every measurable subset $\omega$ of $\Omega$, we have
$$   \int_{\omega\cap \{U_{n} > \sigma\}} 
H_{j,n}\left(U_n+\frac{1}{n},\partial_j U_n\right)
v\,dx\leq \frac{\|v\|_{L^\infty(\Omega)}}{\sigma^{1-\theta_j}} \|\partial_j U_n\|_{L^{p_j}(\Omega)}^{q_j} \left(\mbox{meas}\,(\omega)\right)^{1-\frac{q_j}{p_j}}.
$$ From Proposition~\ref{red}, using that $\sigma \notin \mathcal{E}$, we obtain that 
$\chi_{\{ U_{n}> \sigma\}} \to\chi_{\{U_0>\sigma\}}$ a.e. in the set $\{U_0 \ne \sigma\}$, as well as  
$$  H_{j,n}\left(U_n+\frac{1}{n},\partial_j U_n\right) \chi_{\{U_n>\sigma\}}\,v \to
\frac{|\partial_j U_0|^{q_j}}{U_0^{1-\theta_j}} \chi_{\{U_0>\sigma\}}\,v\quad \mbox{a.e. in } \Omega\ \mbox{as } n\to \infty. 
$$
By Vitali's Theorem, we conclude the proof of \eqref{nim}. 

\medskip
\noindent $(ii)$ Let $Z_\sigma: [0,\infty)\to[0,1]$ be the following function 
$$
Z_\sigma(s)=\left\{
\begin{array}{ll}
1 & \text{if } 0 \le s \le \sigma,
\\
2-s/\sigma & \text{if } \sigma\le s \le 2\sigma,
\\
0 & \text{if } 2\sigma \le s.
\end{array}
\right.
$$ 
For $w\in W_0^{1,\overrightarrow{p}}(\Omega)$, we define  
\begin{equation} \label{goal}  \mathcal L_{\sigma,v} (w):=\sum_{j=1}^N \int_\Omega A_j(w)\,Z_\sigma(w)\,\partial_j v\, dx+ 
\int_\Omega \Phi(w) \,Z_\sigma(w)\,v\,dx.
\end{equation} 
Observe that $Z_\sigma(U_0) \to\chi_{\{U_0=0\}}$ a.e. in $\Omega$ as $\sigma \to 0$ and $U_0 \in W_0^{1,\overrightarrow{p}}(\Omega)$ implies that $\nabla U_0=0$ a.e. in $\{U_0=0\}$. 
From \eqref{bno}, we have 
$\Phi(x,0,0)=0$ and $A_j(x,0,0)=0$ a.e. in $\Omega$, for every $1 \le j \le N$. It follows that 
 $\mathcal L_{\sigma,v} (U_0)\to 0$ as $\sigma \to 0$. 
Hence, we conclude the assertion of $(ii)$ in \eqref{comm} by showing that
\begin{eqnarray} 
& \displaystyle 0\leq \int_{\{U_{n} \le \sigma\}} \Psi_{n,J_2}\left(U_{n}\right)\,v\,dx\leq \mathcal L_{\sigma,v} (U_n)\quad \mbox{for every } n\geq 1,\label{cobo}\\
& \displaystyle  \lim_{n\to \infty} \mathcal L_{\sigma,v} (U_n) =\mathcal L_{\sigma,v} (U_0).  \label{coba}
\end{eqnarray}

\noindent From \eqref{bno}, we have $\langle\mathfrak{B}U_n,Z_\sigma\left(U_n\right)v\rangle\geq 0$ and $\Theta(U_n)\leq 0$ for every $n\geq 1$. Thus, 
by taking $v\, Z_\sigma(U_n)\ge 0$ as a test function in \eqref{fff2} and using the coercivity condition in \eqref{ellip}, we see that  
\begin{equation} \label{opi}
\mathcal L_{\sigma,v} (U_n)\geq \frac{\nu_0}{\sigma} \displaystyle\sum_{j=1}^N\int_{\{ \sigma<U_n<2\sigma\}} |\partial_j U_n|^{p_j} \,v\,dx+
\displaystyle \int_\Omega Z_\sigma(U_n)\,\Psi_{n}(U_n)\,v\, dx.
\end{equation}
Since  $Z_\sigma(U_n) = 1$ in $\{U_n\leq \sigma\}$, from \eqref{opi}, we derive \eqref{cobo}. 

\noindent Using that $Z_\sigma(U_{n}) \to Z_\sigma(U_0)$ a.e. in $\Omega$ as $n\to \infty$, by Lebesgue's Dominated Convergence Theorem, for each $1\leq j\leq N$, we find that $Z_\sigma(U_{n})\,\partial_j v \to Z_\sigma(U_0)\,\partial_jv$ (strongly) in $L^{p_j}(\Omega)$ as $n\to \infty$. This, jointly with \eqref{jok2}, implies that
$$ \lim_{n\to \infty} \sum_{j=1}^N \int_\Omega 
A_j(U_n)\,Z_\sigma(U_n)\,\partial_j v\, dx =\sum_{j=1}^N \int_\Omega A_j(U_0)\,
Z_\sigma(U_0)\,\partial_j v\,dx. 
$$
Similar to the proof of \eqref{gli}, we have $\Phi(U_n) \,Z_\sigma(U_n)\to \Phi(U_0)\, Z_\sigma(U_0)$ in $L^1(\Omega)$ as $n\to \infty$. Then, using $w=U_n$ in \eqref{goal} and letting $n\to \infty$, we obtain \eqref{coba}. 

The proof of \eqref{barbi}, and hence of \eqref{barb}, is now complete. 

\vspace{0.2cm}
This ends the proof of Theorem~\ref{mainth}. \qed


\section{Proof of Proposition \ref{red}} \label{sec3}

For each $n\geq 1$, the solution $U_n\in  W_0^{1,\overrightarrow{p}}(\Omega)\cap L^m(\Omega)$ of \eqref{e1b} given in Lemma \ref{pol} satisfies 
\begin{equation} \label{e22cv}
	\langle \mathcal A U_n, U_n\rangle -\langle \mathfrak{B}U_n,U_n\rangle  +
	 \int_{\Omega}   
	 \Phi(U_n)\,U_n\,dx =-\int_{\Omega}   
	 \Theta(U_n)\,U_n\,dx+
	\int_\Omega \Psi_n (U_n)\,U_n\,dx.
	\end{equation} 

(a) We prove that there exists a constant $C>0$ such that \eqref{e2} holds for all $n\geq 1$.
	
	We first show that $\{U_n\}_{n\geq 1}$ is bounded in $ W_0^{1,\overrightarrow{p}}(\Omega)$. 
We have assumed that $\mathfrak B$ belongs to the class $\mathfrak{BC} ((1-\varepsilon)\, \mathcal A)$ for some $\varepsilon>0$. 	
Using the coercivity condition in \eqref{ellip}, \eqref{etc} and \eqref{impj}, we find that the left-hand side of \eqref{e22cv} is bounded below by 
\begin{equation}\label{alb} 
\varepsilon \nu_0\,\sum_{k=1}^N \|\partial_k U_n\|^{p_k}_{L^{p_k}(\Omega)}
+\langle [(1-\varepsilon) \mathcal A -\mathfrak B] U_n,U_n\rangle
+\mathfrak{a}_0 \|U_n\|^m_{L^m(\Omega)}+
\sum_{j\in  P_{\overrightarrow{\mathfrak{a}}} \cup P_{\overrightarrow{\mathfrak{a}}}^{\,c}} \mathfrak{a}_j \mathfrak{I}_{m,p_j} (U_n).
  \end{equation}
We observe that $\mathfrak {a}_j>0$ for every $j\in P_{\overrightarrow{\mathfrak{a}}} \cup P_{\overrightarrow{\mathfrak{a}}}^{\,c}$. 
We now consider the right-hand side of \eqref{e22cv}. Using \eqref{cond1} 
and the anisotropic Sobolev inequality \eqref{ASI} in the Appendix, we find a positive constant $C$, independent of $n$,  
such that
$$ \left| \int_{\Omega}   
	 \Theta(U_n)\,U_n\,dx\right|\leq C_\Theta \|U_n\|_{L^1(\Omega)} \leq C \|U_n\|_{ W_0^{1,\overrightarrow{p}}(\Omega)}.
$$
By Young's inequality, for each $\delta>0$, there exists a constant $C_\delta>0$, depending on $\delta$, such that 
\begin{equation} \label{bar}  \left| \int_{\Omega}   
	 \Theta(U_n)\,U_n\,dx\right|\leq  C \sum_{k=1}^N  \|\partial_k U_n\|_{L^{p_k}(\Omega)}
	 \leq \delta \sum_{k=1}^N \|\partial_k U_n\|^{p_k}_{L^{p_k}(\Omega)}+C_\delta\quad \mbox{for all } n\geq 1.
\end{equation}
By \eqref{impj}, we have
\begin{equation} \label{nuu}   
  \int_\Omega \Psi_n (U_n)\,U_n\,dx  \leq \int_\Omega \Psi (U_n)\,U_n\,dx
\leq
\sum_{j=1}^N 
\mathfrak{I}_{\theta_j,q_j}(U_n).
\end{equation}
In Lemma~\ref{intorj} below, we obtain a suitable upper bound for $\sum_{j=1}^N 
\mathfrak{I}_{\theta_j,q_j}(U_n)$. To this end, 
we need to distinguish the case $m\geq \theta_j p_j/q_j$ from $m<\theta_jp_j/q_j$ whenever 
$j\in P_{\overrightarrow{\mathfrak{a}}}\cup P^{\,c}_{\overrightarrow{\mathfrak{a}}}$. We observe that  $P_{\overrightarrow{\mathfrak{a}}}\cup P^{\,c}_{\overrightarrow{\mathfrak{a}}}=\{ 1\leq j\leq N: \ \mathfrak{a}_j q_j>0\}$ is 
a union of three sets:
\begin{equation} \label{univ2}  P_{\overrightarrow{\mathfrak{a}}}\cup P^{\,c}_{\overrightarrow{\mathfrak{a}}}=
\widehat{P}_{\overrightarrow{\mathfrak{a}},1}\cup \widehat{P}_{\overrightarrow{\mathfrak{a}},2}\cup  P_{\overrightarrow{\mathfrak{a}},3},
\end{equation}
where we define 
\begin{equation} \label{mii} 
\begin{aligned}
& \widehat{P}_{\overrightarrow{\mathfrak{a}},1}:=\left\{ j\in P_{\overrightarrow{\mathfrak{a}}}\cup P^{\,c}_{\overrightarrow{\mathfrak{a}}}:\  
m\geq  \theta_j p_j/q_j \right\},\\
&  \widehat{P}_{\overrightarrow{\mathfrak{a}},2}:=\left\{ j\in P_{\overrightarrow{\mathfrak{a}}}:\ \theta_j<p ,\quad 
m<\theta_jp_j/q_j \right\}\cup \left\{ j\in P^{\,c}_{\overrightarrow{\mathfrak{a}}}:\ m<  \theta_j p_j/q_j \right\}, \\
& P_{\overrightarrow{\mathfrak{a}},3}:=\left\{ j\in P_{\overrightarrow{\mathfrak{a}}}:\ \theta_j\geq p,\quad 
m<\theta_jp_j/q_j \right\}. 
\end{aligned} \end{equation}

\begin{lemma} \label{intorj} For any $\delta>0$, there exists a positive constant $C_\delta$ such that, for every $n\geq 1$,
	 \begin{equation} \label{har}
		\sum_{j=1}^N 
	\mathfrak{I}_{\theta_j,q_j}(U_n)
	   \leq 
	N \delta  \|U_n\|_{L^m(\Omega)}^m 
	+ \delta\sum_{j\in P_{\overrightarrow{\mathfrak{a}}} \cup P_{\overrightarrow{\mathfrak{a}}}^{\,c}} \mathfrak{I}_{m,p_j}(U_n)
	 +(1+N) \delta\sum_{k=1}^N  \|\partial_k U_n\|^{p_k}_{L^{p_k}(\Omega)}
	+C_\delta.
 \end{equation}
	\end{lemma}
	
	\begin{proof}
For the inequalities in \eqref{har0}, \eqref{haric}, \eqref{har1}--\eqref{hari2} below, we use H\"older's inequality, then Young's inequality (see Lemma~\ref{young} in the Appendix). 
In what follows, we understand that $\delta>0$ is arbitrary and $C_\delta>0$ is a suitable constant depending on $\delta$.  

\vspace{0.2cm}
(I) We first estimate $ \mathfrak{I}_{\theta_j,q_j}(U_n)$ for every 
$j\in  N_{\overrightarrow{\mathfrak{a}}}\cup N_{\overrightarrow{\mathfrak{a}}}^{\,c}$ when we let $c_j$ and $\lambda_j$ be as in  \eqref{xij}. 

\noindent $\bullet$ 
Let $j\in N_{\overrightarrow{\mathfrak{a}}}$. 
Condition \eqref{m} gives that $\lambda_j>1$ so that 
\begin{equation}\label{har0}
\mathfrak{I}_{\theta_j,q_j}(U_n)
 \le 
c_j \|U_n\|_{L^m(\Omega)}^{\theta_j} \|\partial_j U_n\|^{q_j}_{L^{p_j}(\Omega)}
 \leq \delta 
\|U_n\|_{L^m(\Omega)}^m 
+\delta
\| \partial_j U_n\|_{L^{p_j}(\Omega)}^{p_j}+C_\delta.
\end{equation}
 
\noindent $\bullet$ Let $j\in  N_{\overrightarrow{\mathfrak{a}}}^{\,c}$. Using Lemma~\ref{young} and the anisotropic Sobolev inequality \eqref{ASI} in the Appendix, we find a positive constant $C$, depending on $N$, $\overrightarrow{p}$, $q_j$, $\theta_j$ and 
$\mbox{meas}\,(\Omega)$, such that
\begin{equation}\label{haric}
 \begin{aligned}
 \mathfrak{I}_{\theta_j,q_j}(U_n)
 &\le 
c_j\,
  \|U_n\|_{L^{p}(\Omega)}^{\theta_j} \|\partial_j U_n\|^{q_j}_{L^{p_j}(\Omega)}
 \leq C \|\partial_j U_n\|^{\frac{\theta_j}{N}+q_j}_{L^{p_j}(\Omega)}
\prod_{k\in \{1,\ldots,N\}\setminus\{j\}} \|\partial_k U_n\|^{\frac{\theta_j}{N}}_{L^{p_k}(\Omega)}
 \\
 &\leq \delta \sum_{k=1}^N  \|\partial_k U_n\|^{p_k}_{L^{p_k}(\Omega)}
+C_\delta.
 \end{aligned}
\end{equation}

(II) We now estimate $ \mathfrak{I}_{\theta_j,q_j}(U_n)$ for every $j\in P_{\overrightarrow{\mathfrak{a}}}\cup P^{\,c}_{\overrightarrow{\mathfrak{a}}}$ when we define 
\begin{equation} \label{notxi}
c_j:=\left({\rm meas}\, (\Omega)\right)^{1/\lambda_j}\quad \mbox{and}\quad 
\frac{1}{\lambda_j}:=\left\{
\begin{aligned}
& 1-\frac{q_j}{p_j} && \mbox{if } j\in \widehat{P}_{\overrightarrow{\mathfrak{a}},1},&&\\
& 1-\frac{\theta_j}{m} && \mbox{if } j\in P_{\overrightarrow{\mathfrak{a}},3},&&\\
& \frac{q_j(m-\mathfrak{m}_j)}{p_jp} && \mbox{if } j\in 
 \widehat{P}_{\overrightarrow{\mathfrak{a}},2}.
\end{aligned}
\right.
\end{equation}

Recall that $P_{\overrightarrow{\mathfrak{a}}}:=\{1\leq j\leq N:\ 
\mathfrak{a}_jq_j>0, \  \mathfrak{m}_j> 1\}$. Condition \eqref{m} implies that $m>\min\{\theta_j,\mathfrak{m}_j\}$ whenever
$j\in P_{\overrightarrow{\mathfrak{a}}}$ and, moreover, 
$ \min\{\theta_j,\mathfrak{m}_j\}=\theta_j$ if and only if $ \theta_j\geq p$.

\noindent 
$\bullet$ For every $j\in \widehat{P}_{\overrightarrow{\mathfrak{a}},1}$, 
we obtain that 
 \begin{equation}\label{har1}
 \mathfrak{I}_{\theta_j,q_j}(U_n)
\le 
 c_j\, \|\partial_j U_n\|_{L^{p_j}(\Omega)}^{q_j-\frac{p_j\theta_j}{m}} 
\left( \mathfrak{I}_{m,p_j}(U_n)
\right)^{\frac{\theta_j}{m}}\\
\leq \delta \, \mathfrak{I}_{m,p_j}(U_n)
 +\delta \| \partial_j U_n\|_{L^{p_j}(\Omega)}^{p_j}+C_\delta.
\end{equation}

\noindent  $\bullet$ Let $j\in P_{\overrightarrow{\mathfrak{a}},3}$. In this case, we have $m>\theta_j$ so that 
\begin{equation}\label{har2}
\mathfrak{I}_{\theta_j,q_j}(U_n)
\le 
c_j\, \| U_n\|_{L^{m}(\Omega)}^{\theta_j-\frac{m q_j}{p_j}} 
(\mathfrak{I}_{m,p_j}(U_n)
)^{\frac{q_j}{p_j}}\\
\leq \delta \,\mathfrak{I}_{m,p_j}(U_n)
+\delta 
\| U_n\|_{L^{m}(\Omega)}^{m}+C_\delta.
\end{equation}

\noindent $\bullet$ Let $j\in \widehat{P}_{\overrightarrow{\mathfrak{a}},2}$. Then $m>\mathfrak{m}_j$. 
By H\"older's inequality, Lemma~\ref{young} and the anisotropic Sobolev inequality \eqref{ASI} in the Appendix, we find a positive constant $C=C(N,\overrightarrow{p},q_j,\theta_j,m,  
\mbox{meas}\,(\Omega))$ such that 
\begin{equation}\label{hari2}
\begin{aligned}
\mathfrak{I}_{\theta_j,q_j}(U_n)
&\le 
 c_j\, \| U_n\|_{L^{p}(\Omega)}^{\theta_j-\frac{m q_j}{p_j}} 
\left(  \mathfrak{I}_{m,p_j}(U_n)
\right)^{\frac{q_j}{p_j}}\le C \left(  \mathfrak{I}_{m,p_j}(U_n)
\right)^{\frac{q_j}{p_j}} \prod_{k=1}^N 
\|\partial_k U_n\|_{L^{p_k}(\Omega)}^{ \left( \theta_j-\frac{mq_j}{p_j}\right)\frac{1}{N}}
\\
&\leq \delta \, \mathfrak{I}_{m,p_j}(U_n)
+ \delta \sum_{k=1}^N \|\partial_k U_n\|_{L^{p_k}(\Omega)}^{p_k}
+C_\delta.
\end{aligned}
\end{equation}

\noindent By adding the inequalities in \eqref{har0}, \eqref{haric}, \eqref{har1}--\eqref{hari2}, we complete the proof of \eqref{har}. \end{proof}

{\bf Proof of Proposition~\ref{red} completed.}  
From \eqref{etc} and the definition of $P_{\overrightarrow{\mathfrak{a}}}$ and $ P^{\,c}_{\overrightarrow{\mathfrak{a}}}$, we have $\mathfrak a_0>0$ and $\min_{j\in P_{\overrightarrow{\mathfrak{a}}}\cup P^{\,c}_{\overrightarrow{\mathfrak{a}}}}\mathfrak a_j>0$. We choose $\delta>0$ small such that
\begin{equation} \label{smal}
\varepsilon \nu_0>(N+2)\delta,\quad \mathfrak a_0>N\delta \quad \mbox{and}\quad  
\min_{j\in P_{\overrightarrow{\mathfrak{a}}}\cup P^{\,c}_{\overrightarrow{\mathfrak{a}}}}\mathfrak a_j>\delta.
\end{equation}
 
By \eqref{bar}, \eqref{nuu} and Lemma~\ref{intorj}, there exists a positive constant $C_\delta$ such that for each $n\geq 1$, 
the right-hand side of 
\eqref{e22cv} is bounded above by 
\begin{equation} \label{fum}
	 (N+2)\delta \sum_{k=1}^N  \|\partial_k U_n\|^{p_k}_{L^{p_k}(\Omega)}+ N \delta \|U_n\|_{L^m(\Omega)}^m +
	 \delta \sum_{j\in P_{\overrightarrow{\mathfrak{a}}} \cup P_{\overrightarrow{\mathfrak{a}}}^{\,c}} \mathfrak{I}_{m,p_j}(U_n) + C_\delta.
\end{equation}

For ease of reference, we introduce $\mathfrak S_n$ as follows
$$  \mathfrak S_n:=\|U_n\|^m_{L^m(\Omega)}+\sum_{j\in  P_{\overrightarrow{\mathfrak{a}}} \cup P_{\overrightarrow{\mathfrak{a}}}^{\,c}}  \mathfrak{I}_{m,p_j} (U_n)\geq 0.
$$ 
In view of \eqref{e22cv}, the quantity in \eqref{alb} is bounded above by that in \eqref{fum}. 
Hence, using the inequalities in \eqref{smal}, we infer that for some small constant $\varepsilon_1>0$, we have
\begin{equation} \label{opm}   \varepsilon_1 \left(\sum_{k=1}^N \|\partial_k U_n\|^{p_k}_{L^{p_k}(\Omega)}+\mathfrak S_n\right)
+\langle \left[\left(1-\varepsilon\right) \mathcal A-\mathfrak B\right] U_n,U_n\rangle
\leq C_{\delta}
\end{equation} for every $n\geq 1$. Now, from the hypothesis that 
$\mathfrak B$ belongs to the class $\mathfrak{BC}((1-\varepsilon)\mathcal A)$ with $\varepsilon\in (0,1)$, we have the coercivity of the operator 
$(1-\varepsilon) \mathcal A-\mathfrak B$ from $W_0^{1,\overrightarrow{p}}(\Omega)$ into $W^{-1,\overrightarrow{p}'}(\Omega)$. Hence, \eqref{opm} implies that 
$\{U_n\}_{n\geq 1}$ is bounded in $ W_0^{1,\overrightarrow{p}}(\Omega)$. Since $\mathfrak B:W_0^{1,\overrightarrow{p}}(\Omega)\to W^{-1,\overrightarrow{p}'}(\Omega)$ is bounded, we find a constant $ C>0$ such that  
$|\langle \mathfrak B U_n,U_n\rangle |\leq C$ for every $n\geq 1$. Using also the coercivity assumption in \eqref{ellip}, the inequality in \eqref{opm} gives the boundedness of 
$\{\mathfrak S_n\}_{n\geq 1}$. 
Using this fact into \eqref{nuu} and \eqref{har}, we conclude from \eqref{nuu} that 
$$  0\leq \int_\Omega \Psi_n (U_n)\,U_n\,dx  \leq
 \int_\Omega \Psi(U_n)\,U_n\, dx\leq C, 
$$ 
where $C>0$ is a constant independent of $n\geq 1$. Returning to \eqref{e22cv} and using \eqref{bar}, we obtain that the sequence of positive functions 
$\{\Phi(U_n) \,U_n\}_{n\geq 1}$ is bounded in $L^1(\Omega)$. 
The proof of  \eqref{e2} is now complete. \qed

\vspace{0.2cm}
(b) From  \eqref{e2}, there exists a function $U_0 \in W_0^{1,\overrightarrow{p}}(\Omega)\cap L^m(\Omega)$ such that, up to a subsequence, 
 \eqref{uni} holds. This completes the proof of Proposition~\ref{red}. \qed

\begin{rem} \label{rem32} From \eqref{uni}, we have  
 $ U_n^\pm \rightharpoonup U_0^\pm$ (weakly) in $ W_0^{1,\overrightarrow{p}}(\Omega)$ as $n\to \infty$, which yields that  
 $ \lim_{n\to \infty}
 \langle \mathcal A U_0^\pm, U_n^\pm -U_0^\pm \rangle = 0$.  
\end{rem}

\section{Applications of Proposition~\ref{red}} \label{sec5}

Throughout this section, the assumptions of Proposition~\ref{red} hold. For each $n\ge 1$ let $U_n$ be the solution of \eqref{e1b} provided by Lemma \ref{pol}.

\begin{lemma} \label{aplu} 
Let $\omega$ be a measurable subset of  $\Omega$. Assume that $\{V_n\}_{n\geq 1}$ is a sequence in $W_0^{1,\overrightarrow{p}}(\Omega)\cap L^m(\Omega)$ satisfying 
$|V_n|\leq |U_n|$ on $\omega$ for all $n\geq 1$. Then, for every $\tau\in (0,1)$ small enough and $\beta\in (1/\tau,m/\tau)$ fixed, there exists a positive constant
$C$, independent of $\omega$, such that 
\begin{equation} \label{fly} 
\sum_{j=1}^N \int_\omega |U_n|^{\theta_j-1} |\partial_j U_n|^{q_j} |V_n|\,dx\leq C\left( \|V_n\|_{L^p(\Omega)}^\tau+\|V_n\|^\tau_{L^{\tau\beta}(\Omega)}\right)\quad \mbox{for all } n\geq 1. 
\end{equation}
\end{lemma}

\begin{proof} 
Fix $\tau$ small satisfying $0<\tau<\min\{m-1, \min_{1\leq j\leq N}\{\theta_j\},1\}$. 
Since $|V_n|\leq |U_n|$ on $\omega$ for all $n\geq 1$, 
we have   
$|U_n|^{\tau-1}\leq |V_n|^{\tau-1}$ on $\omega$ so that 
\begin{equation} \label{thet}  \int_{\omega} |U_n|^{\theta_j-1} |\partial_j U_n|^{q_j} |V_n| \,dx\leq 
\int_\Omega |U_n|^{\theta_j-\tau} |\partial_j U_n|^{q_j} |V_n|^\tau \,dx
\end{equation}
for every $1\leq j\leq N$. 
Recall from \eqref{npc}, \eqref{univ}, and \eqref{univ2} that 
$$ \{1\leq j\leq N\}= N_{\overrightarrow{\mathfrak{a}}}\cup N_{\overrightarrow{\mathfrak{a}}}^c\cup
\widehat{P}_{\overrightarrow{\mathfrak{a}},1}\cup \widehat{P}_{\overrightarrow{\mathfrak{a}},2} \cup  P_{\overrightarrow{\mathfrak{a}},3}.$$

\noindent By H\"older's inequality, with $c_j$ given \eqref{xij}, for every $j\in N_{\overrightarrow{\mathfrak{a}}}^{\,c}$, we have 
\begin{equation} \label{fj0}  \int_\Omega |U_n|^{\theta_j-\tau} |\partial_j U_n|^{q_j} |V_n|^\tau \,dx\leq 
 c_j\, \|U_n\|_{L^p(\Omega)}^{\theta_j-\tau} \|\partial_j U_n\|_{L^{p_j}(\Omega)}^{q_j} \|V_n\|_{L^p(\Omega)}^\tau.  
\end{equation}
By the definition of $\widehat{P}_{\overrightarrow{\mathfrak{a}},2}$ in \eqref{mii}, we can
take $\tau$ small such that 
$$ 0<\tau< \theta_j-\frac{mq_j}{p_j}\quad \mbox{for every } j\in \widehat{P}_{\overrightarrow{\mathfrak{a}},2}.$$ 
Using $c_j$ given by \eqref{notxi}, for every $j\in \widehat{P}_{\overrightarrow{\mathfrak{a}},2}$, we derive that 
\begin{equation} \label{fj4} 
 \int_\Omega |U_n|^{\theta_j-\tau} |\partial_j U_n|^{q_j} |V_n|^\tau \,dx\leq  
c_j \left(\mathfrak{I}_{m,p_j}\right)^{\frac{q_j}{p_j}} \|U_n\|_{L^p(\Omega)}^{\theta_j-\tau-\frac{mq_j}{p_j}} \|V_n\|_{L^p(\Omega)}^\tau.
\end{equation}
We fix $\beta\in (1/\tau,m/\tau)$. 
From \eqref{m}, we have $\lambda_j>1$ for every $j\in N_{\overrightarrow{\mathfrak{a}}} $, where $\lambda_j$ is given by \eqref{xij}. 
We choose $\tau>0$ small such that 
$(m-1)\,\tau <m/\lambda_j$ for every 
$j\in N_{\overrightarrow{\mathfrak{a}}}$, which implies that $\lambda_j m/(m+\tau \lambda_j)<\beta$. Hence, 
for every
$j\in N_{\overrightarrow{\mathfrak{a}}}$, by defining  $c_{j,N_{\overrightarrow{\mathfrak{a}}}}=\left(\mbox{meas}\,(\Omega)
\right)^{\frac{1}{\lambda_j}+\frac{\tau}{m}-\frac{1}{\beta}}$, we obtain that
\begin{equation} \label{fj2}  
 \int_\Omega |U_n|^{\theta_j-\tau} |\partial_j U_n|^{q_j} |V_n|^\tau \,dx
\leq  
c_{j,N_{\overrightarrow{\mathfrak{a}}}}\,
\|U_n\|_{L^m(\Omega)}^{\theta_j-\tau}  \|\partial_j U_n\|_{L^{p_j}(\Omega)}^{q_j} \|V_n\|_{L^{\tau \beta}(\Omega)}^\tau. 
\end{equation}
We diminish $\tau$ such that $0<\tau<(p_j-q_j)/p_j$ for every
$j\in \widehat{P}_{\overrightarrow{\mathfrak{a}},1}$. 
Using that $m\geq p_j\theta_j/q_j$ for every $ j\in \widehat{P}_{\overrightarrow{\mathfrak{a}},1}$, 
by H\"older's inequality, we infer that
\begin{equation} \label{fj3} 
 \int_\Omega |U_n|^{\theta_j-\tau} |\partial_j U_n|^{q_j} |V_n|^\tau \,dx
\leq c_{j,\widehat{P}_{\overrightarrow{\mathfrak{a}},1}} \left(\mathfrak{I}_{m,p_j}\right)^{\frac{\theta_j-\tau}{m}} 
\|\partial_j U_n\|^{q_j-\frac{p_j(\theta_j-\tau)}{m}}_{L^{p_j}(\Omega)}
\|V_n\|_{L^{\tau\beta}(\Omega)}^\tau
\end{equation}
for every $j\in \widehat{P}_{\overrightarrow{\mathfrak{a}},1}$, where 
$c_{j,\widehat{P}_{\overrightarrow{\mathfrak{a}},1}}=
\left(\mbox{meas}\,(\Omega)\right)^{\frac{p_j-q_j}{p_j}-\frac{1}{\beta}}$. 

\noindent Finally, for every $j\in P_{\overrightarrow{\mathfrak{a}},3}$, we have
$p\leq \theta_j<m<\theta_jp_j/q_j$ in view of \eqref{m}.
We let $\tau>0$ small such that $\tau<(m-\theta_j)/(m-1)$ for every
$j\in P_{\overrightarrow{\mathfrak{a}},3}$. 
Then, H\"older's inequality yields that
\begin{equation} \label{fj5}
 \int_\Omega |U_n|^{\theta_j-\tau} |\partial_j U_n|^{q_j} |V_n|^\tau \,dx
\leq
c_{j,P_{\overrightarrow{\mathfrak{a}},3}} 
\left(\mathfrak{I}_{m,p_j}\right)^{\frac{q_j}{p_j}}
\|U_n\|_{L^m(\Omega)}^{\theta_j-\tau-\frac{mq_j}{p_j}} \|V_n\|_{L^{\tau\beta }(\Omega)}^{\tau}
\end{equation}
for every $ j\in P_{\overrightarrow{\mathfrak{a}},3}$, 
where we define $c_{j,P_{\overrightarrow{\mathfrak{a}},3}}:= 
\left(\mbox{meas}\,(\Omega)\right)^{\frac{m-\theta_j+\tau}{m}-
	\frac{1}{\beta}}$.

From \eqref{thet}--\eqref{fj5}, jointly with the {\it a priori} estimates in   \eqref{e2},
we derive \eqref{fly}.  \end{proof}

\vskip 0.2cm 

We remark that
\begin{equation}\label{www}
\begin{aligned}
&G_k(U_n) \rightharpoonup G_k(U_0) \quad \mbox{(weakly) in
	  } W_0^{1,\overrightarrow{p}}(\Omega) \ \mbox{and in } L^m(\Omega) \quad \mbox{as } n \to \infty
	  \\
& ||G_k(U_n)||_{L^r(\Omega)} \to ||G_k(U_0)||_{L^r(\Omega)} \quad 	 \mbox{as } n \to \infty, \quad \mbox{where }1\le r<m,
\end{aligned} 
\end{equation}
and
 \begin{equation} \label{gun} G_k(U_0^+)\rightharpoonup 0\quad \mbox{(weakly) in
	  } W_0^{1,\overrightarrow{p}}(\Omega)\ \ \mbox{as } k\to \infty.\end{equation} 
	  For every $k\geq 1$, we define 
\begin{equation}\label{cuore}
 z_{n,k}:=U_n^+ -T_k (U_0^+).
\end{equation}
In the proof of Lemma~\ref{step3} below, we need several properties of $\{z_{n,k}^\pm\}_n$, which we summarise next. 
 \subsection*{Properties of $\{z_{n,k}^\pm\}_n$}
From \eqref{e2} and \eqref{cuore}, we see that $\{z_{n,k}^\pm\}_n$ is bounded in $W_0^{1,\overrightarrow{p}}(\Omega)$ and also in $L^m(\Omega)$ and, up to a subsequence, 
\begin{equation}\label{us}
\begin{aligned}
	& z_{n,k}^+\to (U_0^+-T_k(U_0^+))^+=G_k(U_0^+)\ \mbox{a.e. in } \Omega\ \mbox{as } n \to \infty,\\
	& z_{n,k}^-\to (U_0^+-T_k(U_0^+))^-=0 \ \mbox{a.e. in } \Omega\ \mbox{as } n \to \infty.
\end{aligned}
\end{equation}  
Hence, up to a subsequence, using also Remark~\ref{an-sob} in the Appendix, as $n\to \infty$, we have  
\begin{equation}\label{usi} 
\begin{aligned} 
	& z_{n,k}^+\rightharpoonup G_k(U_0^+)\ \mbox{(weakly) in } W_0^{1,\overrightarrow{p}}(\Omega)\ \mbox{and in } L^m(\Omega),\\
	& z_{n,k}^+\to  G_k(U_0^+)\ \mbox{(strongly) in } L^{p}(\Omega),\\
	&  z_{n,k}^-\rightharpoonup 0\ \mbox{(weakly) in } W_0^{1,\overrightarrow{p}}(\Omega)\ \mbox{and in } L^m(\Omega),\\
	&  z_{n,k}^-\to 0 \ \mbox{(strongly) in } L^{p}(\Omega).
\end{aligned} 
\end{equation}
From \eqref{gun} and \eqref{usi}, by passing to a subsequence, we deduce that 
\begin{equation} \label{nec2} 
\begin{aligned}
& \lim_{n\to \infty} \| z_{n,k}^+\|_{L^{p}(\Omega)}= \| G_k(U_0^+)\|_{L^{p}(\Omega)}\to 0\quad \text{as } k\to \infty,\\
&  \lim_{n\to \infty} \| z_{n,k}^-\|_{L^{p}(\Omega)}=0. 
\end{aligned}
\end{equation}
Let $r\in (1,m)$ be arbitrary. By Vitali's Theorem and \eqref{us}, up to a subsequence, we get that 
\begin{equation} \label{nec1} 
\begin{aligned}
&  \lim_{n\to \infty} \| z_{n,k}^+\|_{L^{r}(\Omega)}= \| G_k(U_0^+)\|_{L^{r}(\Omega)}\to 0\quad  \text{as } k\to \infty,\\
&  \lim_{n\to \infty} \| z_{n,k}^-\|_{L^{r}(\Omega)}= 0.
 \end{aligned}
\end{equation}
Since $\mathfrak{B}$ satisfies the property $(P_2)$, from  \eqref{usi} we have, up to a subsequence,
\begin{equation*}\label{rio2}
\lim_{n \to \infty}\langle \mathfrak{B} U_n, z_{n,k}^+\rangle  
= \langle \mathfrak{B}  U_0, G_k(U_0^+)\rangle\ \ \mbox{and }\  \lim_{n \to \infty} \langle \mathfrak{B}  U_n, z_{n,k}^-\rangle = 0. 
\end{equation*}

By applying Lemma~\ref{aplu}, we obtain Lemma~\ref{step3} to be used in the proof of Proposition~\ref{red2}. 

\begin{lemma} \label{step3} There exist $\{W_k\}_{k\geq 1}$ and $\{Z_k\}_{k\geq 1}$ with $\lim_{k\to \infty} W_k=\lim_{k\to \infty} Z_k=0$ such that, up to a subsequence of $\{U_n\}$, we have 
for each $k\geq 1$
\begin{eqnarray} 
&\limsup_{n\to \infty} \|G_k(U_n)\|_{W_0^{1,\overrightarrow{p}}(\Omega)} \leq W_k,\label{woa} \\
&\limsup_{n\to \infty} \langle \mathcal A U_n, z_{n,k}^+\rangle \leq Z_k  . \label{foi}
\end{eqnarray}
\end{lemma}

\begin{proof} By a well-known diagonal argument, it suffices to show that for every $k\geq 1$, there exists a subsequence of $\{U_n\}$ such that \eqref{woa} and \eqref{foi} hold. 
Let $k\geq 1$ be arbitrary. 

We prove \eqref{woa}. 
Since $G_k(U_n)=U_n-T_k(U_n)$ and
$\partial_j T_k(U_n)=\partial_j U_n \,\chi_{\{|U_n|\le k\}}$ for $1\leq j\leq N$, 
by the coercivity assumption in \eqref{ellip}, we have
\begin{equation} \label{fus1} \begin{aligned}
\langle \mathcal A U_n,G_k(U_n)\rangle&= \sum_{j=1}^N 
\int_{\{|U_n|>k\}} A_j(U_n) \,\partial_j U_n\,dx\\
&\geq 
\nu_0 \sum_{j=1}^N \int_{\{|U_n|>k\}} |\partial_j U_n|^{p_j}\,dx=
\nu_0 \sum_{j=1}^N \|\partial_j G_k(U_n)\|^{p_j}_{L^{p_j}(\Omega)}.
\end{aligned}\end{equation}
Since $t \,G_k(t)\geq 0$ for every $t\in \mathbb R$, by the sign-condition in \eqref{cond1}, we find that $G_k(U_n)\,\Phi(U_n)\geq 0$ for all $n\geq 1$. Then, 
by Lemma~\ref{pol}, we can test 
\eqref{fff2} with $v=G_k(U_n)$ and using \eqref{cond1}, we get
\begin{equation} \label{fus2}  \begin{aligned}
\langle \mathcal A U_n,G_k(U_n)\rangle
&\leq 
\langle \mathcal A U_n,G_k(U_n)\rangle+\int_\Omega G_k(U_n)\,\Phi(U_n)\,dx
\\
&\leq \int_\Omega \Psi_n (U_n)\,G_k(U_n)\,dx+
 | \langle \mathfrak B U_n,G_k(U_n)\rangle |+
C_\Theta \int_\Omega |G_k(U_n)|\,dx.
\end{aligned}\end{equation}
Since $\mathfrak{B}$ satisfies the property $(P_2)$, using \eqref{www} 
we infer that 
\begin{equation} \label{tov}  | \langle \mathfrak B U_n,G_k(U_n)\rangle |+
C_\Theta \int_\Omega |G_k(U_n)|\,dx\to 
 | \langle \mathfrak B U_0,G_k(U_0)\rangle |+
C_\Theta \|G_k(U_0)\|_{L^1(\Omega)}\  \mbox{as } n\to \infty. 
\end{equation}
By \eqref{psi}, we see that 
\begin{equation} \label{gaz}
 \int_\Omega \Psi_n (U_n)\,G_k(U_n)\,dx 
\leq \sum_{j=1}^N 
\int_{\{|U_n|>k\} } |U_n|^{\theta_j-1} |\partial_j U_n|^{q_j} |G_k(U_n)| \,dx.
\end{equation}	
Observe that $0<|G_k(U_n)|\leq |U_n|$ on $\{|U_n|>k\}$. Hence, by Lemma~\ref{aplu}, 
for small $\tau>0$ and
$\beta\in (1/\tau,m/\tau)$ fixed, there exists a  
positive constant $C$, independent of $n$ and $k$, such that
\begin{equation} \label{seab}
\sum_{j=1}^N 
\int_{\{|U_n|>k\} } |U_n|^{\theta_j-1} |\partial_j U_n|^{q_j} |G_k(U_n)| \,dx
\leq 
C \left(  \|G_k(U_n)\|^\tau_{L^p(\Omega)}+
\|G_k(U_n)\|^\tau_{L^{\tau\beta}(\Omega)}
\right).
\end{equation}
From \eqref{gaz} and \eqref{seab}, using \eqref{www} it follows that 
\begin{equation*} \label{mik}
\limsup_{n\to \infty}  \int_\Omega \Psi_n (U_n)\,G_k(U_n)\,dx 
\leq 
C \left(  \|G_k(U_0)\|^\tau_{L^p(\Omega)}+
\|G_k(U_0)\|^\tau_{L^{\tau\beta}(\Omega)}
\right):= \mathfrak R_k.
\end{equation*}
Using this fact, jointly with \eqref{fus1}, \eqref{fus2} and \eqref{tov}, we arrive at 
\begin{equation} \label{lav} 
\nu_0 \limsup_{n\to \infty} \sum_{j=1}^N \|\partial_j G_k(U_n)\|^{p_j}_{L^{p_j}(\Omega)}\leq \mathfrak R_k+ | \langle \mathfrak B U_0,G_k(U_0)\rangle |+
C_\Theta \|G_k(U_0)\|_{L^1(\Omega)}.
\end{equation}
Since $G_k(U_0)
\rightharpoonup 0$ (weakly) in $W_0^{1,\overrightarrow{p}}(\Omega)$ and in $L^m(\Omega)$ as $k \to \infty$, using that $\tau\beta\in (1,m)$, we find that 
 (up to a subsequence), $G_k(U_0)\to 0$ (strongly) in $L^{p}(\Omega)$ and in $L^{\tau\beta}(\Omega)$ as $k\to \infty$. Hence, $\lim_{k\to \infty} \mathfrak R_k=0$ and, moreover, the right-hand side of \eqref{lav} converges to $0$ as $k\to \infty$. The proof of \eqref{woa} is complete. 
 
\vspace{0.2cm}
We now establish \eqref{foi}. 
Let $\ell>0$ be arbitrary. 
We take $v=T_\ell(z_{n,k}^+) \in W_0^{1,\overrightarrow{p}}(\Omega)\cap L^\infty(\Omega)$ as a test function in \eqref{fff2} and, proceeding as in the proof of Lemma~\ref{pol}, by letting $\ell\to \infty$, we get that \eqref{fff2} holds for $v=z_{n,k}^+ \in W_0^{1,\overrightarrow{p}}(\Omega)\cap L^m(\Omega)$. 
This, jointly with \eqref{cond1}, implies that
\begin{equation} \label{fff2b}
	\langle \mathcal A U_n, z_{n,k}^+\rangle 
	\leq  \int_\Omega \Psi_n (U_n)\,z_{n,k}^+\,dx  +|\langle \mathfrak{B} U_n,z_{n,k}^+\rangle|+C_\Theta \int_\Omega z_{n,k}^+\,dx. 
	\end{equation}  
From the definition of $\Psi_n$ in \eqref{psi}, we have 
\begin{equation} \label{psii}
 \int_\Omega \Psi_n (U_n)\,z_{n,k}^+\,dx 
\leq \sum_{j=1}^N 
\int_{\{z_{n,k}^+>0\} } |U_n|^{\theta_j-1} |\partial_j U_n|^{q_j} z_{n,k}^+ \,dx.
\end{equation}	
Observe that $z_{n,k}^+\leq U_n$ on $\{z_{n,k}^+>0\}$.  
Then, from Lemma~\ref{aplu}, for sufficiently small $\tau>0$ and
$\beta\in (1/\tau,m/\tau)$ fixed, there exists a  
positive constant $C$, independent of $n$ and $k$, such that
\begin{equation} \label{sea1}
\sum_{j=1}^N 
\int_{\{z_{n,k}^+>0\} } |U_n|^{\theta_j-1} |\partial_j U_n|^{q_j} z_{n,k}^+ \,dx
\leq 
C \left( \|z_{n,k}^+\|^\tau_{L^p(\Omega)}+ 
\|z_{n,k}^+\|^\tau_{L^{\tau\beta}(\Omega)}
\right).
\end{equation}
By using \eqref{fff2b}, \eqref{psii}, \eqref{sea1}, \eqref{nec2} and \eqref{nec1}, we conclude \eqref{foi} with $Z_k$ given by 
	$$ Z_k:=C \left( \|G_k(U_0^+)\|^\tau_{L^p(\Omega)}+
\|G_k(U_0^+)\|^\tau_{L^{\tau\beta}(\Omega)}
\right) +|\langle \mathfrak{B}  U_0 , G_k(U_0^+)\rangle|+C_\Theta \|G_k(U_0^+)\|_{L^1(\Omega)}.
$$
From \eqref{gun}, \eqref{nec2} and \eqref{nec1}, we have
$\lim_{k\to \infty} Z_k=0$ since $\tau\beta\in (1,m)$. 

This ends the proof of Lemma~\ref{step3}.
\end{proof}

For $\lambda>0$, we define $\varphi_\lambda:\mathbb R\to \mathbb R$ as follows
\begin{equation} \label{lam}  \varphi_\lambda(t)=t \exp\,(\lambda t^2) \quad \mbox{for every } t\in \mathbb R.
\end{equation}

We define 
$I_{0}(n,k)$ by 
\begin{equation} \label{I0de} I_0(n,k):=C_\Theta \|\varphi_\lambda(z_{n,k}^-)\|_{L^1(\Omega)}
-\int_\Omega \Psi_n(U_n) \,\varphi_\lambda(z_{n,k}^-)\,dx-\langle \mathfrak B U_n, \varphi_\lambda(z_{n,k}^-)\rangle.\end{equation}

\begin{lemma} \label{I0} Up to a subsequence of $\{U_n\}_{n}$, we have $ \limsup_{n\to \infty} I_0(n,k) \leq 0$ for each $k\geq 1$. \end{lemma}

\begin{proof}
It suffices to show that for each $k\geq 1$, by passing to a subsequence of $\{U_n\}$, we have $ \limsup_{n\to \infty} I_0(n,k) \leq 0$.  
Since $U_n\rightharpoonup U_0$ and $ \varphi_\lambda(z_{n,k}^-)\rightharpoonup 0$ (weakly) in $W_0^{1,\overrightarrow{p}}(\Omega)$ as $n\to \infty$, 
by the property $(P_2)$ for $ \mathfrak{B}$, we have $ \lim_{n\to\infty}  \langle \mathfrak{B} U_n, -\varphi_\lambda(z_{n,k}^-)\rangle=0$.  
Moreover, up to a subsequence, $\varphi_\lambda(z_{n,k}^-)\to 0$ (strongly) in $L^1(\Omega)$ as $n\to \infty$. 
Thus, it remains to show that 
\begin{equation} \label{micv}
\limsup_{n\to \infty} \left(-\int_\Omega \Psi_n(U_n) \,\varphi_\lambda(z_{n,k}^-)\,dx\right)\leq 0.\end{equation}
From \eqref{psi}, we have
\begin{equation} \label{byo}
\begin{aligned} 
-\int_\Omega \Psi_n(U_n)\,\varphi_\lambda (z_{n,k}^-)\,dx
 & \leq \sum_{j\in J_1} \int_{\{U_n\leq 0\}} |U_n|^{\theta_j-1} |\partial_j U_n|^{q_j} \varphi_\lambda(z_{n,k}^-)\,dx\\
 & \le e^{\lambda k^2} \sum_{j\in J_1} \int_\Omega  |U_n|^{\theta_j-1} |\partial_j U_n|^{q_j} \,z_{n,k}^-\,dx.
 \end{aligned}
\end{equation}

\noindent Let $j \in J_1$ be arbitrary. 
In view of Lemma~\ref{aplu}, for sufficiently small $\tau>0$ and
$\beta\in (1/\tau,m/\tau)$ fixed, there exists a  
positive constant $C$, independent of $n$ and $k$, such that
\begin{equation} \label{sea2} \int_{\{|U_n|\geq z_{n,k}^-\}}  |U_n|^{\theta_j-1} |\partial_j U_n|^{q_j} \,z_{n,k}^-\,dx\leq C \left(  \|z_{n,k}^-\|^\tau_{L^p(\Omega)}
+\|z_{n,k}^-\|^\tau_{L^{\tau\beta}(\Omega)} \right).
\end{equation}
We write $\Omega$ as the union of $\{|U_n|< z_{n,k}^-\}$ and $\{|U_n|\geq z_{n,k}^-\}$. 
Since $\theta_j\geq 1$, 
we see that 
$|U_n|^{\theta_j-1}\leq (z_{n,k}^-)^{\theta_j-1}$ on $\{|U_n|< z_{n,k}^-\}$. This and \eqref{sea2} imply that
\begin{equation} \label{plan}
\int_\Omega  |U_n|^{\theta_j-1} |\partial_j U_n|^{q_j} \,z_{n,k}^-\,dx\leq 
 C \left(  \|z_{n,k}^-\|^\tau_{L^p(\Omega)}
+\|z_{n,k}^-\|^\tau_{L^{\tau\beta}(\Omega)} \right)
 + \int_\Omega  |\partial_j U_n|^{q_j} \,(z_{n,k}^-)^{\theta_j}  \,dx. 
\end{equation}
With \eqref{byo} and \eqref{plan} in mind, to conclude \eqref{micv}, it suffices to show that 
for each $j\in J_1$, each term in the right-hand side of \eqref{plan} converges to zero as $n\to \infty$. 

In light of \eqref{nec2} and \eqref{nec1}, we see that the right-hand side of \eqref{sea2} converges to $0$ as $n\to \infty$ using here that $\tau\beta\in (1,m)$. 
For every $j\in J_1$, let  
$\alpha_j\in (0,\theta_j)$ satisfy $1<\vartheta_j<m$, where  
we define $\vartheta_j=(\theta_j -\alpha_j) p_j/(p_j-q_j)$. 
Since $z_{n,k}^-\leq k$, by H\"older's inequality, we have
$$ 
 \int_\Omega  |\partial_j U_n|^{q_j} \,(z_{n,k}^-)^{\theta_j}  \,dx \leq k^{\alpha_j} \int_\Omega  |\partial_j U_n|^{q_j} \,(z_{n,k}^-)^{\theta_j-\alpha_j}  \,dx
  \leq k^{\alpha_j} \|\partial_j U_n \|_{L^{p_j}(\Omega)}^{q_j} \|z_{n,k}^{-}\|_{L^{\vartheta_j}(\Omega)}^{\theta_j-\alpha_j}. 
$$
 The choice of $\alpha_j$ yields that $ \lim_{n\to \infty} \|z_{n,k}^{-}\|_{L^{\vartheta_j}(\Omega)}= 0$. 
Then, for every $j\in J_1$, the last term in the right-hand side of \eqref{plan} converges to $0$ as $n\to \infty$. 
This completes the proof of \eqref{micv}. 
\end{proof}


\section{Proof of Proposition~\ref{red2}} \label{redpro}

As explained in Section~\ref{strat}, we conclude \eqref{gtr} and \eqref{gtree} by showing that \eqref{alu} holds.  
We observe that in Case~2, we need only prove that $ \mathcal E_{U_n}(U_n^+,U_0^+)\to 0$ in $L^1(\Omega)$ as $n\to \infty$ since all $U_n$ and, hence, $U_0$ are non-negative functions.
Similarly, we can establish the other convergence claim in \eqref{alu}. We thus show the details only for $ \mathcal E_{U_n}(U_n^+,U_0^+)$ in \eqref{alu} and leave the modifications for $ \mathcal E_{U_n}(U_n^-,U_0^-)$ to the reader noting that  instead of $z_{n,k}$ in \eqref{cuore}, one needs to work with $y_{n,k}$ defined by $ y_{n,k}:=U_n^- - T_k(U_0^-)$. 

\vspace{0.2cm}
In light of the monotonicity assumption in \eqref{ellip}, we have $ \mathcal E_{U_n}(U_n^+,U_0^+)\geq 0$ a.e. in $\Omega$. Hence, to attain \eqref{alu} for $ \mathcal E_{U_n}(U_n^+,U_0^+)$, it remains to show that
\begin{equation} \label{sst}
\limsup_{n\to \infty}  \int_\Omega \mathcal E_{U_n}(U_n^+,U_0^+)\,dx \leq 0.
\end{equation}

\vspace{0.2cm}
\noindent {\bf Notation.}
Let $\omega$ be a measurable subset of $\Omega$ and $v,w,z\in W_0^{1,\overrightarrow{p}}(\Omega)$. We introduce
\begin{equation} \label{7c}
\begin{aligned}
& E_{j,U_n}(v,w):=A_j(x,U_n(x),\nabla v(x))-A_j(x,U_n(x),\nabla w(x)),\\
& E_{n,\omega}(v,w,z):=  \sum_{j=1}^N\int_\omega E_{j,U_n} (v,w)\,\partial_j z\,dx.
\end{aligned}
\end{equation}
If either of the variables $v,w$ and $z$ or $\omega$  depends on $n$, we drop the subscript $n$ in $E_{n,\omega}(v,w,z)$.

\vspace{0.2cm}
Fix $k>0$. We define $z_{n,k}$ as in \eqref{cuore}.  
From \eqref{gk}, we see that $G_k(U_0^+)\geq 0$ a.e. in $\Omega$.
Since
 \begin{equation*}\label{bee} 
 U_n^+-U_0^+=z_{n,k}^+ - z_{n,k}^-
 -G_k(U_0^+),
 \end{equation*}  
we infer that 
 \begin{equation}\label{bewq} 
 \begin{aligned}
 \int_\Omega \mathcal E_{U_n}(U_n^+,U_0^+)\,dx=& E_{\Omega} (T_k(U_0^+),U_0^+,U_n^+-U_0^+)+E_{\Omega} (U_n^+,T_k(U_0^+),z_{n,k}^+)\\
 &+E_\Omega(U_n^+,T_k(U_0^+),-z_{n,k}^-)+E_\Omega (T_k(U_0^+),U_n^+,G_k(U_0^+)).
\end{aligned}
 \end{equation} 
We show that, up to a subsequence of $\{U_n\}$, there exist $\mu_j\in L^{p_j'}(\Omega)$ for $1\leq j\leq N$ such that
\begin{equation}\label{rio}
 \lim_{n \to\infty} 
E_\Omega (T_k(U_0^+),U_n^+,G_k(U_0^+))
 =\sum_{j=1}^N\int_\Omega \mu_j \,\partial_j G_k(U_0^+)\,dx. 
        \end{equation}
Indeed, by the growth condition in \eqref{ellip}, there exists a positive constant $C$, independent of $n$ and $k$ such that, for $1\leq j\leq N$, it holds
\begin{equation} \label{10v}
\|A_j(x,U_n,\nabla T_k(U_0^+))\|_{L^{p_j'}(\Omega)}+\|A_j(x,U_n,\nabla U_0^+)\|_{L^{p_j'}(\Omega)}+\|A_j(x,U_n,\nabla U_n^+)\|_{L^{p_j'}(\Omega)}
\leq C.
\end{equation}
Hence, passing to a subsequence of $\{U_n\}$, we can find $\mu_j\in L^{p_j'}(\Omega)$ for $1\leq j\leq N$ such that 
\begin{equation} \label{wa}
 E_{j,U_n}(T_k(U_0^+),U_n^+) \rightharpoonup \mu_j \ \mbox{(weakly) in } L^{p_j'}(\Omega)\ \mbox{as } n\to \infty.
\end{equation}
This proves the claim in \eqref{rio}. 

We complete the proof of \eqref{sst} assuming that the next two results hold. 
 \begin{lemma} \label{res1} For every $k\geq 1$, there exist $R_1(k)$ and $R_2(k)$
such that, up to a subsequence of $\{U_n\}_{n\geq 1}$, we have
\begin{equation*}  \label{claim1} 
\begin{aligned}
&  \limsup_{n\to \infty} E_{\Omega} (T_k(U_0^+),U_0^+,U_n^+-U_0^+)\leq R_1(k),\\
&
 \limsup_{n\to \infty} E_{\Omega} (U_n^+,T_k(U_0^+),z_{n,k}^+)\leq R_2(k), \end{aligned}
\end{equation*}  
where $\lim_{k\to \infty} R_1(k)=\lim_{k\to \infty} R_2(k)=0$.
\end{lemma}

\begin{lemma} \label{res2}  For every $k\geq 1$, by passing to a subsequence of $\{U_n\}_{n\geq 1}$, we have 
\begin{equation} \label{claim2} 
 \limsup_{n\to \infty} E_\Omega(U_n^+,T_k(U_0^+),-z_{n,k}^-)\leq 0.
\end{equation}
\end{lemma}

For the proof of Lemmata~\ref{res1} and \ref{res2}, we refer to 
Sections~\ref{sec-res1} and ~\ref{sec-res2}, respectively.

\noindent   
Hence, by using a diagonal argument, there exists a subsequence of $\{U_n\}_{n\geq 1}$ such that for every $k \ge 1$ \eqref{rio} holds  and Lemmata~\ref{res1} and \ref{res2} apply. 

Consequently, using also \eqref{bewq}, we deduce that   
  \begin{equation} \label{egg}  \limsup_{n\to \infty} 
 \int_\Omega \mathcal E_{U_n}(U_n^+,U_0^+)\,dx \leq R_1(k)+R_2(k)
+  \sum_{j=1}^N\int_\Omega \mu_j\,\partial_j G_k(U_0^+)\,dx
 \end{equation}
 for every integer $k\geq 1$.  

Hence, by using \eqref{gun} and  letting $k\to \infty$ in \eqref{egg}, we conclude the proof of \eqref{sst}.

\subsection{Proof of Lemma~\ref{res1}} \label{sec-res1} 
Let $k\geq 1$. By Lemma~\ref{step3}, it suffices to show that 
there exist a positive constant $C$, independent of $k$, and $R_0(k)$ with $\lim_{k \to \infty} R_0(k)= 0$ such that
\begin{eqnarray} 
& \displaystyle \limsup_{n\to \infty} E_{\Omega} (T_k(U_0^+),U_0^+,U_n^+-U_0^+)
\leq C\,\limsup_{n\to \infty} \|G_k(U_n)\|_{W_0^{1,\overrightarrow{p}}(\Omega)}-R_0(k), \label{gol} \\
&\displaystyle  \limsup_{n\to \infty} E_{\Omega} (U_n^+,T_k(U_0^+),z_{n,k}^+)\
 \leq \limsup_{n\to \infty} \langle \mathcal A U_n, z_{n,k}^+\rangle +C\,\limsup_{n\to \infty} \|G_k(U_n)\|_{W_0^{1,\overrightarrow{p}}(\Omega)}.\label{erca}
\end{eqnarray}

{\em Proof of \eqref{gol}.}
We define $L_1(n,k)$, $L_2(n,k)$ and $L_3(n,k)$ by
$$ \begin{aligned}
&L_1(n,k):=E_{\{|U_n|<k\}} (T_k(U_0^+),U_0^+,U_n^+-U_0^+),\\
&  L_2(n,k):=E_{\{ |U_n|\geq k\}} (T_k(U_0^+), U_0^+,U_n^+),\\
&L_3(n,k):=E_{\{|U_n|\geq k\}} (T_k(U_0^+),U_0^+,U_0^+)=\sum_{j=1}^N\int_{\{|U_n|\geq k\}} E_{j,U_n}(T_k(U_0^+),U_0^+)\,\partial_j U_0^+ \,dx.
\end{aligned} $$
It follows that
\begin{equation} \label{nomv}
\begin{aligned}
E_{\Omega} (T_k(U_0^+),U_0^+,U_n^+-U_0^+)=&
L_1(n,k)+L_2(n,k)-L_3(n,k).
\end{aligned} 
\end{equation}
Remark that $\chi_{\{|U_n|\geq k\}}\partial_jU_n^+=\chi_{\{U_n\geq k\}}\partial_j G_k(U_n)$ for every $1 \le j \le N$. Hence, 
by H\"older's inequality and \eqref{10v}, we obtain that 
\begin{equation} \label{lig2} 
\big |L_2(n,k)\big | =\left|E_{\{ U_n\geq k\}} (T_k(U_0^+), U_0^+,G_k(U_n))\right| \leq C \|G_k(U_n)\|_{W_0^{1,\overrightarrow{p}}(\Omega)},
\end{equation}
where $C>0$ is a constant independent of $n$ and $k$.

By the Dominated Convergence Theorem, we see that $\lim_{k\to \infty} R_0(k)=0$, where
we define
\begin{equation} \label{wion}
R_0(k):= \sum_{j=1}^N \int_{\{U_0\geq k\}} E_{j,U_0}(0,U_0) \,\partial_j U_0\,dx.
\end{equation}
Using \eqref{nomv} and \eqref{lig2}, we conclude the proof of \eqref{gol} by showing that 
\begin{equation} \label{wiona} \lim_{n\to \infty} L_3(n,k)=R_0(k)\quad \mbox{and}\quad \lim_{n\to \infty} L_1(n,k)=0.
\end{equation}

Let $1\leq j\leq N$ be arbitrary.  
We have \begin{equation} \label{les} \chi_{\{|U_n|\geq k\}}\partial_j U_0^+\to \chi_{\{U_0\geq k\}}\partial_j U_0\quad \mbox{(strongly) in }L^{p_j}(\Omega)\ \mbox{ as } n\to \infty.\end{equation} 
Since $E_{j,U_n}(T_k(U_0^+),U_0^+)\to E_{j,U_0}(T_k(U_0^+),U_0^+)$ a.e. in $\Omega$ as $n\to \infty$,  
using \eqref{10v} and passing to a subsequence of $\{U_n\}$, we find that
\begin{equation} \label{omm} E_{j,U_n}(T_k(U_0^+),U_0^+)\rightharpoonup E_{j,U_0}(T_k(U_0^+),U_0^+)\ \mbox{(weakly) in } L^{p_j'}(\Omega)\ \mbox{ as }n\to \infty. 
\end{equation}
Hence, from \eqref{les} and \eqref{omm}, we infer that  as $n\to \infty$ 
$$
E_{j,U_n}(T_k(U_0^+),U_0^+) \chi_{\{|U_n|\geq k\}}\partial_j U_0^+\to 
E_{j,U_0}(T_k(U_0^+),U_0^+) \chi_{\{U_0\geq k\}}\partial_j U_0\ \ \mbox{(strongly) in } L^1(\Omega).
$$ 
This proves the first limit in \eqref{wiona} since $E_{j,U_0}(T_k(U_0^+),U_0^+) =E_{j,U_0} (0,U_0)$ on $\{U_0\geq k\}$.

Let $1\leq j\leq N$ be arbitrary. We note that the sequences 
$\{ E_{j,U_n} (T_k(U_0^+),U_0^+) \, \chi_{\{|U_n|\leq k\}} \}_{n\geq 1}$ and 
$\{ A_j(x,U_n,\nabla T_k(U_0^+)) \,\chi_{\{|U_n|\leq k\}}\}_{n\geq 1} $ 
are uniformly integrable in 
$L^{p_j'}(\Omega)$ with respect to $n$. Since $U_n \to U_0$ a.e. in $\Omega$ as $ n\to \infty$, 
by Vitali's Theorem, we obtain that as $n\to \infty$
\begin{eqnarray}
& E_{j,U_n} (T_k(U_0^+),U_0^+) \, \chi_{\{|U_n|\leq k\}} \to E_{j,U_0} (T_k(U_0^+),U_0^+) \, \chi_{\{|U_0|\leq k\}} 
\  \mbox{ in } L^{p_j'}(\Omega),  \label{vioc}  \\
& A_j(x,U_n,\nabla T_k(U_0^+)) \,\chi_{\{|U_n|\leq k\}}\to  
A_j(x,U_0,\nabla T_k(U_0^+)) \,\chi_{\{|U_0|\leq k\}}\  \mbox{in } L^{p_j'}(\Omega).  \label{unt}
\end{eqnarray}
Recall that $\partial_j U_n^+\rightharpoonup \partial_j U_0^+$ and $\partial_j z_{n,k}^+\rightharpoonup \partial_j G_k(U_0^+) $ (weakly) in $L^{p_j}(\Omega)$ as $n\to \infty$. Since $\chi_{\{|U_0|\leq k\}}\, \partial_j G_k(U_0^+) =0$, using \eqref{vioc} and \eqref{unt}, we conclude that as $n\to \infty$
\begin{eqnarray}
 &\displaystyle
 E_{j,U_n} (T_k(U_0^+),U_0^+) \, \chi_{\{|U_n|\leq k\}} \,\partial_j(U_n^+-U_0^+)\to 0 \ \mbox{in } L^1(\Omega), 
 \label{asis}\\
 & \displaystyle 
 A_j(x,U_n,\nabla T_k(U_0^+)) \,\chi_{\{|U_n|\leq k\}}\, \partial_j z_{n,k}^+\to 0  \ \mbox{in } L^1(\Omega).
 \label{asis22}
\end{eqnarray}

Moreover, from \eqref{asis} and 
the squeeze law, we get the second limit in \eqref{wiona}.

\vspace{0.2cm}
{\em Proof of \eqref{erca}.} We define $P_1(n,k)$ and $P_2(n,k)$ as follows
$$ 
\begin{aligned}
& P_1(n,k):= \sum_{j=1}^N  \int_{\{|U_n|<k\}} A_j(x,U_n,\nabla T_k(U_0^+)) \,\partial_j z_{n,k}^+\,dx,\\
& P_2(n,k):=  \sum_{j=1} ^N \int_{\{U_n \ge k\}} A_j(x,U_n,\nabla T_k(U_0^+)) \,\partial_j(- z_{n,k}^+)\,dx.
\end{aligned} $$
Since on the set $\{ U_n\geq k\}$ we have $z_{n,k}^+=U_n-T_k(U_0^+)$ and $\partial_j U_n=\partial_j G_k(U_n)$ for $1\leq j\leq N$,  
the definition of $P_2(n,k)$ yields that $P_2(n,k)=P_{2,1}(n,k)+P_{2,2}(n,k)$, where
\begin{equation*} \label{vep2} 
\begin{aligned}
& P_{2,1}(n,k):=-  \sum_{j=1}^N\int_{\{U_n \ge k\}} A_j(x,U_n,\nabla T_k (U_0^+))\, \partial_j G_k(U_n)\,dx,\\
& P_{2,2}(n,k):= \sum_{j=1}^N \int_{\{U_n \ge k\}\cap \{0<U_0<k\}}  A_j(x,U_n,\nabla U_0)\,\partial_j U_0\,dx.
\end{aligned}\end{equation*}

From \eqref{asis22}, we get $\lim_{n\to \infty} P_1(n,k)=0$. 
When $z_{n,k}^+>0$, then $U_n^+>0$ so that 
\begin{equation*} \label{bor}
 U_n=U_n^+\ \mbox{on } \{z_{n,k}^+>0\}\quad \mbox{and}\quad 
\langle \mathcal A U_n^+, z_{n,k}^+\rangle =\langle \mathcal A U_n, z_{n,k}^+\rangle .
\end{equation*}
Hence, we arrive at 
\begin{equation*} \label{dop}
E_{\Omega} (U_n^+,T_k(U_0^+),z_{n,k}^+)
= \langle \mathcal A U_n, z_{n,k}^+ \rangle  -  P_1(n,k)+ P_2(n,k).
\end{equation*}

Consequently, we end the proof of \eqref{erca} once we show that
\begin{equation} \label{vepa} \limsup_{n\to \infty} P_2(n,k)\leq C \limsup_{n\to \infty} \|G_k(U_n)\|_{W_0^{1,\overrightarrow{p}}(\Omega)},
\end{equation} where $C$ is a positive constant independent of $k$.

As for \eqref{lig2}, we find a positive constant $C$, independent of $n$ and $k$ such that
\begin{equation} \label{laba} |P_{2,1}(n,k)|\leq C\|G_k(U_n)\|_{W_0^{1,\overrightarrow{p}}(\Omega)}.
\end{equation}
For $1\leq j\leq N$, by the Dominated Convergence Theorem, we get
 $ \chi_{\{ U_n\geq k\}}\chi_{\{0<U_0<k\}}\partial_j U_0\to 0$ (strongly) in $ L^{p_j}(\Omega)$ as $n\to \infty$. 
Since $A_j(x,U_n,\nabla U_0)\rightharpoonup A_j(x,U_0,\nabla U_0)$ (weakly) in $L^{p_j'}(\Omega)$ as $n\to \infty$, we infer that $ \lim_{n\to \infty} P_{2,2}(n,k)=0$. This, together with \eqref{laba}, proves 
\eqref{vepa}.

\vspace{0.2cm}
The proof of Lemma~\ref{res1} is now complete. \qed

\begin{rem} \label{rema63} We point out that the reasoning in the proof of
$\lim_{n\to \infty} L_1(n,k)=0$ cannot be extended to get 
$\lim_{n\to \infty}  E_{\Omega} (T_k(U_0^+),U_0^+,U_n^+-U_0^+)=0$. Indeed, in the growth condition in \eqref{ellip}, we have taken the greatest exponent for $|t|$ regarding the anisotropic Sobolev inequalities
so that we don't have the compactness of the embedding $W_0^{1,\overrightarrow{p}}(\Omega)\hookrightarrow L^{p^\ast}(\Omega)$. Hence, 
we cannot infer that 
$\{ E_{j,U_n} (T_k(U_0^+),U_0^+)  \}_{n\geq 1}$ 
is uniformly integrable in 
$L^{p_j'}(\Omega)$ with respect to $n$.
\end{rem}

\subsection{Proof of Lemma~\ref{res2}} \label{sec-res2}
We need to show that, up to a subsequence,  \eqref{claim2} holds, namely,
\begin{equation} \label{clam2} 
\limsup_{n\to \infty} 
\int_\Omega D(n,k)\,dx\leq 0,
\end{equation}
where we define $D(n,k)$ by 
\begin{equation} \label{sho}  D(n,k):=\sum_{j=1}^N \left[A_j(x,U_n,\nabla U_n^+)-A_j(x,U_n,\nabla T_k(U_0^+))\right] \partial_j (-z_{n,k}^-).
\end{equation}
We choose $\lambda=\lambda(k)>0$ large such that $ 4\nu_0^2 \,\lambda>\phi^2(k)$, where $\phi$ appears in the growth assumption on $\Phi$ in \eqref{cond1}, while $\nu_0>0$ 
is given by the coercivity condition in \eqref{ellip}. 
We define $\varphi_\lambda$ as in \eqref{lam}. 
Our choice of $\lambda$ ensures that for every $ t\in \mathbb R$ 
\begin{equation} \label{bea}  \lambda t^2-\frac{\phi(k)}{2\nu_0} |t|+\frac{1}{4}>0 \quad  \text{and, hence,}\quad 
 \varphi_\lambda'(t) -\frac{\phi(k)}{\nu_0} \, |\varphi_\lambda(t)| >\frac{1}{2}. 
\end{equation} 
Recall that $I_{0}(n,k)$ is defined in \eqref{I0de}. For convenience, we set
\begin{equation} \label{hld} 
\begin{aligned}  
& I_1(n,k):= \sum_{j=1}^N \int_\Omega A_j(x,U_n,\nabla T_k(U_0^+))\, \partial_j (\varphi_\lambda(z_{n,k}^-))\,dx+ E_\Omega (U_n,U_n^+,\varphi_\lambda(z_{n,k}^-)),\\
& I_2(n,k):= \sum_{j=1}^N \int_\Omega
  \left[ A_j(x,U_n,\nabla U_n^+)\,
 \partial_j T_k(U_0^+)+A_j(x,U_n,\nabla T_k (U_0^+))\,\partial_j z_{n,k}\right]  \varphi_\lambda( z_{n,k}^-)\,dx.\\
\end{aligned}
\end{equation}

We divide the proof of \eqref{clam2} into two steps.

\vspace{0.2cm}
{\sc Step 1.} {\em Let $\nu_0$ and $c,\phi$ be as in \eqref{ellip} and \eqref{cond1}, respectively. We have}
\begin{equation} \label{erb}
\frac{1}{2} \int_\Omega D(n,k)\,dx\leq I_0(n,k)+I_1(n,k)+\phi(k) \left[ \frac{I_2(n,k)}{\nu_0}+\int_\Omega c(x) \,\varphi_\lambda (z_{n,k}^-)\,dx\right].
\end{equation}

{\em Proof of STEP 1.} 
On the set $\{U_n>T_k(U_0^+)\}$, we have $D(n,k) =0$ since
$z_{n,k}^-=0$ and, hence, $\partial_j z_{n,k}^-=0$ for $1\leq j\leq N$. In turn, on the set $\{U_n\leq T_k(U_0^+)\}$, we find that 
$z_{n,k}^-=T_k(U_0^+)-U_n^+$ and, by the monotonicity condition in \eqref{ellip}, it follows that $D(n,k)\geq 0$. 
Hence, we have 
\begin{equation} \label{mod} D(n,k)\geq 0,\quad  z_{n,k}^-\in [0,k],\quad 
\varphi_\lambda(z_{n,k}^-) \,\partial_j (-z_{n,k}^-)=\varphi_\lambda(z_{n,k}^-) \,\partial_j (U_n^+-T_k(U_0^+))
\ \mbox{ a.e. in } \Omega \end{equation}
for each $1\leq j\leq N$. Then, 
using \eqref{bea}, we find that
\begin{equation} \label{fib1}
\frac{1}{2}\int_\Omega D(n,k) \,dx
\leq \int_\Omega D(n,k) \,\varphi_\lambda'(z_{n,k}^-) \,dx -\frac{\phi(k)}{\nu_0} \int_\Omega 
D(n,k)\,\varphi_\lambda(z_{n,k}^-)\,dx.
\end{equation}

\noindent From \eqref{sho} and \eqref{hld}, we observe that
\begin{equation} \label{fib2} 
\int_\Omega D(n,k) \,\varphi_\lambda'(z_{n,k}^-) \,dx=E_\Omega (U_n^+,T_k(U_0^+),-\varphi_\lambda(z_{n,k}^-))
=I_1(n,k)+ \langle \mathcal A U_n, -\varphi_\lambda(z_{n,k}^-)\rangle.
\end{equation}
Since $ z_{n,k}^-\in W_0^{1,\overrightarrow{p}}(\Omega)\cap L^\infty(\Omega)$, we have $\varphi_\lambda(z_{n,k}^-)\in W_0^{1,\overrightarrow{p}}(\Omega)\cap L^\infty(\Omega)$ so that $\varphi_\lambda(z_{n,k}^-) $ can be taken as a test function in \eqref{fff2}. Hence, using \eqref{cond1} and $I_0(n,k)$ given by \eqref{I0de}, we find that
\begin{equation} \label{fib3}   \langle \mathcal A U_n, -\varphi_\lambda(z_{n,k}^-)\rangle\leq \int_\Omega  \Phi(U_n)\, \varphi_\lambda(z_{n,k}^-)\,dx+I_0(n,k).
\end{equation}
In view of \eqref{fib1}--\eqref{fib3}, we conclude \eqref{erb} by showing that 
\begin{equation} \label{doni}
 \int_\Omega \Phi(U_n)\, \varphi_\lambda(z_{n,k}^-)\,dx
 \leq \phi(k) \left[ \frac{  \int_\Omega 
D(n,k)\,\varphi_\lambda(z_{n,k}^-)\,dx+I_2(n,k)}{\nu_0}+\int_\Omega c(x) \,\varphi_\lambda (z_{n,k}^-)\,dx\right]. \end{equation}

\noindent To this end, we next prove that 
\begin{equation} \label{don}
 \int_\Omega  \Phi(U_n)\, \varphi_\lambda(z_{n,k}^-)\,dx\leq \frac{\phi(k)}{\nu_0} I_3(n,k) +\phi(k) \int_\Omega c(x) \,\varphi_\lambda (z_{n,k}^-)\,dx, 
\end{equation}
where $I_{3}(n,k)$ is defined by 
\begin{equation} \label{don2}
I_{3}(n,k):= \sum_{j=1}^N \int_{\{0<U_n\leq T_k(U_0^+)\}} A_j(x,U_n,\nabla U_n)\, \partial_j U_n\, \varphi_\lambda (z_{n,k}^-)\,dx.
\end{equation}
Indeed, since $z_{n,k}^-=0$ on $\{U_n^+>T_k(U_0^+)\}$ and $ \Phi(U_n) \leq 0\leq \varphi_\lambda (z_{n,k}^-)$ on $\{U_n\leq 0\}$, we have
$$  
\int_\Omega \Phi(U_n)\, \varphi_\lambda(z_{n,k}^-)\,dx=\int_{\{U_n^+\leq T_k(U_0^+)\}}  \Phi(U_n)\, \varphi_\lambda(z_{n,k}^-)\,dx\leq
\int_{\{0<U_n\leq T_k(U_0^+)\}} 
\Phi(U_n)\, \varphi_\lambda(z_{n,k}^-)\,dx.
$$
Next, from the growth condition on $\Phi$ in \eqref{cond1} and the coercivity condition in \eqref{ellip}, we get
$$\begin{aligned} \int_{\{0<U_n\leq T_k(U_0^+)\}} 
\Phi(U_n)\, \varphi_\lambda(z_{n,k}^-)\,dx &\leq 
\phi(k) \int_{\{0<U_n\leq T_k(U_0^+)\}} \left(\sum_{j=1}^N |\partial_j U_n|^{p_j} +c(x)\right) \varphi_\lambda (z_{n,k}^-)\,dx\\ 
& \leq \frac{\phi(k)}{\nu_0} I_{3}(n,k) +\phi(k) \int_\Omega c(x) \,\varphi_\lambda (z_{n,k}^-)\,dx.
\end{aligned}$$ 
Consequently, the assertion of \eqref{don} is proved.  

Since $\varphi(z_{n,k}^-)=0$ on $\{U_n>T_k(U_0^+)\}$, we have
\begin{equation} \label{gxb1} 
 I_{3}(n,k)=
 \sum_{j=1}^N \int_{\Omega}  A_j(x,U_n,\nabla U_n^+)\, \partial_j U_n^+\, \varphi_\lambda (z_{n,k}^-)\,dx 
=\int_\Omega D(n,k)\varphi_\lambda(z_{n,k}^-)\, dx+I_2(n,k),
\end{equation} 
where $I_2(n,k)$ is given in \eqref{hld}. From \eqref{don} and \eqref{gxb1}, we attain \eqref{doni}. This ends the proof of \eqref{erb} and of Step 1. \qed

\vspace{0.2cm} 
{\sc Step 2.} {\em Proof of \eqref{clam2} concluded.}

\vspace{0.2cm}
{\em Proof of STEP 2.} 
\vspace{0.2cm}
Since $0\leq c(x) \,\varphi_{\lambda} (z_{n,k}^-)\leq k\,e^{\lambda k^2}\,c(x) $ a.e. in $\Omega$ and $c(x)\, \varphi_\lambda( z_{n,k}^-)\to 0$ a.e. in $\Omega$ as $n\to \infty$, by the Dominated Convergence Theorem, we have $\lim_{n\to \infty} \int_\Omega c(x) \,\varphi_\lambda (z_{n,k}^-)\,dx=0$. 
In view of Step~1 and Lemma~\ref{I0}, we conclude Step 2 by showing that, up to a subsequence, 
\begin{equation} \label{I12} \lim_{n\to \infty} I_1 (n,k)=0\quad \mbox{and}\quad \lim_{n\to \infty} I_2(n,k)=0.\end{equation}

\noindent From \eqref{usi}, we have that
both $z_{n,k}^-$ and $ \varphi_\lambda(z_{n,k}^-)$ converge to $0$  
weakly in $W_0^{1,\overrightarrow{p}}(\Omega)$ as $n\to \infty$. In particular, for each $1\leq j\leq N$, it holds 
\begin{equation} \label{wom} \partial_j (\varphi_\lambda(z_{n,k}^-))\rightharpoonup 0\ \mbox{ (weakly) in }L^{p_j}(\Omega)\ \mbox{as }n\to \infty.\end{equation}
We recall that $z_{n,k}^-=T_k(U_0^+)$ on $\{U_n\leq 0\}$ and $\varphi_\lambda(z_{n,k}^-)=0$ on $\{U_n>T_k(U_0^+)\}$. 

\vspace{0.2cm}
1) We show that $\lim_{n\to \infty} I_{1}(n,k)=0$. If $I_{1,1}(n,k)$ is the first term in $I_1(n,k)$ in \eqref{hld}, then   
\begin{equation} \label{jot} \begin{aligned} 
I_{1,1}(n,k)&= \sum_{j=1}^N \int_{\{U_n< 0\}} A_j(x,U_n,\nabla T_k(U_0^+)) \,\partial_j (\varphi_\lambda(T_k(U_0^+))\,dx\\
&\quad +  \sum_{j=1}^N \int_{\{0\leq U_n\leq T_k(U_0^+)\}} A_j(x,U_n,\nabla T_k(U_0^+)) \,\partial_j (\varphi_\lambda(z_{n,k}^-))\,dx.
\end{aligned} \end{equation}
Let $1\leq j\leq N$ be arbitrary. By the Dominated Convergence Theorem, we get
\begin{equation} \label{zmn}
 \chi_{\{U_n\leq 0\}}\,\partial_j (\varphi_\lambda(T_k(U_0^+)))\to 0\ \mbox{ (strongly) in }
L^{p_j}(\Omega)\ \mbox{as } n\to \infty. \end{equation} 
Hence, using that $A_j(x,U_n,\nabla T_k(U_0^+))\rightharpoonup A_j(x,U_0,\nabla T_k(U_0^+))$ (weakly) in $L^{p_j'}(\Omega)$ as $n\to \infty$, we obtain that the first term in the right-hand side of \eqref{jot} converges to $0$ as $n\to \infty$. Furthermore, on the set 
$\{0\leq U_n\leq T_k(U_0^+)\}$, we have $U_n\leq k$ and the family $\{|A_j(x,U_n,\nabla T_k(U_0^+))|^{p_j'}\}_{n\geq 1}$ is uniformly integrable. 
Then, based on $A_j(x,U_n,\nabla T_k(U_0^+))\to A_j(x,U_0,\nabla T_k(U_0^+))$ a.e. in $\Omega$ as $n\to \infty$, by Vitali's Theorem, we infer the strong convergence as $n\to \infty$ 
\begin{equation} \label{forte}  
A_j(x,U_n,\nabla T_k(U_0^+))\, \chi_{\{0\leq U_n\leq T_k(U_0^+)\}}\to A_j(x,U_0,\nabla T_k(U_0^+))\, \chi_{\{0\leq U_0\leq T_k(U_0^+)\}}\ \mbox{in } L^{p_j'}(\Omega) 
. \end{equation} This, jointly with \eqref{wom}, implies that the second term in the right-hand side of \eqref{jot} converges to $0$ as $n\to \infty$. This proves that
$\lim_{n\to \infty} I_{1,1}(n,k)=0$.

\vspace{0.2cm}
We now show that the remaining term in the definition of $I_1(n,k)$ in \eqref{hld} converges to $0$ as $n\to \infty$, that is, $\lim_{n\to \infty} E_\Omega (U_n,U_n^+,\varphi_\lambda(z_{n,k}^-))=0$.
By the definition of $E_{n,\omega}(\cdot,\cdot,\cdot)$
in \eqref{7c}, since $z_{n,k}^-=T_k(U_0^+)=-z_{n,k}$ on $\{U_n\leq 0\}$, we get
$$  
E_\Omega (U_n,U_n^+,\varphi_\lambda(z_{n,k}^-))
=\sum_{j=1}^N\int_{\{U_n\leq 0\}}  \left[A_j(x,U_n,\nabla U_n)-A_j(x,U_n,\nabla U_n^+)\right] \partial_j (\varphi_\lambda(T_k(U_0^+))\,dx.
$$
As for \eqref{wa}, by passing to a subsequence of $\{U_n\}$,  we get that 
$$ \{A_j(x,U_n,\nabla U_n)\}_n,\quad \{A_j(x,U_n,\nabla U_n^+)\}_n\quad \mbox{and}\quad \{A_j(x,U_n,\nabla T_k (U_0^+))\}_n
$$ 
converge weakly in $L^{p_j'}(\Omega)$ as $ n\to \infty$. This and \eqref{zmn} yield $\lim_{n\to \infty} E_\Omega (U_n,U_n^+,\varphi_\lambda(z_{n,k}^-))=0$.

\vspace{0.2cm}
2) We show that $\lim_{n\to \infty} I_2(n,k)=0$. From \eqref{us} and $0\leq z_{n,k}^-\leq k$ a.e in $\Omega$, we have $\varphi_\lambda(z_{n,k}^-)\to 0$ a.e. in $\Omega$ as $n\to \infty$ and $0\leq \varphi_\lambda (z_{n,k}^-)\leq k \,e^{\lambda k^2}$ a.e. in $\Omega$. Thus, by the Dominated Convergence Theorem, for each $1\leq j\leq N$, we find that as $n\to \infty$
\begin{equation*} \label{runn}
 \varphi_\lambda( z_{n,k}^-)\,\partial_j T_k(U_0^+) \to 0\ \mbox{and} \ \ \chi_{\{U_n\leq 0\}}\, \varphi_\lambda(T_k(U_0^+))\,\partial_j z_{n,k} \to 0\ \mbox{(strongly) in } L^{p_j}(\Omega).
\end{equation*}
Consequently, we get \begin{equation} \label{farr} \begin{aligned} 
& \sum_{j=1}^N  \int_\Omega
A_j(x,U_n,\nabla U_n^+)\,
\varphi_\lambda( z_{n,k}^-)\,  \partial_j T_k(U_0^+) \,dx\to 0\quad \mbox{as } n\to \infty,\\
 &   \sum_{j=1}^N \int_{\{U_n<0\}}A_j(x,U_n,\nabla T_k (U_0^+)) \,\varphi_\lambda (T_k(U_0^+))\,\partial_j z_{n,k}\,dx
\to 0\quad \mbox{as } n\to \infty.
 \end{aligned}
\end{equation}

For $1\leq j\leq N$, we have $ z_{n,k}=-z_{n,k}^-$ on $\{0\leq U_n\leq T_k(U_0^+)\}$ so that 
using \eqref{forte} and the weak convergence $\partial_j z^-_{n,k}\rightharpoonup 0$ in $L^{p_j}(\Omega)$ as $n\to \infty$, we arrive at 
$$
A_j(x,U_n,\nabla T_k (U_0^+)) \,\chi_{\{0\leq U_n\leq T_k(U_0^+)\} }\,\partial_j z_{n,k}\to 0\ \ \mbox{in } L^1(\Omega) \ \  \mbox{as } n\to \infty.
$$
It follows that
\begin{equation} \label{firt2} 
 \sum_{j=1}^N\int_{\{0\leq U_n\leq T_k(U_0^+)\}} 
A_j(x,U_n,\nabla T_k (U_0^+)) \,\varphi_\lambda(z_{n,k}^-)\,\partial_j z_{n,k}\,dx\to 0\ \mbox{as } n\to \infty.
\end{equation}
Since $\varphi_\lambda(z_{n,k}^-)=0$ on $\{U_n>T_k(U_0^+)\}$, from \eqref{farr} and \eqref{firt2}, we find that $\lim_{n\to \infty} I_2(n,k)=0$, completing the proof of \eqref{I12} and of Step 2.\qed

This finishes the proof of Lemma~\ref{res2}. \qed 

\section{Proof of Theorem~\ref{teo2}} \label{sec7}

Let \eqref{IntroEq0}, \eqref{ass}, \eqref{m} and \eqref{ellip}--\eqref{etc} hold and, in addition, $\min_{1\leq j \leq N} \mathfrak a_j>0$. 
Here, we suppose that the function $f$ in \eqref{eq1} is not identically 0. In Case 2, we assume that $f\geq 0$ a.e. in $\Omega$. 
We approximate $f$ by a sequence of functions $f_n\in L^\infty(\Omega)$, taking for instance
$$f_n(x):=\frac{f(x)}{1+ |f(x)|/n}\quad \mbox{ for a.e. }  x\in \Omega.$$
In particular, in Case 2, we have $f_n\geq 0$ a.e. in $\Omega$. We remark the following properties
\begin{equation} \label{nice}
 |f_n|\leq |f|\ \mbox{a.e. in }\Omega, \ f_n\to f \ \mbox{a.e. in }\Omega \ \mbox{ and }f_n\to f \ \mbox{(strongly) in } L^1(\Omega)\ \mbox{ as }n\to \infty.
\end{equation}

\noindent With this approximation, assuming that $\mathfrak B$ belongs to $\mathfrak {BC}((1-\varepsilon)\,\mathcal A)$ for some $\varepsilon\in (0,1)$, in either Case 1 or Case 2, we can apply Theorem~\ref{mainth} for the problem generated by \eqref{eq1} with $f_n$ instead of $f$. 
Then such an approximate problem admits at least a solution $u_n$, namely, 
\begin{equation} \label{eqn}
\left\{  \begin{aligned} 
&\mathcal A u_n +\Phi(u_n)+\Theta(u_n) =\Psi(u_n)+
\mathfrak{B}u_n+f_n\quad \mbox{in } \Omega,\\ 
& u_n\in W_0^{1,\overrightarrow{p}}(\Omega), \quad \Phi(u_n)\in L^1(\Omega).
\end{aligned} \right.
\end{equation}
Moreover,  $\Phi(u_n)\,u_n$ and $\Psi(u_n)\,u_n$ belong to $ L^1(\Omega)$, $I_{u_n}(v):=\int_{\{|u_n|>0\}} \Psi(u_n)\,v\,dx\in \R$ and
	\begin{equation} \label{baaa}
S_{u_n,\Theta,f_n}(v)=\
I_{u_n}(v) +\langle \mathfrak{B}u_n,v\rangle \qquad \mbox{for every } v\in W_0^{1,\overrightarrow{p}}(\Omega)\cap L^\infty(\Omega).
	\end{equation} 
Furthermore, \eqref{baaa} holds for $v=u_n$. 

In the rest of the paper, we understand that $u_n$ is a solution of \eqref{eqn} with the above-mentioned properties that we obtain from Theorem~\ref{mainth}.

But, unlike Theorem~\ref{mainth}, to prove that $\{u_n\}_{n\geq 1}$ is uniformly bounded in $W_0^{1,\overrightarrow{p}}(\Omega)$, we need 
$\mathfrak{B}$ to satisfy the extra condition $(P_3)$ associated with $(1-\varepsilon)\,\mathcal A$, namely, 
for every $k>0$, 
\begin{equation} \label{exb}
(1-\varepsilon)\,\nu_0  \sum_{j=1}^N ||\partial_j u||_{L^{p_j}(\Omega)}^{p_j} -\langle \mathfrak{B} u,T_k (u)\rangle\to \infty \quad \mbox{as } \|u\|_{W_0^{1,\overrightarrow{p}}(\Omega)}
\to \infty.
\end{equation}
Thus, $\mathfrak B$ belongs to the class $\mathfrak{BC}_+((1-\varepsilon)\,\mathcal A)$. This assumption is made throughout this section. 

All the results in this section are derived in the framework of Theorem~\ref{teo2}.

\subsection{{\em A priori} estimates}
In order to obtain {\em a priori} estimates for $u_n$ solving \eqref{eqn} we need the following result, which is in the spirit of Lemma \ref{intorj}. 

\begin{lemma} \label{sima}
Let $k\geq 1$ be arbitrary and $\Phi_0$ be given by  \eqref{phiz}.  
Then, for every $\rho>0$, there exists a constant $C_\rho>0$ such that for all $n\geq 1$, we have
\begin{equation} \label{fik} 
 I_{u_n}(T_k(u_n))\leq \rho \sum_{j=1}^N  \|\partial_j u_n\|^{p_j}_{L^{p_j}(\Omega)}+\rho \int_{\{|u_n|\geq k\}} 
 |\Phi_0(u_n)|\,dx 
+C_\rho.
\end{equation}
\end{lemma}

\begin{rem} The property $\Phi(u_n)\,u_n\in L^1(\Omega)$ and \eqref{etc} ensure that $\int_{\{|u_n|\geq k\}} 
 |\Phi_0(u_n)|\,dx<\infty $ for all $k\geq 1$ and $n\geq 1$. 
\end{rem}

We define
$$  \mathfrak I_{m-1} (k,u_n)=\int_{ \{|u_n|>k\}} |u_n|^{m-1}\,dx,\quad
 \mathfrak I_{m-1,p_j} (k,u_n):= \int_{\{|u_n|> k\}} |u_n|^{m-1} |\partial_j u_n|^{p_j}\,dx.
$$
 If in the definition of $ \mathfrak I_{m-1,p_j} (k,u_n)$, we replace $m-1$ and $p_j$ by $\theta_j-1$ and $q_j$, respectively, then we obtain
$\mathfrak I_{\theta_j-1,q_j} (k,u_n)$.

\begin{proof}[Proof of Lemma~\ref{sima}] 
We observe that $$\mathfrak I_{m-1}(k,u_n)+\sum_{j=1}^N \mathfrak I_{m-1,p_j}(k,u_n)\leq \frac{1}{\min_{0\leq k\leq N} \mathfrak{a}_k} \int_{\{|u_n|\geq k\}} |\Phi_0(u_n)|\,dx<\infty.$$
Using the definition of $I_{u_n}(v)$, we see that
\begin{equation} \label{gas} \begin{aligned}  
 I_{u_n}(T_k(u_n))&=\sum_{j=1}^N \int_{\{0<|u_n| \le k\}} |u_n|^{\theta_j} |\partial_j u_n|^{q_j}\,dx+k
\sum_{j=1}^N \mathfrak I_{\theta_j-1,q_j} (k,u_n) \\
& \leq 2 \sum_{j=1}^N k^{\theta_j} \int_\Omega |\partial_j u_n|^{q_j} \,dx+k \sum_{j\in J_1}   \mathfrak I_{\theta_j-1,q_j} (k,u_n).
\end{aligned}
\end{equation} 
Let $\delta>0$ be arbitrary. By H\"older's inequality and Young's inequality, there exists a constant $C_\delta>0$ such that for every $n\geq 1$,
\begin{equation} \label{zz} \sum_{j=1}^N\int_\Omega |\partial_j u_n|^{q_j} \,dx\leq \sum_{j=1}^N \left({\rm meas}\,(\Omega)\right)^{1-\frac{q_j}{p_j}} \|\partial_j u_n\|_{L^{p_j}(\Omega)}^{q_j}\leq
\delta \sum_{j=1}^N \|\partial_j u_n\|_{L^{p_j}(\Omega)}^{p_j}+C_\delta.
\end{equation}

Let $j\in J_1$ be arbitrary. To estimate $\mathfrak I_{\theta_j-1,q_j} (k,u_n)$, we distinguish several cases:

\vspace{0.2cm}
Case (a) Let $q_j=0$ and $\theta_j\geq p+1$. Then, $j\in N_{\overrightarrow{\mathfrak{a}}}$ and from \eqref{m}, we have 
$m> \theta_j$. Hence, by Young's inequality, there exists a constant $C_{\delta}>0$ such that
for all $ n\geq 1$,
 \begin{equation} \label{miv}  \mathfrak I_{\theta_j-1,q_j} (k,u_n)=\int_{\{|u_n|> k\}} |u_n|^{\theta_j-1} \,dx\leq \delta \,
 \mathfrak I_{m-1}(k,u_n)+C_\delta. 
 \end{equation}

Case (b) Let $q_j=0$ and $\theta_j< p+1$. We set $\gamma_j=1-(\theta_j-1)/p$ and $C_j=\left({\rm meas}\,(\Omega)\right)^{\gamma_j}$. 
By H\"older's inequality, Remark~\ref{an-sob} and Lemma~\ref{young} in the Appendix, we find constants $C>0$ 
(depending on $N$, $\overrightarrow{p}$, $\theta_j$ and 
$\mbox{meas}\,(\Omega)$) and $C_\delta>0$ such that for all $n\geq 1$, 
\begin{equation}\label{hari0} 
\mathfrak I_{\theta_j-1,q_j} (k,u_n)
 \le C_j\,
  \|u_n\|_{L^{p}(\Omega)}^{\theta_j-1} 
 \leq C
\prod_{i=1}^N \|\partial_i u_n\|^{\frac{\theta_j-1}{N}}_{L^{p_i}(\Omega)}
\leq  \delta \sum_{i=1}^N  \|\partial_i u_n\|^{p_i}_{L^{p_i}(\Omega)}
+C_\delta.
\end{equation}

When $q_j>0$, we define $\zeta_j$ as follows
\begin{equation} \label{zetag} \zeta_j:=\theta_j -1-\frac{(m-1)\,q_j}{p_j}.
\end{equation}

Case (c) Let $q_j>0$ and $\zeta_j\leq 0$. We set $\gamma_j=1-q_j/p_j$ and $C_j=\left(\mbox{meas}\,(\Omega)\right)^{\gamma_j} $. 
Then, by H\"older's inequality and Lemma~\ref{young}, there exists $C_\delta>0$ such that
\begin{equation} \label{yy}   \begin{aligned} 
\mathfrak I_{\theta_j-1,q_j} (k,u_n)
& \leq C_j\,
 \|\partial_j u_n\|_{L^{p_j}(\Omega)}^{-\frac{p_j\zeta_j}{m-1}}
\left(\mathfrak I_{m-1,p_j}(k,u_n)
\right)^{\frac{\theta_j-1}{m-1}}\\
 &\leq \delta \mathfrak I_{m-1,p_j}(k,u_n)
+  \delta\|\partial_j u_n\|_{L^{p_j}(\Omega)}^{p_j}+C_\delta.
\end{aligned} 
\end{equation}

Case (d) Let $q_j>0$ and $\zeta_j>0$. We distinguish three sub-cases:
\begin{itemize} 
\item[(d$_1$)] Let $\mathfrak{m}_j>1$ and $\theta_j\geq p$. Then, $j\in P_{\overrightarrow{\mathfrak{a}}}$ and from \eqref{m}, we have 
$m>\theta_j=\min\{\mathfrak m_j,\theta_j\}$. 
We set $\gamma_j:=(m-\theta_j)/(m-1)$ and $C_j:=\left(\mbox{meas}\,(\Omega)\right)^{\gamma_j}$. It follows that
\begin{equation} \label{figo}
\begin{aligned} \mathfrak I_{\theta_j-1,q_j} (k,u_n)
& \leq C_j
\left(
\mathfrak I_{m-1}(k,u_n)\right)^{\frac{\zeta_j}{m-1}} 
 \left(\mathfrak I_{m-1,p_j}(k,u_n) \right)^{\frac{q_j}{p_j}}\\
&\leq \delta\,\mathfrak I_{m-1}(k,u_n) +\delta \, \mathfrak I_{m-1,p_j}(k,u_n)+C_\delta,
\end{aligned}
 \end{equation}
where $C_\delta>0$ is a suitable constant depending on $\delta$.

\item[(d$_2$)] Let $\mathfrak{m}_j>1$ and $\theta_j<p$. Then, from \eqref{m}, we see that $m>\mathfrak{m}_j=\min\{\mathfrak{m}_j,\theta_j\}$. 

\item[(d$_3$)] Let $\mathfrak{m}_j\leq 1$. Here, we have $m>1\geq \mathfrak{m}_j$. 
\end{itemize}

We next treat sub-cases (d$_2$) and (d$_3$) together to get \eqref{ttt}  below. 
Using that 
$m>\mathfrak{m}_j$, we define 
\begin{equation} \label{gem} 
\gamma_j:=1-\frac{\zeta_j}{p} -\frac{q_j}{p_j}=\frac{1}{p} \left[ \frac{q_j}{p_j}(m-\mathfrak{m_j}) +1-\frac{q_j}{p_j}\right]\in (0,1).
\end{equation}
We let $C_j=\left(\mbox{meas}\,(\Omega)\right)^{\gamma_j}$. 
By H\"older's inequality, the anisotropic Sobolev inequality \eqref{ASI} in the Appendix and Lemma~\ref{young}, we find constants $C>0$ 
(depending on $N$, $\overrightarrow{p}$, $\theta_j$, $q_j$, $m$ and 
$\mbox{meas}\,(\Omega)$) and $C_\delta>0$ such that for all $n\geq 1$, we have
\begin{equation} \label{ttt} \begin{aligned} 
\mathfrak I_{\theta_j-1,q_j} (k,u_n)&\leq C_j \|u_n\|_{L^{p}(\Omega)}^{\zeta_j}
\left(
\mathfrak I_{m-1,p_j}(k,u_n)
\right)^{\frac{q_j}{p_j}}\\
&\leq C \left(
\mathfrak I_{m-1,p_j}(k,u_n)
\right)^{\frac{q_j}{p_j}} \prod_{i=1}^N \|\partial_i u_n\|_{L^{p_i}(\Omega)}^{\frac{\zeta_j}{N}}\\
& \leq \delta \,\mathfrak I_{m-1,p_j}(k,u_n)+ \delta \sum_{i=1}^N  \|\partial_i u_n\|_{L^{p_i}(\Omega)}^{p_i} +C_\delta.
\end{aligned}\end{equation} 

Since $\delta>0$ is arbitrary, the conclusion of Lemma~\ref{sima} follows
from \eqref{gas} based on the inequalities in \eqref{zz}--\eqref{hari0}, \eqref{yy}, \eqref{figo} and \eqref{ttt}. 
\end{proof}

We can now proceed with the proof of the {\it a priori} estimates of $u_n$.

\begin{prop} \label{lem-ad} The following hold.
\begin{itemize}
\item[(a)] There exists a positive constant $C$ such that for all $n\geq 1$, we have
\begin{equation}\label{cas}
\|u_n\|_{W_0^{1,\overrightarrow{p}}(\Omega)}+ \int_\Omega |\Phi(u_n)|\,dx \leq C.
\end{equation}

\item[(b)]There exists $u_0\in W_0^{1,\overrightarrow{p}}(\Omega)$ such that, up to a subsequence of $\{u_n\}_{n\geq 1}$,  
\begin{equation} \label{wwa} 
 u_n\rightharpoonup u_0\ \mbox{(weakly) in } W_0^{1,\overrightarrow{p}}(\Omega),\quad 
u_n \to u_0\ \mbox{a.e. in } \Omega\ \mbox{as }n\to \infty.
\end{equation} 
\end{itemize}
\end{prop}

\begin{proof} 
$(a)$ 
We fix $k\geq 1$ large such that $k^{m-1} (k-1) \min_{1\leq j\leq N} \mathfrak{a}_j\geq \nu_0$. 
We define 
\begin{equation} \label{knk} 
 K_{n,k}:=\sum_{j=1}^N\int_{\{|u_n|<k\}}
  A_j(u_n)\,\partial_j u_n \,dx-\langle \mathfrak{B} u_n,T_k (u_n)\rangle. 
\end{equation} 
We have $\partial_j T_k(u_n)=\chi_{\{|u_n|< k\}} \,\partial_j u_n$ a.e. in $ \Omega$ for $1\leq j\leq N$. 
By the sign-condition of $\Phi$ in \eqref{cond1}, we see that $\Phi(u_n) \,T_k (u_n)=\Phi(u_n) \,u_n\geq 0$ on $\{|u_n| < k\}$. 
Since $\|f_n\|_{L^1(\Omega)}\leq \|f\|_{L^1(\Omega)}$, 
by taking $v=T_k(u_n)\in  
W_0^{1,\overrightarrow{p}}(\Omega)\cap L^\infty(\Omega)$ in \eqref{baaa}, we find that 
\begin{equation} \label{vio}  
K_{n,k} +k\int_{\{|u_n|\geq k\}} |\Phi(u_n)|\,dx\leq 
C_0 + I_{u_n}(T_k(u_n)),
\end{equation}
where $C_0:=k \left(\|f\|_{L^1(\Omega)}+C_\Theta \,{\rm meas}\,(\Omega)\right)$. 
Lemma~\ref{sima} gives that
for every $\rho>0$, there exists a constant $C_\rho>0$ such that \eqref{fik} holds for all $n\geq 1$.
Using \eqref{fik} into \eqref{vio}, we find that
\begin{equation} \label{ihu}
K_{n,k} +
\left(k-\rho\right) \int_{\{|u_n|\geq k\}} |\Phi(u_n)| \,dx
\leq   C_0+ \rho \sum_{j=1}^N  \|\partial_j u_n\|^{p_j}_{L^{p_j}(\Omega)}+ C_\rho.
\end{equation}
We fix $0<\rho<\min\,\{1, \varepsilon \nu_0\}$. Hence, using \eqref{knk}, \eqref{etc}, our choice of $k$ and 
the coercivity condition in \eqref{ellip}, we derive that
\begin{equation*} 
 \nu_0  \sum_{j=1}^N  \|\partial_j u_n\|^{p_j}_{L^{p_j}(\Omega)}-\langle \mathfrak{B} u_n,T_k(u_n)\rangle\leq  
C_0+ \rho \sum_{j=1}^N  \|\partial_j u_n\|^{p_j}_{L^{p_j}(\Omega)}+ C_\rho.
\end{equation*}
By the choice of $\rho$ and \eqref{exb}, we conclude 
the boundedness of $\{u_n\}_{n\geq 1}$ in $W_0^{1,\overrightarrow{p}}(\Omega)$. Since $\mathfrak B$ is a bounded operator from $W_0^{1,\overrightarrow{p}}(\Omega)$ into its dual, we have 
$|\langle \mathfrak{B}u_n,T_k (u_n)\rangle|\leq C$, where $C$ is a positive constant independent of $n$. Thus, from \eqref{ihu}, we readily deduce that
\begin{equation*} \label{nuc1} 
 \int_{\{|u_n|\geq k\}} |\Phi(u_n)|\,dx \leq C.
\end{equation*} 
Using the growth condition of $\Phi$ in \eqref{cond1}, we find that
$ \int_{\{|u_n|<k\}} |\Phi(u_n)|\,dx \leq C$ for all $ n\geq 1$.  
This completes the proof of \eqref{cas}. 

\medskip 
$(b)$ Up to a subsequence, the assertion  in \eqref{wwa} follows from \eqref{cas}. 
\end{proof}

\subsection{Strong convergence of $T_k(u_n)$} Our aim in this section is to prove the following Proposition~\ref{st-con}. 

\begin{prop} \label{st-con} Up to a subsequence of $\{u_n\}_n$, we have
\begin{equation} \label{hcx}  \nabla u_n\to \nabla u_0\ \mbox{a.e. in } \Omega\ \mbox{and } 
T_k (u_n)\to T_k (u_0)\ \mbox{(strongly) in } W_0^{1,\overrightarrow{p}}(\Omega)\ \mbox{as } n\to \infty 
\end{equation}
 for every positive integer $k$. 
\end{prop}

\begin{rem} \label{om1} We have $\Phi(u_0)\in L^1(\Omega)$. Indeed, using the a.e. convergences of $\{u_n\}$ and $\{\nabla u_n\}$ in \eqref{wwa} and \eqref{hcx}, respectively, we obtain that 
$|\Phi(u_n)|\to |\Phi(u_0)|$ a.e. in $\Omega$ as $n\to \infty$. Then, the claim follows from \eqref{cas} and Fatou's Lemma.
\end{rem}

To derive \eqref{hcx}, we can proceed as in the proof of Lemma 4.2 in \cite{baci}. 
However, the new ingredient here is Lemma~\ref{lema74}, which 
is due to the introduction of $\Psi$ in \eqref{eq1}. 

We define $Q_j(n,k)$, $R_{j}(n,k)$, $V_j(n,k)$ and $W_j(n,k)$ as follows
\begin{equation} \label{crav}
 \begin{aligned} 
& Q_j(n,k):=\int_{\{u_n\geq  k\}}  |u_n|^{\theta_j-1} |\partial_j u_n|^{q_j} (k-T_k(u_0))\,dx && \mbox{for } 1\leq j\leq N,&\\
& R_j(n,k):= \int_{\{u_n\leq  -k\}} |u_n|^{\theta_j-1} |\partial_j u_n|^{q_j} (k+T_k(u_0))\,dx && \mbox{for } 1\leq j\leq N,&\\
& V_j(n,k):=\int_{\{ 0<|u_n|< k\}} 
|\partial_j u_n|^{q_j} |u_n-T_k(u_0)|\,dx && \mbox{if } j\in J_1,&\\
& W_j(n,k):=\int_{\{T_k(u_0) <u_n < k\}} |u_n|^{\theta_j-1} |\partial_j u_n|^{q_j} (u_n-T_k(u_0))\,dx&& \mbox{if } j\in J_2.&
\end{aligned} \end{equation}

Let $\varphi_\lambda$ be as in the proof of Lemma~\ref{res2} (see \eqref{lam}). For every $n,k\geq 1$, we set
\begin{equation} \label{nubb} Z_{n,k}:=T_k(u_n)-T_k(u_0).\end{equation}

\begin{lemma} \label{lema74} We have 
\begin{equation} \label{hub} \limsup_{n\to \infty} I_{u_n} (\varphi_\lambda(Z_{n,k}))\leq 0.
\end{equation}
\end{lemma}

\begin{proof} Since $\varphi_\lambda(Z_{n,k})=Z_{n,k} \,e^{\lambda (Z_{n,k})^2}$, from \eqref{ass} we find that
$$   \begin{aligned}
I_{u_n} (\varphi_\lambda(Z_{n,k}))=& \sum_{j=1}^N \int_{\{u_n\geq k\}} |u_n|^{\theta_j-1} |\partial_j u_n|^{q_j} (k-T_k(u_0))\, e^{\lambda (Z_{n,k})^2}\,dx\\
&+  \sum_{j=1}^N\int_{\{u_n\leq  -k\}}  |u_n|^{\theta_j-1} |\partial_j u_n|^{q_j} (k+T_k(u_0))\, e^{\lambda (Z_{n,k})^2}\,dx\\
&+ \sum_{j=1}^N\int_{\{0<|u_n|< k\}}  |u_n|^{\theta_j-2} u_n  |\partial_j u_n|^{q_j} (u_n-T_k(u_0))\, e^{\lambda (Z_{n,k})^2}\,dx.
\end{aligned} $$

Since $|Z_{n,k}|\leq 2k$ a.e. in $\Omega$, using \eqref{crav}, we infer that
\begin{equation} \label{cbb} \frac{I_{u_n} (\varphi_\lambda(Z_{n,k}))}{e^{4\lambda k^2}}\leq   \sum_{j=1}^N (Q_j(n,k)+R_j(n,k))+\sum_{j\in J_1} k^{\theta_j-1} V_j(n,k)+
\sum_{j\in J_2} W_j(n,k).
\end{equation} We separate the case $j\in J_2$ from $j\in J_1$.

\vspace{0.2cm}
(I) Let $j\in J_2$, which pertains to Case 2 when $u_n\geq 0$ a.e. in $\Omega$ and, hence, $u_0\geq 0$ a.e. in $\Omega$. 
We remark that $(k-T_k(u_0)) \chi_{\{u_n\geq  k\}} \to 0$ in $L^{(p_j/q_j)'}(\Omega)$ as $n\to \infty$. Since $\{ \partial_j u_n\}_{n\geq 1}$ is bounded in 
$L^{p_j}(\Omega)$, by H\"older's inequality, we infer that 
$$ 0\leq Q_j(n,k)\leq k^{\theta_j-1}  \|\partial_j u_n\|_{L^{p_j}(\Omega)}^{q_j} \| (k-T_k(u_0)) \chi_{\{u_n\geq k\}}\|_{L^{(p_j/q_j)'}(\Omega)}\to 0\quad \mbox{as } n\to \infty.
$$
With a similar argument, we obtain that $\lim_{n\to \infty}R_j(n,k)=0$. 

Let $0<\tau<\min_{i\in J_2} \theta_i$. Hence, $\tau\in (0,1)$ and $|u_n|^{\theta_j-1} (u_n-T_k(u_0))\leq k^{\theta_j-\tau}(u_n-T_k(u_0))^\tau$ on the set
$\{ T_k(u_0)\leq u_n\leq k\}$. Since $(u_n-T_k(u_0))^\tau \chi_{\{ T_k(u_0)\leq u_n\leq k\}}\to 0$ in $L^{(p_j/q_j)'}(\Omega)$ as $n\to \infty$, proceeding as above, we obtain that
$$ 0\leq W_j(n,k)\leq k^{\theta_j-\tau} \|\partial_j u_n\|_{L^{p_j}(\Omega)}^{q_j} \| (u_n-T_k(u_0))^\tau \chi_{\{ T_k(u_0)\leq u_n\leq k\}}\|_{L^{(p_j/q_j)'}(\Omega)}\to 0\ \mbox{as } n\to \infty.
$$

(II) We now assume that $j\in J_1$ and  we show that
\begin{equation} \label{qes}
 Q_j(n,k)\to 0\quad \mbox{as } n\to \infty.
 \end{equation} 
As in the proof of Lemma~\ref{sima}, we distinguish several situations:

\vspace{0.2cm}
Case (a). Let $q_j=0$ and $\theta_j\geq p+1$. In this case, $j\in N_{\overrightarrow{\mathfrak{a}}}$ and hence, \eqref{m} yields that 
$m> \theta_j$. 

Case (b).  Let $q_j=0$ and $\theta_j< p+1$. 

In Cases (a) and (b) above, we define $\gamma_j=1-(\theta_j-1)/r$, where $r=m-1$ in Case (a) and $r=p$ in Case (b). 
Then, by H\"older's inequality and \eqref{cas}, we have
$$ 0\leq Q_j(n,k)\leq  \|u_n\|_{L^r(\Omega)}^{\theta_j-1} \| (k-T_k(u_0))\,\chi_{\{u_n\geq k\}} \|_{L^{1/\gamma_j}(\Omega)}\to 0\  \mbox{as } n\to \infty.
$$

For Cases (c) and (d) below, we define $\zeta_j$ as in \eqref{zetag}. 

Case (c) Let $q_j>0$ and $\zeta_j\leq 0$. Defining $\gamma_j=1-q_j/p_j$, similar to \eqref{yy}, we get 
\begin{equation} \label{ron1} 
Q_j(n,k)\leq 
 \|\partial_j u_n\|_{L^{p_j}(\Omega)}^{-\frac{p_j\zeta_j}{m-1}}
\left( \mathfrak I_{m-1,p_j}(u_n)
\right)^{\frac{\theta_j-1}{m-1}} 
\|(k-T_k(u_0))\,\chi_{\{u_n\geq k\}} \|_{L^{1/\gamma_j}(\Omega)}.
\end{equation} 

\vspace{0.2cm}
Case (d) Let $q_j>0$ and $\zeta_j>0$. We have three sub-cases, see (d$_1$)--(d$_3$) in Lemma~\ref{sima}. 

\begin{itemize}
\item[(d$_1$)] Let $\mathfrak{m}_j>1$ and $\theta_j\geq p$. Defining $\gamma_j=(m-\theta_j)/(m-1)$, 
similar to \eqref{figo}, we see that  
\begin{equation} \label{ron2}
Q_j(n,k)\leq \left(
\mathfrak I_{m-1}(u_n)\right)^{\frac{\zeta_j}{m-1}} 
 \left(\mathfrak I_{m-1,p_j}(u_n)
\right)^{\frac{q_j}{p_j}} \|(k-T_k(u_0))\,\chi_{\{u_n\geq k\} } \|_{L^{1/\gamma_j}(\Omega)}.
\end{equation} 
\end{itemize}

We treat the remaining sub-cases (d$_2$) and (d$_3$) together and define $\gamma_j$ as in \eqref{gem}. 
Analogous to \eqref{ttt}, we find that 
\begin{equation} \label{ron3}
Q_j(n,k)\leq  \|u_n\|_{L^{p}(\Omega)}^{\zeta_j}
\left(\mathfrak I_{m-1,p_j}(u_n)
\right)^{\frac{q_j}{p_j}} \| (k-T_k(u_0))\,\chi_{\{u_n\geq k\} }\|_{L^{1/\gamma_j}(\Omega)}. 
\end{equation}

By \eqref{cas}, the right-hand side of each of the inequalities in \eqref{ron1}, \eqref{ron2} and \eqref{ron3} converges to $0$ as $n\to \infty$. 
So, in any of the Cases (a)--(d), we get \eqref{qes} for $j\in J_1$. With the same reasoning, we obtain that $\lim_{n\to \infty} R_j(n,k)= 0$ for every $j\in J_1$. Using that $(u_n-T_k(u_0)) \chi_{\{ 0<|u_n|< k\}}\to 0$ in $L^{(p_j/q_j)'}(\Omega)$ as $n\to \infty$, we find that $V_j(n,k)\to 0$ as $n\to \infty$.
Thus, the right-hand side of \eqref{cbb} converges to $0$ as $n\to \infty$. The proof of \eqref{hub} is complete. 
\end{proof}

\noindent {\bf Proof of Proposition~\ref{st-con}.} 
Using Lemma~A.5 in \cite{baci}, to obtain \eqref{hcx}, it suffices to show that for every integer $k\geq 1$, there exists a subsequence of $\{u_n\}$ (depending on $k$ and relabeled 
$\{u_n\}$) such that \eqref{alt} holds. 
 We first note that 
\begin{equation} \label{mimi}
 \mathcal E_{u_n}(T_k(u_n),T_k(u_0))\, \chi_{\{|u_n|\geq  k\}}\to 0\ \mbox{ (strongly) in } L^1(\Omega)\quad \mbox{ as } n\to \infty.\end{equation} 
 Indeed, from \eqref{7c} and \eqref{nubb}, we have 
 $$ \mathcal E_{u_n}(T_k(u_n),T_k(u_0))=\sum_{j=1}^N E_{j,u_n}(T_k(u_n),T_k(u_0))\,\partial_j Z_{n,k}. 
 $$ 
For all $1\leq j\leq N$, since $\partial_j T_k(u_n)=\chi_{\{|u_n|<k\}}\partial_j u_n $, the Dominated Convergence Theorem yields
\begin{equation} \label{dom} \partial_j Z_{n,k}\,\chi_{\{|u_n|\geq k\}}=-\partial_j u_0\, \chi_{\{|u_n|\geq k\}} \chi_{\{|u_0|<k\}}\to 0\quad \mbox{(strongly) in } L^{p_j}(\Omega)\ \mbox{as } n\to \infty.
\end{equation}
Similar to \eqref{wa}, by passing to a subsequence of $\{u_n\}$, for each $1\leq j\leq N$, we see that $\{ E_{j,u_n}(T_k(u_n),T_k(u_0))\}_{n}$ converges weakly in $L^{p_j'}(\Omega)$ as $n\to \infty$. 
Hence, we obtain \eqref{mimi}. 

Using the monotonicity assumption in \eqref{ellip}, we get that $  \mathcal E_{u_n}(T_k(u_n),T_k(u_0))\geq 0$ a.e. in $\Omega$. Hence, in view of \eqref{mimi}, to conclude \eqref{alt}, it remains to show that, up to a subsequence, 
\begin{equation} \label{erc2} \limsup_{n\to \infty} 
 \int_{\{|u_n|<k\}} \mathcal E_{u_n}(T_k(u_n),T_k(u_0))\,dx\leq 0.
\end{equation}
We set
\begin{equation} \label{obb} \begin{aligned}
& f_\lambda(n,k):=\varphi_\lambda'(Z_{n,k})-\frac{\phi(k)}{\nu_0} |\varphi_\lambda (Z_{n,k})|,  \\
& \mathcal F_{n,k}(v):= \sum_{j=1}^N \int_{\{|u_n|<k\}}  A_j(x,u_n,\nabla v) \,f_{\lambda}(n,k)\,\partial_j Z_{n,k} \, dx,\ \mbox{where } v\in W_0^{1,\overrightarrow{p}}(\Omega).
\end{aligned} \end{equation}
Since $T_k(u_n)=u_n$ on $\{|u_n|<k\}$, from \eqref{bea} and the definition of $\mathcal E_u$ in \eqref{muz}, we infer that 
$$ \frac{1}{2}  \int_{\{|u_n|<k\}} \mathcal E_{u_n}(T_k(u_n),T_k(u_0))\,dx\leq \mathcal F_{n,k}(u_n)-\mathcal F_{n,k}(T_k(u_0)).
$$

The proof of \eqref{erc2} follows now by establishing that 
\begin{equation} \label{imz}
{\rm (i)}\ \ \lim_{n\to \infty} \mathcal F_{n,k}(T_k(u_0))=0,\quad
{\rm (ii)}\ \  \limsup_{n\to \infty} \mathcal F_{n,k}(u_n)
 \leq 0.
\end{equation} 
Since $|Z_{n,k}|\leq 2k$, we find a constant $C_k>0$ such that 
$ |f_\lambda(n,k)| \leq C_k$ for all $n\geq 1$. For arbitrary $1\leq j\leq N$,  
with the same reasoning as for \eqref{unt}, we have that $A_j(x,u_n,\nabla T_k(u_0)) \,\chi_{\{|u_n|\leq  k\}}$ converges to 
$A_j(x,u_0,\nabla T_k(u_0)) \,\chi_{\{|u_0|\leq  k\}}$ (strongly) in $L^{p_j'}(\Omega)$ as $n\to \infty$. Hence, using that $\partial_j Z_{n,k}\rightharpoonup 0$ (weakly) in $L^{p_j}(\Omega)$ as $n\to \infty$, we 
find that
$ A_j(x,u_n,\nabla T_k(u_0)) \,\chi_{\{|u_n|\leq  k\}} \,\partial_j Z_{n,k}\to 0$ in $ L^1(\Omega)$ as $ n\to \infty$. 
Thus, by the squeeze law, we obtain the first limit in \eqref{imz}.

To prove (ii) in \eqref{imz}, we take as a test function in \eqref{baaa} the function 
\begin{equation} \label{vpl} v=\varphi_\lambda(Z_{n,k})\in W_0^{1,\overrightarrow{p}}(\Omega)\cap L^\infty(\Omega).\end{equation}
Compared with \cite{baci}, we have the extra term $I_{u_n}(v)$ in the right-hand side of \eqref{baaa}. Then, $I_{u_n} (\varphi_\lambda(Z_{n,k}))$ is the additional term which appears when bounding from above $\mathcal F_{n,k}(u_n)$. By following the ideas in the proof of Lemmata 3.2 and 4.2 in \cite{baci} (see Lemma~\ref{occ} in the Appendix for details), we arrive at
\begin{equation} \label{nica}
\mathcal F_{n,k}(u_n)\leq S_k(n)+
I_{u_n}(\varphi_\lambda (Z_{n,k})),
\end{equation}
where, up to a subsequence of $\{u_n\}$, $\lim_{n\to \infty} S_k(n)= 0$. 
From Lemma~\ref{lema74} and \eqref{nica}, we conclude (ii) in \eqref{imz}. This 
ends the proof of Proposition~\ref{st-con}.
\qed

\subsection{Proof of Theorem~\ref{teo2} concluded} 
Here, we obtain that $u_0\in W_0^{1,\overrightarrow{p}}(\Omega) $ is a solution of \eqref{eq1} 
by combining Propositions~\ref{lem-ad} and \ref{st-con} with Lemma~\ref{lemi} below.

\begin{lemma} \label{lemi}
Let $u_n$ and $u_0$ be as in Proposition~\ref{lem-ad}. Then, 
up to a subsequence, we have
\begin{equation}  
I_{u_0}(v)=\displaystyle \lim_{n\to \infty} I_{u_n} (v)=
S_{u_0,\Theta,f}(v)
-\langle\mathfrak{B} u_0, v\rangle  \label{fox2}
\end{equation}  	
for every $v\in W_0^{1,\overrightarrow{p}}(\Omega)\cap L^\infty(\Omega)$. 
\end{lemma}

\begin{proof} We start by proving the second equality in \eqref{fox2}. Let $v\in W_0^{1,\overrightarrow{p}}(\Omega)\cap L^\infty(\Omega)$ be arbitrary. From \eqref{nice}, we have 
$\lim_{n\to \infty} \int_\Omega f_n v\,dx= \int_\Omega fv\,dx$. 
Reasoning as in the proof of \eqref{lli}, we obtain 
$ \Theta(u_n)\, v\to \Theta(u_0)\, v$ in $L^1(\Omega)$ as $n\to \infty$, $\lim_{n\to \infty} \langle \mathcal A (u_n) ,v\rangle=\langle \mathcal A (u_0) ,v\rangle$ and 
$\lim_{n\to \infty}  \langle \mathfrak{B} u_n, v\rangle = \langle \mathfrak{B} u_0, v\rangle$. 
Since $\Phi(u_0)\in L^1(\Omega)$ (see Remark~\ref{om1}),  
it is enough to show that 
\begin{equation} \label{yem} \Phi(u_n)\to \Phi(u_0) \ \ \mbox{(strongly) in } L^1(\Omega)\ \mbox{as } n\to \infty. 
\end{equation} By Vitali's Theorem, it suffices to show the uniform integrability of 
$\{\Phi(u_n)\}_{n\geq 1}$ over $\Omega$. 

Fix $M>2$ arbitrary. Let $\omega$ be 
any measurable subset  of $\Omega$.  
We regain \eqref{net} with $u_n$ instead of $U_n$. 
However, the proof of \eqref{new1} does not translate here since from Proposition~\ref{lem-ad} we only have 
the uniform boundedness in $L^1(\Omega)$ for
$\{\Phi(u_n)\}_{n\geq 1}$ (rather than
for $\{\Phi(u_n)\,u_n\}_{n\geq 1}$). The case $\Psi=0$ is treated in \cite{baci}*{Lemma 4.3} by adapting and extending to the anisotropic case an approach from \cite{BGM}. 
We give the details since compared with \cite{baci} we need to deal with the new term $\Psi$ in \eqref{ass}. In \eqref{baaa}, we take
$$v=T_1(G_{M-1}(u_n)) \in W_0^{1,\overrightarrow{p}}(\Omega)\cap L^\infty(\Omega).
$$
 Then, using the coercivity condition in \eqref{ellip} and \eqref{cond1}, we obtain the estimate 
\begin{equation} \label{vio2}  
\begin{aligned}
\int_{\{|u_n|> M\}} |\Phi(u_n)|\,dx\leq & \int_{\{|u_n|\geq M-1\}} (|f_n|+C_\Theta)\,dx+
|\langle \mathfrak{B} u_n, T_1(G_{M-1} (u_n))\rangle |\\
&+ I_{u_n}(T_1(G_{M-1}(u_n)) ).
\end{aligned}\end{equation}
Now, up to a subsequence of $\{u_n\}$, from \eqref{wwa}, we have
$$ T_1(G_{M-1} (u_n))\rightharpoonup  
T_1(G_{M-1} (u_0))\ \mbox{(weakly) in } W_0^{1,\overrightarrow{p}}(\Omega)\ \mbox{as }n\to \infty.
$$ Using this in \eqref{vio2}, jointly with \eqref{nice} and the property $(P_2)$ for $\mathfrak{B}$, we find that 
\begin{equation} \label{foz} \begin{aligned}  \limsup_{n\to \infty} \int_{\{|u_n|>M\}} |\Phi(u_n)| \,dx\leq &
\int_{\{|u_0|\geq M-1\}} (|f|+C_\Theta) \,dx+|\langle \mathfrak B u_0,T_1(G_{M-1} (u_0))\rangle|\\
&+\limsup_{n\to \infty}  |I_{u_n}(T_1(G_{M-1}(u_n)) )|.
\end{aligned}
\end{equation}

Since $T_1(G_{M-1}(u_n))=0$ on $\{|u_n|\leq M-1\}$, 
we have
\begin{equation} \label{fan} |I_{u_n}(T_1(G_{M-1}(u_n)) )| \leq \sum_{j=1}^N \int_{\{|u_n|\geq M-1\}} |u_n|^{\theta_j-1} |\partial_j u_n|^{q_j}\,dx.
\end{equation}
Let $\mu_{M-1}(v):=\mbox{meas}\,\{|v|\geq M-1\}$. 
We next bound from above the right-hand side of \eqref{fan}. 

\vspace{0.2cm}
(I) For every $j\in J_2$, using that $M>2$ and $\theta_j\leq 1$, we find that
\begin{equation} \label{ron0}  \int_{\{|u_n|\geq M-1\}} |u_n|^{\theta_j-1} |\partial_j u_n|^{q_j}\,dx\leq (M-1)^{\theta_j-1}
\|\partial_j u_n\|^{q_j}_{L^{p_j}(\Omega)}\left(\mu_{M-1}(u_n)\right)^{1-q_j/p_j}.
\end{equation}

(II) Let $j\in J_1$ corresponding to $\theta_j>1$. 
We are guided by the reasoning in Lemma~\ref{lema74}. 
In relation to the upper bound for 
$Q_j(n,k)$ in the proof of \eqref{qes}, we replace $(k-T_k(u_0)) \chi_{\{u_n\geq k\}}$ by $\chi_{\{|u_n|\geq M-1\}}$. Hence, using also \eqref{cas}, we obtain a positive constant $C$, independent of $n$ and $M$, such that
\begin{equation} \label{roni} 
 \int_{\{|u_n|\geq M-1\}} |u_n|^{\theta_j-1} |\partial_j u_n|^{q_j}\,dx\leq C\,  \left( \mu_{M-1}(u_n)\right)^{\gamma_j},
\end{equation} 
where $\gamma_j\in (0,1)$ is defined according to (a)--(d) in the proof of \eqref{qes}.

\vspace{0.2cm}
\noindent In light of \eqref{ron0} and \eqref{roni}, we infer from \eqref{fan} that
$$ \limsup_{n\to \infty}  |I_{u_n}(T_1(G_{M-1}(u_n)) )| \leq C \left(\sum_{j\in J_2} (\mu_{M-1} (u_0))^{1-q_j/p_j}+ \sum_{j\in J_1}  \left( \mu_{M-1}(u_0)\right)^{\gamma_j} \right),
$$ where $C>0$ is a constant independent of $M$. As $\mu_{M-1}(u_0)\to 0$ as $M\to \infty$, by choosing $M>2$ large, we can make 
$ \limsup_{n\to \infty}  |I_{u_n}(T_1(G_{M-1}(u_n)) )|$ as small as desired. Using this fact in \eqref{foz}, we conclude that 
$ \int_\omega |\Phi(u_n)| \chi_{\{ |u_n|>M\}}\,dx
$ is small uniformly in $n$ and $\omega$. This finishes the proof of \eqref{yem}. 

\vspace{0.2cm}
We now establish the first equality in \eqref{fox2} for every  $v \in W_0^{1,\overrightarrow{p}}(\Omega)\cap L^\infty(\Omega)$. 
We follow the ideas in the proof of \eqref{lima}, working here with $\Psi$, $u_n$ and $u_0$ instead of $\Psi_n$, $U_n$ and $U_0$, respectively. 
Hence, for every $j\in J_1$, the reader should replace 
$H_{j,n}(U_n,\partial_j U_n)$ by $|u_n|^{\theta_j-2} u_n |\partial_j u_n|^{q_j}$. 

\noindent For every $j\in J_1$, corresponding to \eqref{lsmvbf}, we want to show that there exists $s_j>1$ such that 
\begin{equation} \label{suv} \| |u_n|^{\theta_j-2} u_n |\partial_j u_n|^{q_j}\|_{L^{s_j}(\Omega)}\leq C
\end{equation} for a positive constant $C$ independent of $n$. 
We need to adjust the argument given in Section~\ref{sec-c1}. The reason is that instead of $\{\|u_n\|_{L^m(\Omega)}\}_{n\geq 1}$ and
$\{\mathfrak I_{m,p_j}(u_n)\}_{n\geq 1}$ being uniformly bounded in $n$, we only have that $\{\int_{\Omega} |u_n|^{m-1}\,dx\}_{n\geq 1}$ and $\{\mathfrak I_{m-1,p_j}(u_n)\}_{n\geq 1}$ are uniformly bounded (see Proposition~\ref{lem-ad} (a)). With similar ideas to those given in the proof of \eqref{qes}, based on H\"older's inequality, we obtain \eqref{suv}  
by taking $s_j=1/(1-\gamma_j)$, where $\gamma_j\in (0,1)$ is defined as for \eqref{roni}. 

\noindent Now, we use Proposition~\ref{lem-ad} (b) and \eqref{hcx} to deduce that 
$$ |u_n|^{\theta_j-2} u_n |\partial_j u_n|^{q_j}\to |u_0|^{\theta_j-2} u_0 |\partial_j u_0|^{q_j}\quad \mbox{a.e. in } \Omega\ \mbox{as } n\to \infty.
$$
Using \eqref{suv}, we infer that, up to a subsequence, $|u_n|^{\theta_j-2} u_n |\partial_j u_n|^{q_j}\rightharpoonup 
|u_0|^{\theta_j-2} u_0 |\partial_j u_0|^{q_j}$ (weakly) in $L^{s_j}(\Omega)$ as $n\to \infty$, proving that 
\begin{equation} \label{noc} \sum_{j\in J_1} \int_\Omega  |u_n|^{\theta_j-2} u_n |\partial_j u_n|^{q_j} v\,dx=
 \sum_{j\in J_1} \int_\Omega  |u_0|^{\theta_j-2} u_0 |\partial_j u_0|^{q_j} v\,dx.
\end{equation}
As mentioned before, for $w\in W_0^{1,\overrightarrow{p}}(\Omega)\cap L^\infty(\Omega)$, we have $\nabla w=0$ a.e. in $\{w=0\}$. Hence, the above identity holds if instead of   
$\Omega$ we put $\{|u_n|>0\}$ in the left-hand side of \eqref{noc} and $\{|u_0|>0\}$ in the right-hand side. 
This completes the proof of \eqref{fox2} in Case 1.

In  Case 2, the proof of \eqref{fox2} adapts almost verbatim from Section~\ref{sec-th2} remembering to work with $\Psi$ instead of $\Psi_n$.
This ends the proof of Lemma~\ref{lemi}. \end{proof}

\appendix
\section{}

In this paper, we need the following version of Young's inequality. 

\begin{lemma}[Young's inequality] \label{young} Let $N\geq 2$ be an integer. Assume that $\beta_1,\ldots, \beta_N$ are positive numbers
	and $1<R_k<\infty$ for each $1\leq k\leq N-1$. If $\sum_{k=1}^{N-1} (1/R_k)<1$, then for every $\delta>0$, there exists
	a positive constant $C_\delta$ (depending on $\delta$) such that 
	$$ \prod_{k=1}^N \beta_k\leq \delta\sum_{k=1}^{N-1} \beta_k^{R_k} +C_\delta \,\beta_N^{R_N},  
	$$ where we define $R_N=\left[1-\sum_{k=1}^{N-1} (1/R_k)\right]^{-1}$. 
\end{lemma}

We recall the anisotropic Sobolev inequality in \cite{T}*{Theorem~1.2}.

\begin{lemma} \label{sob} Let $N\geq 2$ be an integer. If $1 < p_j <\infty$ for every $1\le j \le N$ and $p<N$,  
then there exists a
positive constant $\mathcal{S}=\mathcal{S}(N,\overrightarrow{p})$, such that 
\begin{equation} \label{send} 
\|u\|_{L^{p^\ast}(\R^N)}\leq \mathcal{S}  \prod_{j=1}^N \| \partial_j  u\|_{L^{p_j}(\R^N)}^{1/N} \quad \mbox{for all } u\in C_c^\infty(\R^N),
\end{equation}
where, as usual, $p^\ast:=Np/(N-p)$.
\end{lemma}

\begin{rem} \label{an-sob} Let $\Omega$ be a bounded, open subset of $\R^N$ with $N\geq 2$. If  $1 < p_j <\infty$ for every $1\le j \le N$ and $p<N$, then by a density argument, 
\eqref{send} extends to all $u\in  W_0^{1,\overrightarrow{p}}(\Omega) $ so that the arithmetic-geometric mean inequality yields
\begin{equation}\label{ASI}
\|u\|_{L^{p^\ast}(\Omega)}\leq \mathcal{S}  \prod_{j=1}^N \| \partial_j  u\|_{L^{p_j}(\Omega)}^{1/N} \leq 
\frac{\mathcal{S}}{N} \sum_{j=1}^N \| \partial_j  u\|_{L^{p_j}(\Omega)}=\frac{\mathcal{S}}{N}\|u\|_{W_0^{1,\overrightarrow{p}}(\Omega)}
\end{equation}
for all $u\in W_0^{1,\overrightarrow{p}}(\Omega)$. Moreover, 
using H\"older's inequality, the embedding $W_0^{1,\overrightarrow{p}}(\Omega)\hookrightarrow L^s(\Omega)$ is continuous 
for every
$s\in [1,p^\ast]$ and compact for every $s\in [1,p^\ast)$.
\end{rem}

\begin{lemma} \label{cori} Let the assumptions of Proposition~\ref{red} hold. 
Suppose that \begin{equation} \label{amil1} \mathcal E_{U_n}(U_n^\pm,U_0^\pm)\to 0\ \mbox{ in } L^1(\Omega)\ \mbox{ as }  n\to \infty.
\end{equation}  
Then, up to a subsequence, we have
\begin{eqnarray} 
&\nabla U_n^\pm\to \nabla U_0^\pm\ \mbox{a.e. in }  \Omega\ \mbox{as } n\to \infty, \label{qmz}\\
 & U_n^\pm\to U_0^\pm\ \mbox{(strongly) in } W_0^{1,\overrightarrow{p}}(\Omega)\ \mbox{as } n\to \infty. \label{hsn21}
\end{eqnarray}
\end{lemma}

\begin{proof}
From \eqref{amil1}, we see that, up to a subsequence, 
$\mathcal E_{U_n}(U_n^\pm,U_0^\pm)\to 0$ a.e. in $\Omega$ as  $n\to \infty$.

We prove \eqref{qmz}. Recall from Proposition~\ref{red} that \eqref{uni} holds. Let $Z$ be a subset of $\Omega$ with $\mbox{meas}\,(Z)=0$ such that for every $x\in \Omega\setminus Z$, we have $|U_0^\pm(x)|<\infty$, $|\nabla U_0^\pm(x)|<\infty$,  
$|\eta_j(x)|<\infty$ for all $1\leq j\leq N$, as well as 
\begin{equation} \label{opak1} 
U_n(x)^\pm\to U_0^\pm(x), \quad \mathcal{E}_{U_n}(U_n^\pm,U_0^\pm)(x)\to 0\ \mbox{as }n\to \infty.
\end{equation} 
We fix $x\in \Omega\setminus Z$.  
By the monotonicity and coercivity assumptions in \eqref{ellip}, we find that 
\begin{equation} \label{mica}  
 \mathcal E_{U_n}(U_n^\pm,U_0^\pm)\geq \nu_0 \sum_{j=1}^N |\partial_j  U_n^\pm(x)|^{p_j}-
 \sum_{j=1}^N  |A_{j}(x,U_n,\nabla U_n^\pm)\, \partial_j U_0^\pm
+A_j(x,U_n,\nabla U_0^\pm) \,\partial_j U_n^\pm |.
\end{equation}    
By Young's inequality, for every $\delta>0$, there exists $C_\delta>0$ such that 
\begin{equation} \label{abi1m1} \begin{aligned}
&  |A_{j}(x,U_n,\nabla U_n^\pm)\, \partial_j U_0^\pm| \leq 
 \delta\,  |A_j(x,U_n,\nabla  U_n^\pm)|^{p_j'}+C_\delta  |\partial_j U_0^\pm(x)|^{p_j}\\
& 
 |A_j(x,U_n,\nabla U_0^\pm) \,\partial_j U_n^\pm|
\leq 
\delta\,  |\partial_j U_n^\pm(x)|^{p_j} +C_\delta |A_j(x,U_n,\nabla U_0^\pm)|^{p_j'}
\end{aligned}
\end{equation} for every $1\leq j\leq N$. 
We use the growth condition  in \eqref{ellip} to bound from above the right-hand side of each inequality in \eqref{abi1m1}. Then, there exist positive constants $C$ and 
$C'_\delta$, both independent of $n$ (only $C'_\delta$ depends on $\delta$) such that $ \mathcal E_{U_n}(U_n^\pm,U_0^\pm)(x)$ is bounded below by 
\begin{equation*} \label{abi2k1}   (\nu_0-C\delta) \sum_{j=1}^N |\partial_j U_n^\pm(x)|^{p_j}- C'_\delta \left(\sum_{j=1}^N  \eta_j^{p_j'}(x)+|U_n(x)|^{p^\ast} +
\sum_{j=1}^N |\partial_j U_0^\pm(x)|^{p_j}\right).
\end{equation*} 
Using \eqref{opak1} and choosing $\delta\in (0,\nu_0/C)$, 
we conclude that 
\begin{equation} \label{bddk1} 
\{|\nabla U_n^\pm(x)|\}_{n}\ \text{ is uniformly bounded with respect to } n. 
\end{equation} 
Let $x\in \Omega\setminus Z$ be arbitrary. To prove that $\nabla U_n^+(x)\to \nabla U_0^+(x)$ as $n\to \infty$, we show that any accumulation point $\Xi$ of $\{\nabla U_n^+(x)\}_{n}$ 
coincides with $\nabla U_0^+(x)$. From \eqref{bddk1}, we have $|\Xi|<\infty$. 
By \eqref{opak1} and the continuity of $ A_j(x,\cdot,\cdot)$ with respect to the last two variables, we get 
$$
 {\mathcal E}_{U_n}(U_n^+,U_0^+)(x)\to  \sum_{j=1}^N \left[A_j(x,U_0(x),\Xi)-  A_j(x,U_0(x),\nabla U_0^+(x))\right] (\Xi_j- \partial_j U_0^+(x))\ \text{as } n\to \infty.   
$$ 
This, jointly with \eqref{opak1} and the monotonicity condition in \eqref{ellip}, gives that 
$\Xi=\nabla U_0^+(x)$. Similarly, we obtain that   
$\nabla U_n^-(x)\to \nabla U_0^-(x)$ as $n\to \infty$. The proof of \eqref{qmz} is complete since $x\in \Omega\setminus Z$ is arbitrary and $\mbox{meas}\,(Z)=0$.

\vspace{0.2cm}
In order to prove \eqref{hsn21}, we use \eqref{qmz}, \eqref{mica}, Lemma~\ref{step3} and Vitali's Theorem. We see that
$\{|\partial_j U_n^\pm-  \partial_j U_0^\pm|^{p_j}\}_n$ is a sequence of non-negative integrable functions, converging to $0$ a.e. on $\Omega$ as $n\to \infty$. So, we conclude \eqref{hsn21} by 
showing that, up to a subsequence, 
$ \left\{\sum_{j=1}^N |\partial_j U_n^\pm|^{p_j}\right\}_{n}$ is uniformly integrable over $ \Omega$. Now, up to a subsequence, we have for each $1\leq j\leq N$,	
\begin{equation} \label{wia} A_j(x,U_n,\nabla U_n^\pm)
\rightharpoonup A_j(x,U_0,\nabla U_0^\pm)\ \mbox{ (weakly) in }L^{p_j'}(\Omega)\ \mbox{as } n\to \infty. \end{equation} 
Indeed, $\{A_j(x,U_n,\nabla U_n^\pm)\}_n$ is bounded in $L^{p_j'}(\Omega)$ from the growth condition in \eqref{ellip} and the boundedness of $\{U_n\}_{n}$ in $W_0^{1,\overrightarrow{p}}(\Omega)$ and, hence, in $L^{p^\ast}(\Omega)$. 
Moreover, $\{A_j(x,U_n,\nabla U_n^\pm)\}_{n}$ converges to $ A_j(x,U_0,\nabla U_0^\pm)$ a.e. in $ \Omega$ as $n\to \infty$ using \eqref{qmz}, the convergence $U_n\to U_0$ a.e. in $\Omega$ (from \eqref{uni}) and the continuity of $A_j(x,\cdot,\cdot)$ in the last two variables. Hence, up to a subsequence, we have \eqref{wia}. 
Consequently, for each $1\leq j\leq N$, we get
\begin{equation} \label{coni11} 
A_{j}(x,U_n,\nabla U_n^\pm)\, \partial_j U_0^\pm \to  A_j(x,U_0,\nabla U_0^\pm)\,\partial_j U_0^\pm\quad \mbox{in } L^1(\Omega)\ \mbox{as }n\to \infty.
\end{equation}

Let $k\geq 1$ be arbitrary. For each $1\leq j\leq N$, we next prove that, as $n\to \infty $, 
\begin{equation} \label{coni111} 
A_j(x,U_n,\nabla U_0^\pm)  \chi_{\{|U_n|\leq k\}}\,\partial_j U_n^\pm\to A_j(x,U_0,\nabla U_0^\pm) \,\chi_{\{|U_0|\leq k\}}\,\,\partial_j U_0^\pm \ \mbox{in } L^1(\Omega).   
\end{equation}
Let $1\leq j\leq N$ be arbitrary. Note that $\{ |A_j(x,U_n,\nabla U_0^\pm)|^{p_j'}  \chi_{\{|U_n|\leq k\}} \}_n 
$ is uniformly integrable over $\Omega$ and $A_j(x,U_n,\nabla U_0^\pm)\chi_{\{|U_n|\leq k\} } \to A_j(x,U_0,\nabla U_0^\pm) \chi_{\{|U_0|\leq k\}}$ a.e. in $\Omega$ as $n\to \infty$. Thus, by Vitali's Theorem, 
$A_j(x,U_n,\nabla U_0^\pm)\,\chi_{\{|U_n|\leq k\} } \to 
A_j(x,U_0,\nabla U_0^\pm) \chi_{\{|U_0|\leq k\}}$ in $L^{p_j'}(\Omega)$ as $n\to \infty$. This proves \eqref{coni111} since 
$\partial_j U_n^\pm\rightharpoonup \partial_j U_0^\pm$ (weakly) in $L^{p_j}(\Omega)$ as $n\to \infty$ (see Remark~\ref{rem32}).

By H\"older's inequality, we get a constant $C>0$ (independent of $k$) such that for all $n\geq 1$,
\begin{equation} \label{bab}\begin{aligned}
\sum_{j=1}^N \int_{\{|U_n|>k\}} |A_j(x,U_n,\nabla U_0^\pm) \,\partial_j U_n^\pm| \,dx&=
\sum_{j=1}^N  \int_{\{U_n^\pm>k\}}  |A_j(x,U_n,\nabla U_0^\pm) \,\partial_j G_k(U_n)| \,dx\\
 &\leq C \sum_{j=1}^N \|\partial_j G_k(U_n)\|_{L^{p_j}(\Omega)}.
 \end{aligned}
\end{equation}

Using \eqref{bab}, jointly with Lemma~\ref{step3}, we get that, up to a subsequence, 
\begin{equation} \label{lamp} \limsup_{n\to \infty} 
\sum_{j=1}^N \int_{\{|U_n|>k\}} |A_j(x,U_n,\nabla U_0^\pm) \,\partial_j U_n^\pm| \,dx\leq C W_k,
\end{equation} for each $k\geq 1$, where $\lim_{k\to \infty} W_k=0$. Using \eqref{amil1}, \eqref{coni11}, \eqref{coni111} and \eqref{lamp}, from \eqref{mica}, we get that 
$ \left\{\sum_{j=1}^N |\partial_j U_n^\pm|^{p_j}\right\}_{n}$ is uniformly integrable over $ \Omega$. This ends the proof of \eqref{hsn21}.
\end{proof}

\begin{rem} \label{copp}
We need Lemma~\ref{step3} to control the integral in the left-hand side of \eqref{bab}. Indeed, 
we cannot conclude that 
$ A_j(x,U_n,\nabla U_0^\pm)  \,\partial_j U_n^\pm\to A_j(x,U_0,\nabla U_0^\pm) \,\partial_j U_0^\pm$ in $L^1(\Omega)$ as $n\to \infty$ 
for the same reason as in Remark~\ref{rema63}.
\end{rem}

\begin{lemma} \label{occ} In the framework of Theorem~\ref{teo2}, we have \eqref{nica}. 
\end{lemma}

\begin{proof}
Recall that $Z_{n,k}=T_k(u_n)-T_k(u_0)$. From \eqref{wwa}, we have 
$Z_{n,k}\to 0$ a.e. in $\Omega$ as $n\to \infty$ and $Z_{n,k} 
\rightharpoonup  0$ (weakly) in $ W_0^{1,\overrightarrow{p}}(\Omega)$ as $n\to \infty$.  
Moreover, we find that 
\begin{equation} \label{ver} 
\varphi_\lambda(Z_{n,k})\to 0\ \text{ a.e. in }
\Omega\ \text{and } \varphi_\lambda(Z_{n,k}) \rightharpoonup  0\ \text{ (weakly) in } W_0^{1,\overrightarrow{p}}(\Omega) \ 
\text{as }n\to \infty.
\end{equation}
Observe that $u_n\, Z_{n,k}\geq 0$ on the set $\{|u_n|\geq k\}$, which gives that 
$ \Phi(u_n) \,\varphi_\lambda(Z_{n,k})\,\chi_{\{|u_n|\geq k\}}\geq 0$. Thus, by testing \eqref{baaa} with $v=\varphi_\lambda(Z_{n,k})\in W_0^{1,\overrightarrow{p}}(\Omega)\cap L^\infty(\Omega)$,  
we obtain that 
\begin{equation} \label{nou} 
\langle \mathcal A u_n, \varphi_\lambda(Z_{n,k})\rangle +
\int_{\{|u_n|<k\}} \Phi(u_n) \,\varphi_\lambda(Z_{n,k})\,dx\leq 
\ell_{n,k}+I_{u_n}(\varphi_\lambda(Z_{n,k})),
\end{equation} 
where $\ell_{n,k}$ is defined by 
\begin{equation*} \label{supa1}
\ell_{n,k}:=\langle \mathfrak{B} u_n, \varphi_\lambda(Z_{n,k})\rangle
-\int_\Omega \Theta (u_n)\,\varphi_\lambda(Z_{n,k})\,dx+\int_{\Omega} f_n \,\varphi_\lambda(Z_{n,k})\,dx.
\end{equation*}
The first term in $\ell_{n,k}$ converges to $0$ as $n\to \infty$ from  \eqref{wwa}, \eqref{ver} and the property $(P_2)$ of $\mathfrak B$. 
Since $|\Theta(u_n)|\leq C_\Theta$ and \eqref{nice} holds, by the Dominated Convergence Theorem, we get that the second, as well as the third term in $\ell_{n,k}$, converges to $0$ as $n\to \infty$. Hence, 
$\lim_{n\to \infty} \ell_{n,k}=0$. 

To simplify exposition, we now introduce some notation:
\begin{equation*} \label{nota}
\begin{aligned}
& X_k(n):= \phi(k)  \, \int_{\{|u_n|<k\}} 
\left[ \frac{1}{\nu_0}\, \sum_{j=1}^N A_j(u_n)\,\partial_j T_k (u_0) +c(x)\right] |\varphi_\lambda(Z_{n,k})| \,dx,\\
& Y_k(n):= \sum_{j=1}^N \int_\Omega A_j(u_n)\,\, \varphi_\lambda'(Z_{n,k})\,\chi_{\{|u_0|<k\}}\,
\chi_{\{|u_n|\geq k\} }\, \partial_j u_0\,dx.
\end{aligned}
\end{equation*}
We rewrite the first term in the left-hand side of \eqref{nou} as follows
\begin{equation} \label{cop}   \langle \mathcal A u_n, \varphi_\lambda(Z_{n,k})\rangle=
\sum_{j=1}^N \int_{\{|u_n|<k\}} A_j( u_n)  \,\varphi_\lambda'(Z_{n,k})\,\partial_j Z_{n,k}\,dx-Y_k(n).
\end{equation}
The coercivity condition in \eqref{ellip} and the growth condition of $\Phi$ in \eqref{cond1} imply that 
\begin{equation} \label{loc}  | \Phi(u_n) |
\,\chi_{\{|u_n|<k\}}\leq 
\phi(k) \left[ \frac{1}{\nu_0} \sum_{j=1}^N A_j(u_n) \,\partial_j u_n +c(x)\right] 
\,\chi_{\{|u_n|<k\}}. \end{equation} 
In the right-hand side of \eqref{loc} we replace $\partial_j u_n$ by $\partial_j Z_{n,k}+\partial_j T_k (u_0)$, then we multiply the inequality by
$| \varphi_\lambda(Z_{n,k})|$ and integrate over $\Omega$ with respect to $x$. It follows that the second term in the left-hand side of \eqref{nou} is at least 
$$  -\frac{\phi(k)}{\nu_0} \sum_{j=1}^N\int_{\{|u_n|<k\}} A_j(u_n) \,
| \varphi_\lambda(Z_{n,k})| \,\partial_j Z_{n,k} \,dx-X_{k}(n).
$$ 
Using this fact, as well as \eqref{cop}, in \eqref{nou}, we see that $\mathcal F_{n,k}(u_n)$ (defined in \eqref{obb}) satisfies 
\begin{equation} \label{veb}  \mathcal F_{n,k}(u_n)\leq X_k(n)+Y_k(n)+\ell_{n,k}+I_{u_n}(\varphi_\lambda(Z_{n,k})).
\end{equation} 

Since $\lim_{n\to \infty} \ell_{n,k}=0$, by showing that $\lim_{n\to \infty} X_k(n)=\lim_{n\to \infty} Y_k(n)=0$, we conclude \eqref{nica}.  
Using \eqref{ver} and $c\in L^1(\Omega)$, we infer from the Dominated Convergence Theorem that 
\begin{equation} \label{alb1} c(x) |\varphi_\lambda(Z_{n,k})|\,\chi_{\{|u_n|<k\}}\to 0 \ \mbox{in } L^1(\Omega)\ \mbox{as } n\to \infty.\end{equation} 
Next, up to a subsequence of $\{u_n\}$, $A_j(u_n) $ converges weakly 
in $L^{p_j'}(\Omega)$ 
as $n\to \infty$ for all $1\leq j\leq N$. Hence,  
$ \sum_{j=1}^N A_j(u_n)\,\partial_j  u_0 $
converges in $L^1(\Omega)$ as $n\to \infty$. Then, there exists a non-negative function $F\in L^1(\Omega)$ (independent of $n$) such that,  
up to a subsequence of $\{u_n\}$,  
$ \left|\sum_{j=1}^N A_j(u_n)\,\partial_j  u_0\right|\leq F $ a.e. in $ \Omega$ for every $n\geq 1$. 
By the Dominated Convergence Theorem, 
\begin{equation} \label{alb2}  \sum_{j=1}^N A_j(u_n)\, \partial_j T_k (u_0)\, |\varphi_\lambda(Z_{n,k})|\, \chi_{\{|u_n|<k\}}\to 0
\ \mbox{in } L^1(\Omega)\ \mbox{as }n\to \infty.  \end{equation} 
From \eqref{alb1} and \eqref{alb2}, we find that $\lim_{n\to \infty} X_k(n)=0$. 
Remark that $|\varphi_\lambda'(Z_{n,k})|$ is bounded above by a constant independent of $n$ and $\chi_{\{|u_0|<k\}}\,
\chi_{\{|u_n|\geq k\} }\to 0$ a.e. in $\Omega$ as $n\to \infty$. Hence, we can use a similar argument as for $X_k(n)$, to obtain that, up to a subsequence of $\{u_n\}$, 
$\lim_{n\to \infty} Y_k(n)=0$. 
From \eqref{veb}, we conclude \eqref{nica} with $S_k(n)=X_k(n)+Y_k(n)+\ell_{n,k}$. 
\end{proof}

\medskip

{\bf Acknowledgements.} The first author has been supported by the Sydney Mathematical Research Institute International Visitor Program (August--September 2019) and by Programma di Scambi Internazionali dell'Universit\`a degli Studi di Napoli Federico II. The research of the second author is supported by Australian Research Council under the Discovery Project Scheme (DP190102948 and DP220101816).


\end{document}